\documentclass[12pt]{amsart} 

\usepackage{graphicx,amssymb,colordvi,textcomp,latexsym}

\usepackage{tikz}

\usepackage{caption}
\usepackage{subcaption}
\usepackage[title]{appendix}
\usepackage[showframe=false]{geometry}
\usepackage{changepage}

\usepackage{ulem} 

\usepackage{tikz}
\usetikzlibrary{arrows,shapes, positioning, matrix, patterns, decorations.pathmorphing, arrows.meta,automata}

\usepackage{color}
\usepackage{colortbl}
\usepackage{soul}

\usepackage{array}
\usepackage[hidelinks]{hyperref}
\usepackage{amsfonts}
\usepackage{amsmath}
\usepackage{amsthm}
\usepackage{adjustbox}
\usetikzlibrary{arrows,shapes, positioning, matrix, patterns, decorations.pathmorphing}

\newcommand{\R}{\mbox{$\mathbb{R}$}}

\newcommand{\NN}{\mathbf{N}}

\newtheorem{lemma}{Lemma}[section]

\newtheorem{prop}[lemma]{Proposition}
\newtheorem{thm}[lemma]{Theorem}

\theoremstyle{definition}
\newtheorem{Def}[lemma]{Definition}
\newtheorem{exam}[lemma]{Example}

\newtheorem{asps}[lemma]{Assumptions}
\theoremstyle{remark}
\newtheorem{rem}[lemma]{Remark}

\newcommand{\etal}{{\it et al.}}

\newcommand{\ran}{\textnormal{range}}

\title[]{Towards a classification of steady-state bifurcations for networks with asymmetric inputs}
\author{Manuela Aguiar}
\address{Manuela Aguiar, CMUP, Faculdade de Ci\^encias, Universidade do Porto, Rua do Campo Alegre s/n, 4169-007 Porto, Portugal\newline 
Faculdade de Economia, Universidade do Porto, Rua Dr Roberto Frias, 4200-464 Porto, Portugal}
\email{maguiar@fep.up.pt}
\author{Ana Dias}
\address{Ana Dias, Centro de Matem\'atica da Universidade do Porto, Departamento de Matem\'atica, Faculdade de Ci\^encias, 
Universidade do Porto, Rua do Campo Alegre s/n, 4169-007 Porto, Portugal}
\email{apdias@fc.up.pt}
\author{Pedro Soares}
\address{Pedro Soares, Departamento de Matem\'atica, ISEG-Lisbon School of Economics and Management, Universidade de Lisboa,
 Rua do Quelhas 6, 1200-781 Lisbon, Portugal}
\email[Corresponding author]{psoares@iseg.ulisboa.pt}

\keywords{Coupled cell network, asymmetric inputs, steady-state bifurcation, synchrony space, synchrony subspaces lattice}
\subjclass[2010]{Primary: 34C23; Secondary:  37C10, 34C15, 05C90}

\date{\today}

\begin{document}

\begin{abstract}
We consider homogeneous coupled cell networks with asymmetric inputs. 
We obtain general results concerning codimension-one steady-state bifurcations for networks with any number of cells and any number of asymmetric inputs. 
These results rely solely on the network adjacency matrices eigenvalue structure and the existence, or not, of network synchrony subspaces. 
For networks with three-cells, we describe the possible lattices of synchrony subspaces annotated with the eigenvalues on each synchrony subspace.
Applying the previous results, we classify the synchrony-breaking steady-state bifurcations that can occur for three-cell minimal networks with one, two or six asymmetric inputs.
\end{abstract}

\maketitle

\section{Introduction}

In this paper we consider coupled cell systems, that is, continuous dynamical systems whose structure can be schematized through a network.
We follow the formalisms of Stewart, Golubitsky and co-workers~\cite{SGP03, GST05} and Field~\cite{F04}, where a coupled cell network is a directed graph with nodes and edges representing, respectively,   the cells (sets of dynamical systems, here, systems of ordinary differential equations) and  the couplings between the cells.  
One of the key properties of coupled cell systems is the existence of synchrony spaces --  spaces defined in terms of equalities of cell coordinates and that are left invariant under any coupled cell system consistent with a given network.
Remarkably, the existence of such spaces depends solely on the network structure and not on the given admissible vector field.

We concentrate our work on networks with $k$ asymmetric inputs -- there are  $k \in \NN$ different input types and each cell receives exactly one input of each type. 
The corresponding coupled cell systems have all the cells  with the same internal phase space (identical cells) and the cells are all input equivalent as every cell receives  exactly $k$ couplings of the $k$ different types. 
These networks are formally defined by $k$ adjacency matrices, one for each type of coupling. 
Any coupled cell system associated to a network with asymmetric inputs admits the full-synchrony subspace where all cells are synchronized.  
It is also known that the set of synchrony spaces of a given network is a lattice where the bottom is the full-synchrony subspace and the top is the network phase space, see Stewart~\cite{S07}. Moreover,  coupled cell systems restricted to  any  network synchrony subspace correspond to  coupled cell systems consistent with  a smaller network determined by the original network and the synchrony space, ~\cite{SGP03, GST05}.

Consider a network with $k$ asymmetric inputs.  Take a one-parameter family of coupled cell systems associated with that network and possessing an equilibrium in the full-synchrony subspace. 
One important observation we make is that the Jacobian matrix at that equilibrium is a linear combination of the adjacency matrices of the network, and its eigenvalues are functions of the first order derivatives of the associated coupled cell systems.
Assume that one of these  eigenvalues crosses $0$ as the parameter changes. 
In any neighborhood of the full synchrony equilibrium, new equilibria may appear forming a steady-state bifurcation branch. 
Such bifurcation branch can have less synchrony and, in this case, we say that a synchrony-breaking steady-state bifurcation has occurred.
The synchrony of a bifurcation branch is the smallest synchrony subspace that contains it. 
In this work, we describe the synchrony subspaces that might robustly support a steady-state bifurcation branch for networks with three cells and two asymmetric inputs.
Together with the characterization given in Leite and Golubitsky~\cite{LG06} for the $34$ networks with three cells and two symmetric inputs, it gives a complete characterization of the patterns of synchrony-breaking steady-state bifurcations for networks with three cells and two inputs.

To obtain a complete synchrony-breaking steady-state bifurcation characterization, even for networks with three cells and two inputs, it is not feasible to study every possible network as the number of possible networks is large.
Moreover, there are networks that have different topologies but the same type of generic dynamics. 
When two different networks support the same space of coupled cell systems, they are said to be ODE-equivalent.
In previous work, Aguiar, Dias and Soares~\cite{ADS20} describe the ODE-classes of networks with three cells and two asymmetric inputs.
This reduces the number of networks to be considered to 48.

As coupled cell systems restricted to  any  network synchrony subspace correspond to  coupled cell systems consistent with  a smaller network, the  lattice of 
the network synchrony spaces can have its elements (synchrony subspaces) annotated by the eigenvalues of the smaller networks determined by  those synchrony 
subspaces.
For networks with three cells, there are seven possible annotated lattices of synchrony spaces.
We show that the seven cases can be studied using three bifurcation results about the synchrony spaces that support a bifurcating branch.  
In this work, we first prove these bifurcation results for networks independently of the number of cells.
The non-degeneracy conditions described in the bifurcation results generically depend solely on the network.
Finally, we apply the results to the 48 ODE-classes of networks with three cells and two asymmetric inputs, according to its annotated lattice and checking the non-degeneracy conditions.
This provides a complete characterization of the 
synchrony spaces that support a 
(codimension-one) steady-state
bifurcating branch for networks with three cells and two asymmetric inputs.

\subsection*{Further motivation}
The analysis of small networks, Milo~\etal~\cite{MSIKCA02}, that are part of complex networks modelling real-world problems can help understand the dynamical properties of those big networks. 
One important perspective in science is precisely to find the small building blocks, usually called {\it motifs}, that are often occurring in biological 
networks, study their dynamics and then understand how cellular function emerges from the interactions between the motifs. 
See Morone, Leifer and Makse~\cite{MLM20} and Leifer~\etal~\cite{LSI21}.
Some examples of applied studies about small networks are given next. 

Jia~\etal~\cite{Jia17} consider two common two node motifs that are often found in many cell-fate decisions during embryonic development.
One motif, the toggle switch (TS), it is comprised of two transcription factors  that mutually inhibit each other; a TS where each transcription factor self-activate is a self-activating toggle switch. 
Their work is focus into answering two questions. 
What are the types of co-existing stable equilibria (phenotypes) and how the coupling between two such TS influences the state-space of each other.
Mangan and Alon~\cite{MA03} consider feed-forward loops (FFLs) structures with three genes (nodes) corresponding to biochemical wiring patterns (network motifs) which recur throughout in transcription networks. 
There are eight possible FFLs to be considered, as each of the three interactions in the FFL can be activating or repressing. 
The authors showed that four of them act as sign-sensitive accelerators, while the other four types act as sign-sensitive delays.
They also remarked that some of these FFLs appear more often probably due to the reduced functionality of the rare ones. 
In \cite{Jolly15} a theoretical study is made considering the coupling between miR-200/ZEB (the decision-making feedback loop for Epithelial-Mesenchymal Transition) and LIN28/let-7 (the decision-making feedback loop for tumour-initiation potential) showing the ability to disseminated primary tumour cells to form metastases at other organs.
In \cite{HLJK23}, Hernandez~\etal\, break the complex network into smaller independent subnetworks to describe the steady-states of complex networks arising from biochemical systems, which often describe their long-term behaviours. 
The steady-state solutions of each subnetworks are then stitch together to lead to the analytic steady-states of the original network. 
Prill~\etal~\cite{PIL05} carry out computational analysis about the robust stability in biological networks to small-scale perturbations in biological entities, motifs with three or four nodes.
By analysing the responses to small perturbations from a steady-state, under different assumptions on the parameters, their results suggest that robust stability of networks motifs is an important determinant of biological network structure.

The dynamics associated with networks of just two or three-cells can already be complex.
See, for example, Pasemann~\cite{P02} for the discrete-time case and Aguiar~\etal~\cite{AADF11} for the continuous-time case.

In this work, we study which synchrony-breaking patterns emerge as a bifurcation occurs in three-cell networks.
Examples of previous works about this topic are Beer~\cite{B95}, Leite and Golubistky~\cite{LG06} and  Golubitsky and Wang~\cite{GW20}.

\subsection*{Details on the main results of the paper}

In Section~\ref{sec:ssbresults}, we give general results about the synchrony-breaking steady-state bifurcations for networks with asymmetric inputs.
More specifically, we show which synchrony spaces support a  bifurcating branch of equilibria  arising from a codimension-one steady-state bifurcation at a full synchronous equilibrium, for generic coupled cell systems.
 The results are organized according to the number of synchrony subspaces intersecting, in a non trivial way, the generalized kernel $K$ of the Jacobian matrix at a full synchronous equilibrium. 

The first result corresponds to the case where  $K$  has dimension one and where it is taken the smallest synchrony subspace containing $K$. 
It is usually called the synchrony bifurcation branch,  and has been proved for different types of networks, such as networks with symmetric inputs, see Soares~\cite{S17} and Golubitsky and Lauterbach~\cite{GL09}.
Using the Lyapunov-Schmidt Reduction Method, see for example Golubitsky and Schaeffer~\cite{GS85}, we show in Theorem~\ref{thm:LSred1} that each such synchrony subspace supports a steady-state bifurcation branch. 

In the second result, we consider the case where $K$ has dimension $m$ and the smallest synchrony subspace $\Delta$ which contains $K$ also includes $2^{m}-1$ synchrony subspaces intersecting $K$  in a one-dimensional space. 
By Theorem~\ref{thm:LSred1} mentioned above, we know that those $2^{m}-1$ synchrony subspaces intersecting $K$ support a bifurcation branch. 
Using B\'ezout's Theorem, see for example Blum~\etal~\cite{BCSS98}, we prove in Theorem~\ref{thm:LSred2} that the synchrony subspace $\Delta$, containing $K$, does not support a bifurcation branch.

The third result considers the case of $K$ being two-dimensional and the kernel of the Jacobian matrix at a full synchronous equilibrium being one-dimensional. 
We prove in Theorem~\ref{thm:LSred3} that the smallest synchrony subspace containing $K$ supports a bifurcation branch.  

In Section~\ref{sec_three_cells}, we study the codimension-one steady-state bifurcations from a full synchronous equilibrium for continuous-time dynamical systems associated with three-cell networks with any number $k$ of asymmetric inputs. 
The networks are grouped 
according to their annotated lattice of synchrony subspaces. 
We show  in  Theorem~\ref{thm:lala} that there are seven possible synchrony lattice structures for connected three-cell networks with any number of asymmetric inputs, see Figure~\ref{Parte1}. 
Under the assumption of some network non-degeneracy conditions, the 
results obtained in Section~\ref{sec:ssbresults} show which synchrony spaces support a steady-state bifurcating branch. 
Noticeably, we derive that there are eight possible bifurcation diagrams, see  Figure~\ref{bd}, where one lattice leads to two distinct bifurcation diagrams and the other six lead to a distinct bifurcation diagram each.

Finally, in Section~\ref{subsec_case_study}, we apply the methodology developed here, and described above, to the minimal networks with three-cells and one, two and six asymmetric inputs enumerated in ~\cite{ADS20} which cover every possible dynamics with those numbers of cells and asymmetric inputs. 
First,  the network eigenvalues are obtained in Theorem~\ref{thm:classification2} and the annotated lattices of each network in Theorem~\ref{thm:classificationlat}. 
An observation is that only six of the seven possible lattices obtained in Theorem~\ref{thm:lala} occur for the 48 networks under study.
Using the results of Section~\ref{sec:ssbresults}, and checking the network non-degeneracy conditions, we prove in Theorem~\ref{thm:finalbif} which synchrony subspaces support a steady-state bifurcation branch from a full synchronous equilibrium,  obtaining then the corresponding bifurcation diagrams, see Table~\ref{tab:latdia}.
It should be noted that, this characterization is still valid for synchrony subspaces with dimension at most three of a network with one, two or six asymmetric inputs for any number of cells.

\subsection*{Organization of the paper}

The paper is organized in the following way. 
In Section~\ref{sec_back} we recall the main definitions and results concerning coupled cell networks, coupled cell systems, synchrony spaces and codimension-one steady-state bifurcations of coupled cell networks. 
In Section~\ref{sec:ssbresults} we present general results concerning the codimension-one steady-state bifurcations for coupled cell networks with asymmetric inputs. The main results are Theorems~\ref{thm:LSred1}, \ref{thm:LSred2} and \ref{thm:LSred3} proving which  synchrony subspaces robustly support a steady-state bifurcation branch.
In Section~\ref{sec_three_cells}, we consider networks with three-cells and any number of asymmetric inputs, obtaining the network eigenvalues, the lattices of synchrony subspaces and applying the bifurcation results of  Section~\ref{sec:ssbresults}. 
Finally, in Section~\ref{subsec_case_study}, we apply the previous methodology to the minimal networks with three-cells and one, two and six asymmetric inputs. The main result is Theorem~\ref{thm:finalbif} listing the synchrony subspaces supporting a steady-state bifurcation branch.

\section{Steady-state bifurcations for coupled cell networks}\label{sec_back}

We follow the formalisms of Stewart, Golubitsky and Pivato~\cite{SGP03}, Golubitsky, Stewart and T\"{o}r\"{o}k~\cite{GST05} and Field~\cite{F04} on coupled cell networks and the associated coupled cell systems. 

\subsection{Coupled cell networks}

In this paper, we consider $n$-cell {\it coupled cell networks with $k$ asymmetric inputs} which can be represented by directed graphs, where the cells are placed at vertices (nodes) and the couplings are depicted by directed arrows. 
Any cell receives $k$ inputs, one from each type. 
The description of  any such network can be given by $k$ {\it adjacency matrices}, of order $n$, if $n$ is the network number of cells and the rows and columns are indexed by the network cells. 
If $A_1, \ldots, A_k$ are the adjacency matrices, the entry $ij$ of the matrix $A_l$ is $1$ if there is a directed edge from cell $j$ to cell $i$ of type $l$, or $0$ otherwise, where $l=1, \ldots, k$ and $i,j=1,2,\dots,n$. 
Thus, each row of $A_l$ has exactly one entry equal to $1$ and $0$ elsewhere.

We recall that  a network is {\it connected} if there is an undirected path between any two cells.
All networks considered here are connected.

\begin{exam} \label{exam:asymmetric_networks}
In Figure~\ref{fig_1and2} we have two three-cell and one  two-cell coupled networks with asymmetric inputs. 
As an example, the network in the middle has two asymmetric inputs which can be described by the following  two adjacency matrices 
$$
A_1 = 
\left(
\begin{array}{ccc}
0 & 1 & 0 \\
1 & 0 & 0 \\
1 & 0 & 0
\end{array}
\right), \quad 
A_2 = 
\left(
\begin{array}{ccc}
0 & 1 & 0 \\
0 & 0 & 1 \\
0 & 1 & 0
\end{array}
\right)\, .
$$
\hfill $\Diamond$
\end{exam}
 
 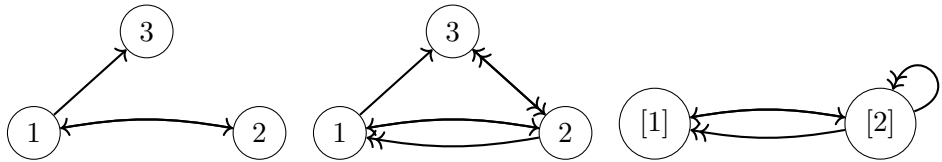
\begin{figure}[h]
 \begin{tabular}{lll}
\begin{tikzpicture}
 [scale=.15,auto=left, node distance=1.5cm, every node/.style={circle,draw}]
 \node[fill=white] (n1) at (4,0) {\small{1}};
  \node[fill=white] (n2) at (24,0) {\small{2}};
 \node[fill=white] (n3) at (14,9)  {\small{3}};
 \draw[->, thick] (n1) edge  [bend right=-10] (n2); 
 \draw[->, thick] (n2) edge  [bend left=-10] (n1); 
\draw[->, thick] (n1) edge (n3); 
\end{tikzpicture}  &
\begin{tikzpicture}
  [scale=.15,auto=left, node distance=1.5cm, every node/.style={circle,draw}]
 \node[fill=white] (n1) at (4,0) {\small{1}};
 \node[fill=white] (n2) at (24,0) {\small{2}}; \node[fill=white] (n3) at (14,9)  {\small{3}};
\draw[->, thick] (n1) edge  [bend right=-10] (n2); 
\draw[->, thick] (n2) edge  [bend left=-10] (n1); 
\draw[->, thick] (n1) edge (n3); 
\draw[->>, thick] (n2) edge  [bend left=10] (n1); 
\draw[->>, thick] (n3) edge   (n2); 
\draw[->>, thick] (n2) edge  (n3); 
\end{tikzpicture} &
\begin{tikzpicture}
  [scale=.15,auto=left, node distance=1.5cm, every node/.style={circle,draw}]
 \node[fill=white] (n1) at (4,0) {\small{$[1]$}};
 \node[fill=white] (n2) at (24,0) {\small{$[2]$}}; 
\draw[->, thick] (n1) edge  [bend right=-10] (n2); 
\draw[->, thick] (n2) edge  [bend left=-10] (n1); 
\draw[->>, thick] (n2) edge  [bend left=10] (n1); 
\draw[->>, thick] (n2) to  [in=70,out=20,looseness=5] (n2);
\end{tikzpicture}
\end{tabular} 
\caption{Examples of coupled cell networks with asymmetric inputs. }
\label{fig_1and2}
 \end{figure}

\subsection{Coupled cell systems}

Let $\mathcal{N}$ be an $n$-cell network with $k$ asymmetric inputs, say of types $1, \ldots, k$. We take a cell to be a system of ordinary differential equations and we consider the class of coupled cell systems consistent with the structure of the network $\mathcal{N}$. 
As every cell of the network receives $k$ inputs, one from each type, we say that  the network is {\it homogeneous}. 
It follows that we consider that 
all the cells are {\it identical}, that is, they have the same phase space $V$ 
assumed here to be
$V = \R$.  
Moreover, the dynamics of each cell $i$ is governed by the same smooth function $f:\, V^{k+1} \to V$, evaluated at cell $i$ and at the $k$  cells belonging to its input set. Thus, for $i=1, \ldots, n$, the evolution of cell $i$ is given by the set of ordinary differential equations 
\begin{equation} \label{eq:CCS1}
\dot{x}_i = f\left(x_i; \, x_{i_1}; \ldots;  x_{i_k}\right), 
\end{equation}
where $i_j$ is the start cell of the edge with type $j$ and heading to cell $i$.  The coupled cell systems with cells governed by equations of the form (\ref{eq:CCS1}) are $\mathcal{N}$-{\it admissible} and  $f^{\mathcal{N}}$ denotes the coupled cell system defined by $f$.

\begin{exam} Consider the network in the middle  of Figure~\ref{fig_1and2} and any smooth function $f:\, \R^3 \to  \R$. Coupled cell systems associated with this network are of the form 
\begin{equation}\label{eq_ex1}
\left\{ 
\begin{array}{l}
\dot{x}_1 = f(x_1; x_2;x_2)\\
\dot{x}_2 = f(x_2; x_1;x_3)\\
\dot{x}_3 = f(x_3; x_1;x_2)
\end{array} \, .
\right. 
\end{equation} 
\hfill  $\Diamond$
\end{exam}

Given two networks, the set of coupled cell systems associated to them can coincide, up to renumbering of the cells. 
In this case, we say that those networks are {\it ODE-equivalent}. 
This defines an equivalence relation between networks whose equivalence classes are called {\it ODE-classes}.
Inside each ODE-class, the networks having a minimal number of inputs are called {\it minimal networks}.
See Dias and Stewart~\cite{DS05} and Aguiar and Dias~\cite{AD07} for details. 

\subsection{Synchrony subspaces}

One of the most remarkable and first observed property of coupled cell systems, which only depends on the network structure and not on the particular coupled cell system, is the existence of flow-invariant subspaces.
Those subspaces are called {\it synchrony subspaces} and are defined by some equalities of  cell's coordinates, $x_i=x_j$. More precisely, one such {\it polydiagonal subspace} $\Delta$ is said to be a network synchrony subspace if it is invariant for any admissible vector field of that network.  
By \cite[Theorem 6.5]{SGP03}, a polydiagonal space is a network synchrony subspace if and only if it is left invariant by every network's adjacency matrix,

\begin{exam} Consider the network ${\mathcal{G}}$ in the middle of Figure~\ref{fig_1and2} and  the associated coupled cell systems in  (\ref{eq_ex1}). The polydiagonal  space $\Delta = \{ (x_1,x_2,x_3):\, x_2 = x_3\}$ is flow-invariant for any such coupled cell system. In fact, looking at the form of the coupled systems in (\ref{eq_ex1}), we see that, given an initial condition where the values for $x_2$ and $x_3$ coincide, then the equations for the two variables are the same and thus they stay synchronized for forward time. Moreover, this property does not depend on the particular choice of the function $f$ nor the choice of the cell phase spaces. The space $\Delta$ is a synchrony space for the network ${\mathcal{G}}$. 
Note that $\Delta$ is left invariant under the network adjacency matrices, $A_1$ and $A_2$ in Example~\ref{exam:asymmetric_networks}. 
\hfill $\Diamond$
\end{exam}

The restriction of a coupled cell system associated with a network $\mathcal{N}$  to a synchrony subspace $\Delta$ corresponds to a coupled cell system consistent with a smaller network ${\mathcal{Q}}$, the {\it quotient network} of $\mathcal{N}$ determined by $\Delta$, see~\cite[Theorem 5.2]{GST05}. 
We can define an equivalence relation $\bowtie$ on the network set of cells in the following way: if $x_i = x_j$ is one of the equalities defining $\Delta$ then $i \bowtie j$. We write $\Delta = \Delta_{\bowtie}$. 
The cells of ${\mathcal{Q}}$ are so the $\bowtie$-equivalence classes and, for each edge-type $l$, an edge $[j] \to^l [i]$ corresponds to the edges $j \to^l i$ in $\mathcal{N}$ where $j\in[j]$ and $i\in[i]$. 
Thus, edge types are preserved and both networks have the same number of asymmetric inputs.
That is, if $\mathcal{N}$ has $k$ asymmetric inputs, the quotient network ${\mathcal{Q}}$ has also $k$ asymmetric inputs.

  \begin{exam} Consider the network ${\mathcal{G}}$ in the middle of  Figure~\ref{fig_1and2} and recall that the space $\Delta = \{ (x_1,x_2,x_3):\, x_2 = x_3\}$ is a synchrony space for that network. Note that $\Delta = \Delta_{\bowtie}$, where $\bowtie = \left\{ [1] =\{1\},\, [2] = \{ 2,3\}\right\}$. The quotient network ${\mathcal{Q}}$ of ${\mathcal{G}}$ determined by $\Delta$ is the network on the right of Figure~\ref{fig_1and2}. 
\hfill $\Diamond$
\end{exam}

Stewart~\cite{S07} showed that the set of synchrony subspaces of a network together with the partial order of inclusion is a finite lattice.
Thus the {\it join} and {\it meet} of any two synchrony subspaces $\Delta_1$ and $\Delta_2$ are well-defined in the set of synchrony subspaces.
The join of $\Delta_1$ and $\Delta_2$ is the smallest synchrony subspace containing $\Delta_1$ and $\Delta_2$.
Analogously, the meet of $\Delta_1$ and $\Delta_2$ is the largest synchrony subspace contained in $\Delta_1$ and $\Delta_2$ (the intersection of the two synchrony subspaces).

For the networks with asymmetric inputs considered here, 
the {\it bottom} of the lattice, the smallest synchrony subspace or the meet of all synchrony subspaces, is the {\it full-synchronized subspace} $\Delta_0$.
The {\it top} of the lattice, the biggest synchrony subspace, or the join of all synchrony subspaces, is the {\it network phase space} $\R^n$, if $n$ is the number of cells in the network.

\begin{exam} The lattice of the synchrony spaces of the network ${\mathcal{G}}$ in the middle of  Figure~\ref{fig_1and2} is formed by the full-synchronized subspace $\Delta_0 = \{ (x_1,x_2,x_3):\, x_1 = x_2 = x_3\}$, $\Delta = \{ (x_1,x_2,x_3):\, x_2 = x_3\}$ and the network phase space $\R^3$. 
\hfill $\Diamond$
\end{exam}

\subsection{Steady-state bifurcations for networks}\label{sec:ssbngen}

Let $\mathcal{N}$ be a network with $k$ asymmetric inputs. 
We address in this paper the possible codimension-one local steady-state bifurcations  of coupled cell systems from a full synchrony equilibrium solution $x_0$ and their spontaneous synchrony breaking. 
Let $f:\, \mathbb{R}\times\mathbb{R}^{k}\times\mathbb{R}\rightarrow \mathbb{R}$ be a family of smooth functions defining a family of coupled cell systems $f^{\mathcal{N}}$ on $\mathcal{N}$ parametrized by a real parameter $\lambda$: 
\begin{equation} \label{eq1}
\dot{x}=f^{\mathcal{N}}(x,\lambda)\, .
\end{equation}
We assume, without loss of generality, that the solution $x_0$ and the bifurcation point $\lambda_0$ are the origin. 
In the following, we will assume that the origin is an equilibrium point for every $\lambda$, 
$$f(0,\lambda) \equiv 0\, .$$

We consider that the family of coupled cell systems $f^{\mathcal{N}}$ has a {\it local steady-state bifurcation} at $(x_0,\lambda_0) = (0,0)$ if the number
 of steady-state solutions of (\ref{eq1}) in any neighbourhood of $x_0$ changes when the parameter $\lambda$ crosses $\lambda_0 =0$. Recall that a necessary condition for the existence of local steady-state bifurcations at $(0,0)$ is that the Jacobian matrix of $f^{\mathcal{N}}(x,0)$ at the origin has a zero eigenvalue. Throughout, we denote by $J^{\mathcal{N}}_f$ the Jacobian matrix of $f^{\mathcal{N}}(x,0)$ at the origin. 

Using the $k$ network adjacency matrices, say $A_1,\dots,A_k$, we have that $J^{\mathcal{N}}_f$ has the following form: 
\begin{equation} \label{eq_Jac}
J_f^{\mathcal{N}}=D_xf^{\mathcal{N}}(0,0)= f_0 \mbox{Id}_n + \sum_{j=1}^k f_j A_j,
\end{equation}
where $f_j$, for $j=0, \ldots, k$, denotes the derivative of $f$ at the origin with respect to the variable $j+1$.  
Since the Jacobian matrix depends linearly on the first derivatives of $f$, there are continuous functions 
$$\mu_1,\dots,\mu_s:\mathbb{R}^{k+1}\rightarrow \mathbb{C}$$
such that $\mu_1(f_0,f_1,\dots,f_k),\dots, \mu_s(f_0,f_1,\dots,f_k)$ are the distinct eigenvalues of $J_f^{\mathcal{N}}$. 
We call these functions the {\it network eigenvalues}. Moreover, as every matrix $A_j$ has row sum one, we have that one of the network eigenvalues is  
$$
\upsilon = \sum_{j=0}^k f_j, 
$$
the {\it network valency eigenvalue}.

A {\it steady-state bifurcation condition} (at $\lambda=0$) for the family $f^{\mathcal{N}}$ is given by the equality of one of those network eigenvalues of $J_f^{\mathcal{N}}$ to zero,
$$\mu_i(f_0,\dots,f_k)=0.$$ 

\begin{exam}
Let $\mathcal{N}$ be the network with three-cells and one asymmetric input forming a $3$-cycle. 
Then the network eigenvalues of $\mathcal{N}$ are
$$\mu_1(f_0,f_1)=f_0+f_1, \quad \mu_2(f_0,f_1)=f_0+f_1e^{\imath \pi/3}, \quad \mu_3(f_0,f_1)=f_0+f_1e^{\imath 2\pi/3}.$$
Note that $\mu_2=0$ implies that $f_0=f_1=0$ and thus $\mu_1=\mu_3=0$.
\hfill $\Diamond$ 
\end{exam}

Since we are interested in steady-state bifurcations, we consider network eigenvalues which are real on some open set of $\mathbb{R}^{k+1}$, i.e., the interior of $\mu_i^{-1}(\mathbb{R})$ is nonempty.  
 
Next, we define the space of functions with a bifurcation condition given by a network eigenvalue.

\begin{Def}
Let $\mathcal{N}$ be a network with $k$ asymmetric inputs and $\mu$ a network eigenvalue. 
The space of one-parameter families of coupled cell systems with a bifurcation condition given by $\mu$ is denoted by $\mathcal{V}_{\mu}({\mathcal{N}})$, where 
\begin{equation}\label{eq2}
\hspace{\textwidth minus \textwidth}\mathcal{V}_{\mu}({\mathcal{N}}):=\{f:\mu(f_0,\dots,f_k)=0, (f_0,\dots,f_k)\in     \textrm{int}(\mu^{-1}(\mathbb{R})), f(0,\lambda) \equiv 0\}\, .
\hspace{\textwidth minus \textwidth} \Diamond
\end{equation}
\end{Def}

\begin{rem}
Given a network $\mathcal{N}$ with a synchrony space $\Delta$ and ${\mathcal{Q}}$ the quotient network  of $\mathcal{N}$ determined by $\Delta$,
 as the restriction of a coupled cell system associated with $\mathcal{N}$ to $\Delta$ corresponds to a coupled cell system consistent with ${\mathcal{Q}}$, we have that any eigenvalue of ${\mathcal{Q}}$ is also an eigenvalue of $\mathcal{N}$.  
Moreover, the spaces of functions with a bifurcation condition given by an eigenvalue $\mu$ of ${\mathcal{Q}}$ is equal for the networks $\mathcal{N}$ and ${\mathcal{Q}}$:
\[\hspace{\textwidth minus \textwidth}\mathcal{V}_{\mu}({\mathcal{N}})=\mathcal{V}_{\mu}({\mathcal{Q}})\, .\hspace{\textwidth minus \textwidth} \Diamond\]
\end{rem}

\begin{Def}
An eigenvalue of  a network $\mathcal{N}$ which is also an eigenvalue of the quotient network ${\mathcal{Q}}$  of $\mathcal{N}$ determined by a synchrony space $\Delta$  is called 
{\it an eigenvalue of $\Delta$}. 
\hfill $\Diamond$
\end{Def}

Following usual conventions in bifurcation theory, we consider generic functions satisfying a bifurcation condition.
A function is {\it generic} if it satisfies a finite number of non-degenerated conditions.
A non-degenerated condition is a non-trivial inequality on the derivatives of the function $f$ at the origin which  does not contradict the bifurcation condition. 
Trivially, from (\ref{eq_Jac}), we see that a bifurcation condition $\mu(f_0,\dots,f_k)=0$ determines the value of $f_0$ based on the other first derivatives of $f$.
Thus, we consider non-degenerated conditions given by inequalities in the derivatives of the function $f$ not including  the derivative $f_0$.

\begin{Def}
Given a bifurcation problem determined by $f \in  \mathcal{V}_{\mu}({\mathcal{N}})$.\\
(i) A function $b:R\rightarrow \mathbb{R}^n$ is a {\it bifurcation branch} of $f^{\mathcal{N}}$  if $b\not\equiv0$ and
$$f^{\mathcal{N}}(b(\lambda),\lambda)=0,$$
where $\lambda\in R$ and $R$ is a connected region of $\mathbb{R}$ containing the origin.\\
(ii) A bifurcation branch  $b$ of $f^{\mathcal{N}}$ has {\it synchrony} $\Delta_b$ if $\Delta_b$ is the smallest synchrony subspace which contains $b(R)$. 
\hfill $\Diamond$
\end{Def}

\begin{Def}\label{def:Delta_supports}
A synchrony subspace $\Delta$ {\it supports a  steady-state synchrony-breaking bifurcation}  if there is a generic bifurcation problem $\mathcal{V}_{\mu}({\mathcal{N}})$ where some bifurcation branch has synchrony $\Delta$, for a generic family of functions $f \in \mathcal{V}_{\mu}({\mathcal{N}})$.
\hfill $\Diamond$
\end{Def}

We use the following terminology:

\begin{Def}\label{types_eigen}
Let $\mathcal{N}$ be a network with $k$ asymmetric inputs and $\mu:\mathbb{R}^{k+1}\rightarrow \mathbb{C}$ an eigenvalue of $\mathcal{N}$. We say that $\mu$ is: \\
(i) {\it semisimple with multiplicity} $m$, if $\mu(f_0,\dots,f_k)$ is a semisimple eigenvalue of $J^{\mathcal{N}}_f$ with multiplicity $m$, for $f$ generic. \\
(ii) {\it simple} if  it is semisimple with multiplicity $1$.\\
(iii) {\it defective with multiplicity} $(m,n)$ if $m<n$ and $\mu(f_0,\dots,f_k)$ is an eigenvalue of $J^{\mathcal{N}}_f$ with geometric multiplicity $m$ and algebraic multiplicity $n$, for $f$ generic.
\hfill $\Diamond$ 
\end{Def}
The next definition relates eigenspaces and synchrony subspaces. 

\begin{Def}\label{types_syn}
Let $\mathcal{N}$ be a network with $k$ asymmetric inputs, $\Delta$ a synchrony subspace of $\mathcal{N}$ and 
$\mu:\mathbb{R}^{k+1}\rightarrow \mathbb{C}$ a network eigenvalue of $\Delta$. We say that $\Delta$ is: \\
(i) $\mu$-{\it maximal} when $\mu$ is not an eigenvalue of $\overline{\Delta}$ for any synchrony subspace $\overline{\Delta}\subsetneq \Delta$.\\
(ii) $\mu$-{\it submaximal of order} $j\geq 1$, if there are $j$ simple and maximal synchrony subspaces $\Delta_1,\dots,\Delta_j \subsetneq \Delta$ and, for any generic $f\in\mathcal{V}_{\mu}({\mathcal{N}})$ and $\overline{\Delta}\subsetneq \Delta$, 
\[\hspace{4cm}\overline{\Delta}\cap \bigcup_k \ker({J^{\mathcal{N}}_f}^k)\subsetneq \Delta\cap \bigcup_k \ker({J^{\mathcal{N}}_f}^k). \hspace{4cm} \Diamond\]
\end{Def}

In the first case of the above definition, the synchrony subspace is one of the lowest in the lattice of synchrony subspaces to have that network eigenvalue.
In the second case, the synchrony subspace is not one of the lowest in the lattice, but the eigenspace increases as we move up in the lattice at such synchrony subspace.    

\begin{exam}
Consider the networks ${\mathcal{N}}_1, {\mathcal{N}}_2, {\mathcal{N}}_3$ in the left, middle and right, respectively, in Figure~\ref{fig_1and2} and take the adjacency matrices 
$$
A_1 = \left( 
\begin{array}{lll}
0 & 1 & 0\\
1 & 0 & 0\\
1 & 0 & 0
\end{array}
\right), \quad 
A_2 = 
\left( 
\begin{array}{lll}
0 & 1 & 0\\
0 & 0 & 1\\
0 & 1 & 0
\end{array}
\right), \quad
A_3 = 
\left( 
\begin{array}{ll}
0 & 1 \\
1 & 0 
\end{array}
\right), \quad 
A_4 = 
\left( 
\begin{array}{ll}
0 & 1 \\
0 & 1 
\end{array}
\right)
\, .
$$
Note that $A_1$ is the adjacency matrix of ${\mathcal{N}}_1$ and $A_1,A_2$ ($A_3$, $A_4$) are the adjacency matrices of the network with two  
asymmetric inputs ${\mathcal{N}}_2$ (${\mathcal{N}}_3$, respectively). 
The eigenvalues of ${\mathcal{N}}_1$ (${\mathcal{N}}_2$) are $\upsilon = f_0 + f_1,\, f_0,\, f_0-f_1$ ($\upsilon = f_0 + f_1 + f_2,\, f_0 - f_1,\, f_0 - f_2$, respectively). 
The network ${\mathcal{N}}_3$  is a quotient network of ${\mathcal{N}}_2$. 
Its eigenvalues are
 $\upsilon = f_0 + f_1 + f_2,\, f_0 - f_1$. All the eigenvalues are simple. 
Note that $\mu (f_0,f_1,f_2) = f_0-f_1$ is an eigenvalue of ${\mathcal{N}}_2$ and ${\mathcal{N}}_3$. 
The synchrony space $\{x:\, x_2 = x_3\}$ of ${\mathcal{N}}_2$ is $\mu$-maximal since $\mu$ is not an eigenvalue in the full-synchrony subspace.
\hfill $\Diamond$ 
\end{exam}

\section{Steady-state bifurcations for networks with asymmetric inputs: general results}\label{sec:ssbresults}

Let $\mathcal{N}$ be a network with $k$ asymmetric inputs and consider an one-parameter family of coupled cell systems for $\mathcal{N}$ as in (\ref{eq1}). 
Let $\mu$ be an eigenvalue of $\mathcal{N}$ and $f \in \mathcal{V}_{\mu}({\mathcal{N}})$. 
We address three general cases concerning  the relation between  the network synchrony spaces and the eigenspace associated with the eingenvalue $\mu$, namely,  when the synchrony subspace is  
(i) $\mu$-simple and maximal, (ii) $\mu$-semisimple with multiplicity $m$ and it is submaximal of order $2^m-1$, and (iii) $\mu$-defective with multiplicity (2,1) and submaximal of order $1$. We point out that these three cases cover the study of minimal networks with three-cells and two asymmetric inputs presented in the next two sections.

\subsection{Simple eigenvalue and maximal synchrony space}

We start by addressing the simplest case where $\mu$ is simple and the synchrony space $\Delta$ is maximal, which is a common case, specially when the synchrony subspace $\Delta$ has a low dimension. Similar cases have been studied in \cite{SG11, S17} for regular networks and our approach is similar. The idea is to  apply Lyapunov-Schmidt Reduction \cite{GS85}, reducing the steady-state bifurcation problem to a one-dimensional steady-state bifurcation problem, since the eigenvalue $\mu$ is simple. 
Moreover, the reduced problem is finitely determined and we can find at least one bifurcation branch. Since the synchrony subspace $\Delta$ is maximal, the bifurcation branch must have that synchrony. Next, we state this result and sketch its proof for completeness. We note that the
details can be filled using well-known techniques of bifurcation theory which can be found in many references including the references mention above.

\begin{thm}\label{thm:LSred1}
Let $\mathcal{N}$ be a network with $k$ asymmetric inputs and $\mu$ a network eigenvalue. Assume that $f \in \mathcal{V}_{\mu}({\mathcal{N}})$ is generic,  $\mu$ is simple and that $\Delta$ is a $\mu$-maximal synchrony subspace of $\mathcal{N}$.  Then, besides the trivial branch,  there exists a bifurcation branch of $f^{\mathcal{N}}$ with the synchrony  $\Delta$.
\end{thm}

\begin{proof}
Note that $\mu$ is assumed to be a simple eigenvalue of the quotient network 
of $\mathcal{N}$ determined by $\Delta$. 
Thus, we look for steady-state bifurcation branches with synchrony $\Delta$ which corresponds to solve 
the following equation in a neighbourhood of $(0,0)$:
$$
f^{\mathcal{N}}(x,\lambda) = 0,\quad (x,\lambda) \in \Delta \times \R\, . 
$$

Assume $\Delta$ is (isomorphic to) $\mathbb{R}^n$. 
In order to study the equation $f^{\mathcal{N}}(x,\lambda)=0$ in a neighbourhood of $(0,0)$, we use Lyapunov-Schmidt Reduction, 
\cite[Chapter I, Section 3; Chapter VII]{GS85}.  Since $\mu$ is simple, we can find $v, v^*\in \mathbb{R}^n$ such that $\ker(J_f^{\mathcal{N}})=\{\alpha v: \alpha\in\mathbb{R}\}$, $\ran( J_f^{\mathcal{N}})^{\bot}=\{\alpha v^*: \alpha\in\mathbb{R}\}$ and  $\langle v^*,v\rangle=1$. There exists a function $g:\mathbb{R}\times\mathbb{R}\rightarrow\mathbb{R}$ such that the solutions of $g=0$ are in one-to-one correspondence with the solutions of $f=0$. 
Moreover, the derivatives of $g$ at the origin can be computed using the derivatives of $f$ at the origin. In particular,
$$g_x=0, \quad g_\lambda=0, \quad g_{x\lambda}=\langle v^*, (D_x f_\lambda^{\mathcal{N}})_{(0,0)}(v)\rangle= 
\mu(f_{0\lambda},f_{1\lambda},\dots,f_{k\lambda})\, .$$
Since $f$ is generic, we can assume that $g_{x\lambda}\neq 0$. 

Next, we see that there exists $r>1$ such that  the $r$-derivative of $g$ with respect to $x$, $g_{x^{r}}$, is different from zero.
Using the formulas for the derivative of $g$, presented in \cite{GS85}, and the form of the admissible vector fields we know that
$$g_{x^{r}}= f_{0^r}\langle v^*,v^{[r]}\rangle + h_r(f),$$
where $f_{0^r}$ is the $r$ derivative of $f$ with respect to the first variable, $v^{[r]}=(v_1^r,\dots,v_n^r)$ and $h_r(f)$ is a polynomial function where its variables are the derivatives of $f$ with order less or equal to $r$ excluding $f_{0^r}$.
By \cite[Theorem 6.1.]{SG11}, we know that there exists $r$ such that $\langle u,v^{[r]}\rangle\neq 0$.
Since $f$ is generic, we can assume that $$f_{0^r}\neq - h_r(f)/\langle u,v^{[r]}\rangle \Leftrightarrow g_{x^{r}}\neq 0\, .$$
Let $r$ be the minimum integer such that $g_{x^{r}}\neq 0$ for $f$ generic.

Thus the function $g$ has the following Taylor expansion around the origin:
$$g(x,\lambda)=g_{x\lambda}x\lambda+ \frac{g_{x^{r}}}{r!}x^r + \cdots,$$
where $\cdots$ includes the terms which are higher order in $x$ and may depend at the bifurcation parameter $\lambda$.
Factoring out the variable $x$, we can apply the Implicit Function Theorem and conclude that there exists a nontrivial solution $\beta:]-\epsilon,\epsilon[\rightarrow\mathbb{R}$ such that 
$$g(x,\beta(x))=0.$$
This proves the existence of a bifurcation branch.
It follows from the maximality of the synchrony subspace that the bifurcation branch cannot be contained in any smaller synchrony subspace.
\end{proof}

\vspace{2mm}

In the previous proof, we defined $r$ as the minimum integer such that $g_{x^{r}}\neq 0$. 
In this case, we say that the {\it bifurcation problem is $r$-determined}.

\vspace{1mm}

\subsection{Semisimple eigenvalue and submaximal synchrony spaces} 

We consider now  two cases of synchrony-breaking steady-state bifurcations where the network eigenvalue $\mu$ is semisimple, $f\in\mathcal{V}_{\mu}({\mathcal{N}})$ and $\Delta$ is a $\mu$-submaximal synchrony subspace. 
In the first, we suppose that the eigenvalue has multiplicity $m$ and the submaximal synchrony subspace has order $2^m-1$.
In the second case, the network eigenvalue is the network valency.
The second case have been studied in Aguiar, Dias and Soares \cite{ADS19} and we recall known results in the end of this section. 
We note that the results we obtain can be applied independently of the network number of cells.

Suppose 
$\mu$ has multiplicity $m$ and $\Delta$ has order $2^m-1$.
Let $\Delta_1,\dots,\Delta_{2^m-1}$ be the $\mu$-maximal synchrony subspaces.
It follows from Theorem~\ref{thm:LSred1} that each $\mu$-maximal synchrony subspaces generically supports a bifurcation branch when $f\in\mathcal{V}_{\mu}({\mathcal{N}})$.
We prove in Theorem~\ref{thm:LSred2}  below that,  under certain conditions strictly depending on the network, there are no more bifurcation branches.
Using the terminology of Definition~\ref{def:Delta_supports}, we have then that 
$\Delta$ does not generically support a bifurcation branch when $f\in\mathcal{V}_{\mu}({\mathcal{N}})$.

In the proof of Theorem~\ref{thm:LSred2}  below, 
we use the  Lyapunov-Schmidt Reduction method to obtain a function $g:\mathbb{R}^m\times\mathbb{R}\rightarrow\mathbb{R}^m$ such that the solutions of $g=0$ (near the origin) are in one-to-one correspondence with the solutions of $f^{\mathcal{N}}=0$ (near the origin).
Since $f(0,\lambda)\equiv 0$ and each $\Delta_1,\dots,\Delta_{2^m-1}$ supports a bifurcation branch, we know that the reduced problem $g=0$ has at least $2^m$ zeros.
Supposing that the study of $g=0$ is equivalent to its quadratic approximation, we apply then
B\'ezout's Theorem \cite[Section 10.5]{BCSS98} to conclude that there are exactly those $2^m$ zeros.
Thus every bifurcation branch belongs to the $\mu$-maximal synchrony subspaces and $\Delta$ does not generically support a bifurcation branch.
The result holds under the following assumptions:

\begin{asps}\label{asps:LSred2}
Let $h$ be the second-order Taylor expansion at zero of the reduced function $g$.  Assume that:\\
{\bf 1.} $D_{(y,\lambda)}h(\tilde{y},1)$ has rank $m$, for each solution of $h(\tilde{y}\lambda,\lambda)=0$ with $\tilde{y}\neq 0$.\\
{\bf 2.} The homogeneous quadratic polynomial components  $h_i$, for $i=1,\dots, m$, of $h$  do not share a common factor.
\end{asps}

Depending on the network, these assumptions lead to non-degeneracy conditions of the bifurcation problem.
It follows from the first assumption that we can focus on the quadratic approximation of $g=0$. 
B\'ezout's Theorem can be applied when the second assumption holds.
First we state and prove the result, then we return to the previous assumptions and see how they depend on the network structure.

\begin{thm}\label{thm:LSred2}
Let $\mathcal{N}$ be a network with k asymmetric inputs,
$\mu$ a network eigenvalue, $f\in\mathcal{V}_{\mu}({\mathcal{N}})$ generic and $\Delta$ a synchrony subspace of $\mathcal{N}$ which is $\mu$-semisimple with multiplicity $m$ and  $\mu$-submaximal with order $2^m-1$.
Suppose that the bifurcation problem on each $\mu$-maximal synchrony subspace of $\mathcal{N}$ contained in $\Delta$ is $2$-determined and that Assumptions~\ref{asps:LSred2} hold.
Then there is no bifurcation branch of $f^{\mathcal{N}}$ with synchrony $\Delta$.
\end{thm}

\begin{proof}
Looking for steady-state bifurcation branches with synchrony $\Delta$ corresponds to solve the following equation in a neighbourhood of $(0,0)$, 
$$
f^{\mathcal{N}}(x,\lambda) = 0,\quad (x,\lambda) \in \Delta \times \R.
$$
As before,  assume $\Delta$ is (isomorphic to) $\mathbb{R}^n$. 

Let $\mu$ be a semisimple network eigenvalue with geometric
multiplicity $m$ and $f\in\mathcal{V}_{\mu}({\mathcal{N}})$ generic.
Suppose that $\R^n$ is $\mu$-submaximal with order $2^m-1$ and that Assumptions~\ref{asps:LSred2} hold.
Let $\Delta_1,\dots,\Delta_{2^m-1}$ be the $\mu$-maximal synchrony subspaces such that the bifurcation problem is $2$-determined on $\Delta_i$ for $ i=1,\dots, 2^m-1$.

Take $v_1,v_2,\dots,v_m, v_1^*,v_2^*,\dots, v_m^* \in \R^n$ such that $\ker(J_f^{\mathcal{N}}) = \mbox{span} \left(\{v_1,v_2,\dots,v_m\}\right)$, $ \ran(J_f^{\mathcal{N}})^{\bot} = \mbox{span} \left( \{v_1^*,v_2^*,\dots, v_m^*\} \right)$ and 
$$\displaystyle 
\left[
\begin{matrix}
\langle v^*_1,v_1\rangle & \langle v^*_1,v_2\rangle & \dots & \langle v^*_1,v_m\rangle\\ 
\langle v^*_2,v_1\rangle & \langle v^*_2,v_2\rangle & \dots & \langle v^*_2,v_m\rangle\\
\vdots & \vdots & \ddots& \vdots\\
\langle v^*_m,v_1\rangle & \langle v^*_m,v_2\rangle & \dots & \langle v^*_m,v_m\rangle
\end{matrix} 
\right] 
= \mbox{Id}_m\, .$$

By the Lyapunov-Schmidt Reduction Method, there exists a function $g:\mathbb{R}^m\times\mathbb{R}\rightarrow\mathbb{R}^m$ such that the solutions of $g=0$ (near the origin) are in one-to-one correspondence with the solutions of $f^{\mathcal{N}}=0$ (near the origin). 
Moreover, we have the following derivatives of $g$ at the origin
$$g_{x_i}=0, \quad g_\lambda=0, \quad g_{\lambda\lambda}=0, \quad (g_{x_i\lambda})_i=\mu(f_{0\lambda},\dots,f_{k\lambda}), \quad (g_{x_i\lambda})_j=0,$$
where $i,j=1,\dots,m$ and $j\neq i$. 
Since $f$ is generic we assume that $\mu(f_{0\lambda},\dots,f_{k\lambda})\neq 0$.
Expanding $g$ in its Taylor series around the origin, we have that
$$
g(x_1,\dots,x_m,\lambda)= \mu(f_{0\lambda},\dots,f_{k\lambda})\lambda\,  \mbox{Id}_m \, x+ Q(x)+\cdots,
$$
where $Q=(q_1,\dots,q_n)$ has quadratic homogeneous polynomials components in $x$.
Denote by
$h$ the second-order Taylor expansion truncation
of $g$, 
\begin{equation}
h(x_1,\dots,x_m,\lambda)= \mu(f_{0\lambda},\dots,f_{k\lambda})\lambda\,  \mbox{Id}_m \, x + Q(x), 
\label{eq:homquadh}
\end{equation}
which satisfies
$$h(\lambda x_1,\dots,\lambda x_m,\lambda)= \lambda^2(\mu(f_{0\lambda},\dots,f_{k\lambda}) x + Q(x))=\lambda^2 h(x_1,\dots,x_m,1)\, .$$
Since $\Delta_1,\dots,\Delta_{2^m-1}$ are $\mu$-maximal synchrony subspaces, it follows from Theorem~\ref{thm:LSred1} that $\Delta_i$ supports a bifurcation branch for each  $i=1,\dots, 2^m-1$.
Adapting \cite[Proposition 3.5]{S17}, we know for each $i=1,\dots, 2^m -1$ that the bifurcation branch on $\Delta_i$ leads to a non-trivial solution $\tilde{y}^i=(\tilde{y}^i_1,\dots,\tilde{y}^i_m)$ of $h(x_1,\dots, x_m,1)=0$, since the bifurcation problem is $2$-determined on $\Delta_i$.
Moreover, we have that $\tilde{y}^i_1 v_1+\dots+\tilde{y}^i_m v_m\in\Delta_i$ and the solutions $\tilde{y}^1, \dots, \tilde{y}^{2^m -1}$ are distinct.
Using Assumption~\ref{asps:LSred2}.1 and adapting \cite[Proposition 3.6]{S17}, we see that each of these solutions corresponds to an unique solution of $g=0$ and a unique bifurcation branch of $f^{\mathcal{N}}$.

From B\'ezout's Theorem \cite[Section 10.5]{BCSS98} and Assumption~\ref{asps:LSred2}.2, the system $$h(x_1,\dots,x_m,1)=(h_1(x_1,\dots,x_m,1),\dots,h_m(x_1,\dots,x_m,1))=0$$ has at most $2^m$ solutions as the polynomials $h_1,\dots,h_m$ are homogeneous of degree $2$.
Counting the solutions $\tilde{y}^1, \dots, \tilde{y}^{2^m -1}$ together with the trivial solution $\tilde{y}=0$, there are $2^m$ solutions of $h(x_1,\dots,x_m,1)=0$.
So there is no more solutions of $h(x_1,\dots,x_m,1)=0$ and there are exactly the $2^m -1$ bifurcation branches. 
Thus the $\mu$-submaximal synchrony subspace does not support a bifurcation branch.
\end{proof}

\begin{rem}
Note that the valency eigenvalue of a network with $n$ cells has multiplicity $n$ if and only if every cell of the network is isolated from the others.
Thus any network eigenvalue of a connected network with $n$ cells has multiplicity less or equal than $n-1$.
Moreover, $\R^n$ has $2^{n-1}-1$ polydiagonal subspaces with dimension $2$ since each polydiagonal subspace with dimension $2$ defines a partition of the set $\{ 1,\ldots,n\}$ into two complementary subsets.
Thus there are at most $2^{n-1}-1$ synchrony subspaces with dimension $2$.
\hfill $\Diamond$
\end{rem}

\begin{rem}
In Theorem~\ref{thm:LSred2}, we assume that the bifurcation condition is given by a semisimple network eigenvalue $\mu$ with multiplicity $m$ and the submaximal subspace has order $2^m-1$. From the previous remark, we see that this condition is not expected to occur frequently for coupled cell systems. Nevertheless, we show in the  three-cell 
networks steady-state bifurcations classification done in the  following two sections that there is one network for which the bifurcation condition  assumed in Theorem~\ref{thm:LSred2} holds. Namely, network $C_1 \& D_1$ in Table~\ref{tab:C3L12.tex} has  the eigenvalue $f_0$ with multiplicity $2$, the network phase space $\R^3$ is $f_0$-submaximal with order $3$ (see Figure~\ref{lat:C2L3s}) and Theorem~\ref{thm:LSred2} holds. 
\hfill $\Diamond$ 
\end{rem}

We see now that Assumptions~\ref{asps:LSred2} depend solely on the network structure. 
More precisely, the dependence is on 
the second derivatives of the function $g$ at the origin, which in their turn 
are computed using only
the second derivatives of the function $f$ at the origin, and the vectors $v_1,\dots,v_m,v_1^*,\dots,v_m^*$. 
In order to calculate the second derivatives of $g$ and $h$ at the origin, we use the formulas provided in \cite[Chapter I, Section 3; Chapter VII]{GS85}.

The function $f^{\mathcal{N}}$ has the following second-order Taylor expansion around the origin
$$f^{\mathcal{N}}(x,\lambda)=\sum_{j=0}^k f_j A_j x+\sum_{j=0}^k f_{j\lambda} \lambda A_j x + \sum_{j_1=0}^k \sum_{j_2=0}^k \frac{f_{j_1j_2}}{2} (A_{j_1} x)*(A_{j_2} x)+\cdots$$ 
and the second order directional derivative of $f^{\mathcal{N}}$ in the directions $v_a$ and $v_b$ is
$$d^2f^{\mathcal{N}}(v_a,v_b)=\sum_{j_1=0}^k \sum_{j_2=0}^k f_{j_1j_2} (A_{j_1} v_a)*(A_{j_2} v_b),$$
where
$a,b=1,\dots, m$ and $f_{j_1j_2}$ is the second derivative of $f$ with respect to $x_{j_1}$ and $x_{j_2}$ at the origin for $j_1,j_2=0,\dots,k$. 
Here, $*$ denotes the componentwise product.
The vectors $v_1,\dots,v_m$ usually depend on the first order derivatives of $f$ and the expressions above cannot be further simplified. 
Note that the bifurcation condition may be seen as a restriction on the value of $f_0$.
The functions $g$ and $h$ have the same second order derivatives at the origin
$$(g_i)_{ab}=(h_i)_{ab}=\sum_{j_1=0}^k \sum_{j_2=0}^k f_{j_1j_2} <v_i^*,(A_{j_1} v_a)*(A_{j_2} v_b)>,$$
where $i,a,b=1,\dots, m$.

The first assumption in Assumptions~\ref{asps:LSred2} may be expressed as a set
of inequalities using the previous derivatives, excluding $f_0$.
Thus, if those inequalities are nontrivial, they lead to non-degeneracy conditions on the function $f$.
The fact that those inequalities are trivial or not depends only on the network structure.
Hence, if the first assumption in Assumptions~\ref{asps:LSred2} holds for some function, then they must hold for any generic function $f$.

For the second assumption in Assumptions~\ref{asps:LSred2}, we have the following lemma about the common factors of the components $h_1,\dots,h_m$ of the function $h$ given in (\ref{eq:homquadh}).

\begin{lemma}\label{aux_lemma}
Let $h_1,\dots,h_m:\, \R^{m+1} \to \R$ be polynomial functions with the following form
$$h_i(x_1,\dots,x_m,\lambda)=\mu x_i \lambda + q_i(x_1,\dots,x_m),$$
where $i=1,\dots,m$ and $q_i$ are homogeneous quadratic polynomials.
If $i\not=j$ and $h_i$ and $h_j$ have a common factor, then
\begin{equation}\label{formulas}
\begin{array}{lll}
(h_i)_{dl}=0 \ (\mbox{if } l,d\neq i),  &   (h_j)_{dl}=0 \ (\mbox{if } l,d\neq j), & (h_i)_{ki}= (h_j)_{kj}
\end{array}
\end{equation}
where $(h_a)_{bc}$ denotes the second order derivative of $h_a$ with respect to the variables $x_b$ and $x_c$ at the origin and $ a,b,c=1, \dots , m$.
\end{lemma}

\begin{proof}
Let $h_1,\dots,h_m$ be the functions with the given form.
Suppose that $h_i$ and $h_j$ have a common factor.
Since there is only one term of $h_i$ which includes the variable $\lambda$ and it is equal to $\mu x_i \lambda$, 
we know that the common factor of $h_i$ and $h_j$ must be linear.
That is, $h_i$ and $h_j$ have the following form: 
$$
\begin{array}{ll}
h_i(x_1, \ldots, x_m,\lambda) = x_i  \left( \mu  \lambda +   \displaystyle \sum_{l=1}^m c_l x_l\right), &
h_j(x_1, \ldots, x_m,\lambda) = x_j \left( \mu  \lambda +    \displaystyle \sum_{l=1}^m c_l x_l\right)
\end{array}
$$
for constants $c_1, \ldots, c_m \in \R$. Thus 
$$
\begin{array}{ll}
q_i(x_1, \ldots, x_m) = \displaystyle \sum_{l=1}^m c_l x_i x_l, \quad & \quad 
q_j(x_1, \ldots, x_m) = \displaystyle \sum_{l=1}^m c_l x_j x_l\, .
\end{array}
$$ We have then the equalities (\ref{formulas}).
\end{proof}

It follows from Lemma~\ref{aux_lemma} and the expression for the second derivatives of $h$ given above that the second assumption in Assumptions~\ref{asps:LSred2} leads to a set 
of inequalities that may be expressed using the derivative of $f$, excluding $f_0$.
If those inequalities are nontrivial, then they form a set of non-degeneracy conditions on the function $f$.
Thus if  both Assumptions~\ref{asps:LSred2} are satisfied by one function $f$ for a given network $\mathcal{N}$, then every generic coupled cell system of $\mathcal{N}$ satisfies Assumptions~\ref{asps:LSred2}.

We finish this section with the bifurcation problem when the bifurcation condition is given by the valency, $\upsilon$. 
We point out that this case has  been studied in \cite{ADS19}. 

We recall the definition of a source of a network:

\begin{Def}
Let $\mathcal{N}$ be a network. 
We say that a subset of cells $S$ is a {\it source} if every edge targeting a cell in $S$ starts in a cell of $S$, there exists a directed path between any two cells of that subset and $S$ is the maximal subset that satisfy the previous two conditions.
The {\it number of source components} of a network is denoted by $s({\mathcal{N}})$.
\hfill $\Diamond$
\end{Def}

\begin{exam}
The network $E_6\&E_4$ in Table~\ref{tab:C3L12.tex} of Section~\ref{subsec_case_study} has two source components $S_1=\{1\}$ and $S_2=\{3\}$.
\hfill $\Diamond$ 
\end{exam}

Note that there exists at least one source, i.e., $s({\mathcal{N}})\geq 1$.

In a network with $k$ asymmetric inputs, the polydiagonal where $x_i = x_j$ for all the cells $i,j$ in one source is a network synchrony space because  all the cells inside a source receive $k$ edges, only from cells inside that source.
Suppose that $s({\mathcal{N}})\geq 2$. 
Given two source components, the polydiagonal subspace given by $x_i = x_j$ for all the cells $i,j$ in these two source components is also a synchrony subspace.
Repeating this synchronization for more sources, we can find synchrony subspaces such that the corresponding quotient network has exactly two source components. 
In particular, we can find the smallest synchrony subspace where the corresponding quotient network has only two source components.

In those smallest synchrony subspaces the valency eigenvalue is semisimple with multiplicity $2$ and they are $\upsilon$-submaximal with order $1$.
We denote those smallest synchrony subspaces by {\it valency synchrony-breaking subspaces} and the name is clear by the next result.

\begin{prop}\cite[Proposition 5.7]{ADS19}\label{prop:valsynbre}
Let $\mathcal{N}$ be a network with asymmetric inputs, $\Delta$ a  valency synchrony-breaking subspace of $\mathcal{N}$ and $f\in\mathcal{V}_{\upsilon}({\mathcal{N}})$ generic where $\upsilon$ denotes the valency eigenvalue of $\mathcal{N}$. Then there exists a bifurcation branch of $f^{\mathcal{N}}$ with the synchrony associated to $\Delta$.
\end{prop}

\subsection{Defective synchrony spaces}
There are networks with defective network eigenvalues, i.e., the algebraic and geometric multiplicity do not coincide.
In this section, we study the case of defective synchrony spaces in networks with any number of cells.
We prove that a $\mu$-defective synchrony subspace with multiplicity $(1,2)$ supports a bifurcation branch.

Consider a network $\mathcal{N}$ with asymmetric inputs, a network eigenvalue $\mu$, a generic function $f\in\mathcal{V}_{\mu}({\mathcal{N}})$ and a synchrony subspace $\Delta$ which is $\mu$-defective with multiplicity $(1,2)$ and $\mu$-submaximal.
Note that $\Delta$ must be submaximal with order $1$, because $\mu$ has geometric multiplicity $1$ and algebraic multiplicity $2$.
Denote by $\Delta_1$ the $\mu$-maximal synchrony subspace contained in $\Delta$.
Note that $\ker({J_f^{\mathcal{N}}})\cap \Delta \subset \Delta_1$ and there are $v_1,v_2 \in \ker({J_f^{\mathcal{N}}}^2)$ and $v^*_1,v^*_2 \in (\ran({J_f^{\mathcal{N}}}^2))^{\bot}$ such that $v^*_1,v_1\in\Delta_1$, $J_f^{\mathcal{N}} v_2= v_1$, $\langle v^*_1, v_1\rangle =1$, $\langle v^*_2, v_2\rangle =1$ and $\langle v^*_1, v^*_2\rangle =0$.
To prove the generic existence of bifurcation branches, we adapt the Lyapunov-Schmidt Reduction Method~\cite{GS85} 
considering vectors in the above conditions. 
Moreover, we need to impose the following condition:
there are $(p,q)$ such that
\begin{equation}
\langle v_2^*,[A_p(v_{2}- (PJ^{\mathcal{N}}_f)^{-1}P v_1)]*[A_q(v_{2}- (PJ^{\mathcal{N}}_f)^{-1}P v_1)]\rangle\neq 0,
\label{eq:2detdef}
\end{equation}
where $A_0= \mbox{Id}$, $A_1$,\dots,$A_k$ are the adjacency matrix of ${\mathcal{N}}$ and $P$ is the projection onto $\ran({J_f^{\mathcal{N}}}^2)$ and kernel $(\ran({J_f^{\mathcal{N}}}^2))^{\bot}$.
This condition leads to a non-degeneracy condition on the function $f$. 

We can now state the result:

\begin{thm} \label{thm:LSred3}
Let ${\mathcal{N}}$ be a network with asymmetric inputs,
$\mu$ be a defective network eigenvalue, $f\in\mathcal{V}_{\mu}({\mathcal{N}})$ generic and $\Delta$ be a $\mu$-submaximal synchrony subspace of ${\mathcal{N}}$ where $\mu$ has multiplicity $(1,2)$.  
Suppose that condition (\ref{eq:2detdef}) holds. 
Then there exists a bifurcation branch of $f^{\mathcal{N}}$ with the synchrony $\Delta$.
\end{thm}

\begin{proof}
Since we are interested in proving the existence of bifurcation branches with the synchrony $\Delta$, 
that is, we look for steady-state bifurcation branches with synchrony $\Delta$,   the following equation has to be solved in a neighbourhood of $(0,0)$:
$$
f^{\mathcal{N}}(x,\lambda) = 0,\quad (x,\lambda) \in \Delta \times \R\, . 
$$

Assume $\Delta$ is (isomorphic to) $\mathbb{R}^n$. 
Denote by $\Delta_1$ the $\mu$-maximal synchrony subspace and let $v_1,v_2 \in \ker({J_f^{\mathcal{N}}}^2)$ and $v^*_1,v^*_2 \in (\ran({J_f^{\mathcal{N}}}^2))^{\bot}$ such that $v^*_1,v_1\in\Delta_1$, $J_f^{\mathcal{N}} v_2= v_1$, $\langle v^*_1, v_1\rangle =1$, $\langle v^*_2, v_2\rangle =1$ and $\langle v^*_1, v^*_2\rangle =0$.
Note that $v_2^*\bot \Delta_1$ and we have the following splits of $\mathbb{R}^{n}$: 
$$\mathbb{R}^{n}= \ker({J_f^{\mathcal{N}}}^2)\oplus \ran({J_f^{\mathcal{N}}}^2),\quad\quad \mathbb{R}^{n}= (\ran({J_f^{\mathcal{N}}}^2))^{\bot}\oplus \ran({J_f^{\mathcal{N}}}^2).$$

Applying the Lyapunov-Schmidt Reduction Method \cite{GS85} to $f^{\mathcal{N}}=0$, we obtain a function $g:\mathbb{R}^2\times\mathbb{R}\rightarrow\mathbb{R}^2$ such that the solutions of $g=0$ are in one-to-one correspondence with the solutions of $f=0$. 
Note that $g_2(x_1,0,\lambda)=0$ and 
$$\frac{\partial g_1}{\partial \lambda}=\frac{\partial g_2}{\partial \lambda}=\frac{\partial g_1}{\partial x_1}=\frac{\partial g_1}{\partial x_2}=\frac{\partial g_2}{\partial x_2}=0 \quad \frac{\partial g_1}{\partial x_2}=1,$$
$$\frac{\partial^2 g_1}{\partial \lambda^2}=\frac{\partial^2 g_2}{\partial x_1 \partial \lambda}=\frac{\partial^2 g_2}{\partial x_1^2}=0,\quad \frac{\partial^2 g_1}{\partial x_1 \partial \lambda}=\frac{\partial^2 g_2}{\partial x_2 \partial \lambda}=\mu(f_{0\lambda},f_{1\lambda},\dots,f_{k\lambda})\neq 0,$$
$$\frac{\partial^2 g_1}{\partial x_2 \partial \lambda}=\mu(f_{0\lambda},f_{1\lambda},\dots,f_{k\lambda})\langle v_1^*,v_2\rangle + 1 \neq 0,$$
for a generic $f$. 

It follows from the Lyapunov-Schmidt Reduction Method that there exists a function
$W:\mathbb{R}^2\times\mathbb{R}\rightarrow \ran({J_f^{\mathcal{N}}}^2)$ such that $Pf^{\mathcal{N}}(x_1v_1+x_2v_2+W(x_1,x_2,\lambda),\lambda)=0$ where $P$ is the projection onto $\ran({J_f^{\mathcal{N}}}^2)$ and kernel $(\ran({J_f^{\mathcal{N}}}^2))^{\bot}$. The first derivatives of $W$ are 
$$W_{1}=0, \quad \quad W_2=- (PJ^{\mathcal{N}}_f)^{-1}P v_1.$$
We can also calculate the second derivatives of $g$ and we obtain that
$$(g_i)_{j_1 j_2}:=\frac{\partial^2 g_i}{\partial x_{j_1} \partial x_{j_2}}=\langle v_i^*, d^2f^{\mathcal{N}}(v_{j_1}+ W_{j_1},v_{j_2}+W_{j_2})\rangle$$
$$=\sum_{p,q=0}^k f_{pq}\langle v_i^*,[A_p(v_{j_1}+W_{j_1})]*[A_q(v_{j_2}+W_{j_2})]\rangle,$$
where $j_1 j_2=1,2$, $f_{pq}$ are the second order derivative of $f$ at the origin with respect to the variables $p+1$ and $q+1$ and $W_j$ is the derivative of $W$ with respect to $x_j$ at the origin. 

Since $g_2(x_1,0,\lambda)=0$, there is $h_2(x_1,x_2,\lambda)$ such that $g_2(x_1,x_2,\lambda)=x_2h_2(x_1,x_2,\lambda)$ and 
$$g_2(x_1,x_2,\lambda)=0\Leftrightarrow x_2=0\vee h_2(x_1,x_2,\lambda)=0\, .$$

In the first case, $x_2=0$, we are looking for bifurcation branches in the synchrony subspace $\Delta_1$.   
 By Theorem~\ref{thm:LSred1},  there exists a bifurcation branch of steady-state solutions with synchrony $\Delta_1$. 
We are now interested in solving $h_2(x_1,x_2,\lambda)=0$ providing a bifurcation branch of steady-state solutions having synchrony $\Delta$ but which are not  
$\Delta_1$-synchronous. 

In the second case, in order to solve equation $h_2(x_1,x_2,\lambda)=0$, we use condition (\ref{eq:2detdef}) which implies that $(g_2)_{2 2}\neq 0$ for $f$ generic.
It follows from the Implicit Function Theorem that there exists $\beta:\mathbb{R}\times \mathbb{R}\rightarrow \mathbb{R}$ such that $x_2=\beta(x_1,\lambda)$ is the unique solution of $h_2(x_1,x_2,\lambda)=0$ in a neighbourhood of the origin. The derivative of $\beta$ with respect to $\lambda$  at the origin is:   
$$
\beta_{\lambda}=\frac{\partial \beta}{\partial \lambda}=-\frac{(g_2)_{2\lambda}}{(g_2)_{22}}\neq 0\, .
$$
Replacing $x_2$ by $\beta$ in the function $h_2$ of equation $h_2(x_1,x_2,\lambda)=0$, 
we obtain the function 
$$h_1(x_1,\lambda)=g_1(x_1,\beta(x_1,\lambda),\lambda)$$
which has the following nonnull derivative at the origin: 
$$\frac{\partial h_1}{\partial \lambda}=\beta_{\lambda}\neq 0.$$
Again by the  Implicit Function Theorem, there exists a continuous function $\Lambda:\mathbb{R}\rightarrow \mathbb{R}$ such that 
$$h_1(x_1,\Lambda(x_1))=0.$$
Therefore 
$$f^{\mathcal{N}}(x_1 v_1 + \beta(x_1,\Lambda(x_1))v_2+W(x_1,\beta(x_1,\Lambda(x_1)),\Lambda(x_1) ),\Lambda(x_1))=0.$$

Since the function $f$ is generic,  the origin is an isolated zero of the function $f$ at $\lambda=0$. 
Thus the function $\Lambda$ is not constant and we can write, at least part of, the graph $(x_1,\Lambda(x_1))\subset \mathbb{R}^2$ as a graph of a function in $\lambda$. We obtain so a nontrivial bifurcation branch of steady-state solutions of $f^{\mathcal{N}}$ with synchrony $\Delta$. 
\end{proof}

\begin{rem}\label{latfigura60}
The previous proof also holds even if the synchrony subspace is $\mu$-maximal instead of $\mu$-submaximal.  
Specifically, if the synchrony subspace $\Delta$ in Theorem~\ref{thm:LSred3} is $\mu$-maximal instead of $\mu$-submaximal, 
we can still find a solution of $g=0$ with $x_2=0$ and this solution leads to a non-trivial bifurcation branch of steady-state solutions 
 of $f^{\mathcal{N}}$ with synchrony $\Delta$, since $\Delta$ is maximal.
\hfill $\Diamond$
\end{rem}

\section{Steady-state bifurcations for three-cell networks with asymmetric inputs}\label{sec_three_cells}

In this section we address three-cell networks with any number of asymmetric inputs. We obtain the network eigenvalues and lattices of synchrony subspaces which combined with the results of  the previous section  derive  the possible codimension-one steady-state for three-cell networks with asymmetric inputs and corresponding bifurcation diagrams.

\subsection{Eigenvalue structure}

 Let ${\mathcal{N}}$ be a three-cell network with asymmetric inputs and $f^{\mathcal{N}}$ an admissible coupled cell system for ${\mathcal{N}}$. 
Recalling (\ref{eq_Jac}), we have that the 
 Jacobian matrix of $f^{\mathcal{N}}$ at the origin is a $3\times 3$ matrix determined by the adjacency matrices of ${\mathcal{N}}$ and the first derivatives of $f$ at the origin.  It follows, in particular, that  the Jacobian matrix has constant row-sum, say $\upsilon$, which is an eigenvalue of such matrix and $(1,1,1)$ is 
a corresponding eigenvector. For completeness, we collect in Proposition~\ref{Prop:eigen_reg_matrix} the possible eigenvalue structures of a general $3\times 3$ matrix with constant row sum.

\begin{prop} \label{Prop:eigen_reg_matrix}
Let $A$ be a $3 \times 3$ matrix with real entries and  constant row-sum $\upsilon\not=0$. 
Suppose that 
\begin{equation} \label{eq:mat}
A= \left[
\begin{array}{ccc}
a & b& \upsilon-a-b \\
c & d& \upsilon-c-d \\
e & f& \upsilon-e-f
\end{array}
\right]\, .
\end{equation}

\noindent and denote by  $\displaystyle \alpha_0 = \mbox{det}(A) / \upsilon$ and $\alpha_1 = \mbox{tr}(A) - \upsilon$. 
Table~\ref{tab:esA} lists the possible cases for the eigenvalue structure  of the matrix $A$. 
\end{prop}

\begin{proof}
Suppose $A$ is given by (\ref{eq:mat}).  
As $(1,1,1)$ is an eigenvector of $A$ associated with the eigenvalue $\lambda_1 = \upsilon$, taking the non-singular matrix $P$ and the $2 \times 2$ matrix $S$ 
given by 
$$P =  \left[
\begin{array}{ccc}
1 & 0 & 1 \\
0 & 1 & 1 \\
0 & 0 & 1
\end{array}
\right], \qquad 
S = 
 \left[
\begin{array}{cc}
 a-e & b-f \\
c-e & d-f 
\end{array}
\right], 
$$ 
we have that 
$$
P^{-1} A P = \left[
\begin{array}{c|c}
S & 0_{2,1} \\
\hline
e \, \, f & \upsilon
\end{array}
\right]\, .
$$
 Thus $\alpha_0 = \mbox{det}(S) = \mbox{det}(A) /\upsilon, \ \alpha_1 = \mbox{tr}(S) = \mbox{tr}(A) - \upsilon$ and the characteristic polynomial of $A$ is given by $p_A(\mu) = \left| A -\mu \mbox{Id}_3 \right| = -(\mu - \upsilon) (\mu^2 - \alpha_1\mu +\alpha_0)$. Denote by $m_a(\upsilon)$ the algebraic multiplicity of the constant row-sum eigenvalue $\upsilon$. We have the following cases: \\
(i) $m_a(\upsilon)=3$ if and only if $\alpha_0 = \upsilon^2$ and $\alpha_1=2\upsilon$.\\
(ii) $m_a(\upsilon)=2$ if and only if $\alpha_1 \not= 2\upsilon$ and $\upsilon^2 -\alpha_1 \upsilon + \alpha_0 = 0$. \\
(iii) $m_a(\upsilon)=1$ and there is an eigenvalue $\lambda_2 \not=\upsilon$ with algebraic multiplicity $2$ if 
and only if $\alpha_1^2 = 4\alpha_0$ and $\alpha_1\neq 2\upsilon$. 
Moreover, $\lambda_2 = \lambda_3 = \alpha_1/2$ and the geometric multiplicity of $\lambda_2$ is equal to the dimension of the kernel of the matrix $P^{-1} A P - \alpha_1/2 \mbox{Id}_3$. 
Trivially, the dimension of the kernel of the matrix $P^{-1} A P - \alpha_1/2 \mbox{Id}_3$ is equal to the dimension of the kernel of the matrix $B = 2 \left( S - \alpha_1/2 \mbox{Id}_2\right)$. 
Note that 
$$B=\left[
\begin{array}{cc}
(a-e)-(d-f) & 2(b-f) \\
2(c-e) & (d-f)-(a-e)  
\end{array}
\right], 
$$
and so, the dimension of $\mbox{ker}(B)$ is two if and only if $B=0$ if and only if $b-f=0,\, c-e =0,\, d-f = a-e$.\\
(iv) Finally, $A$ has three distinct eigenvalues if and only if $\upsilon$ is not an eigenvalue of $S$ ($\upsilon^2 -\alpha_1 \upsilon + a_0 \not= 0$, that is, $\upsilon (\alpha_1-\upsilon) \neq \alpha_0$) and $S$ has two distinct roots ($\alpha_1^2 \neq 4\alpha_0$).  The roots of the characteristic polynomial of $S$ are real if and only if $\alpha_1^2>4\alpha_0$. 
\end{proof}

\begin{table}[t!]
\resizebox{1 \textwidth}{!}{ 
\begin{tabular}{|c|c||c|c|}
\hline 
Eigenvalues of $A$ & Conditions & Eigenvalues of $A$ & Conditions\\[0.3cm]
\hline 
$\lambda_1 = \lambda_2 = \lambda_3 = \upsilon$ & $\alpha_0=\upsilon^2$ and $\alpha_1=2 \upsilon$ & 
$\lambda_1 = \lambda_2 = \upsilon,\, \lambda_3 = \alpha_1 - \upsilon \not=\upsilon$ & $\alpha_0 = \upsilon (\alpha_1 -\upsilon)$ and $\alpha_1\neq 2\upsilon$\\[0.3cm]
\hline
$\lambda_1 = \upsilon,\, \lambda_2 = \lambda_3 =\frac{\alpha_1}{2} \not=\upsilon$ & $\alpha_1\neq 2\upsilon$ and $\alpha_1^2 = 4\alpha_0$ & 
$\lambda_1 = \upsilon,\, \lambda_2 = \lambda_3 =\frac{\alpha_1}{2} \not=\upsilon$ & $\alpha_1\neq 2\upsilon$ and $\alpha_1^2 = 4\alpha_0$  \\
$A$ is not diagonalizable & $(c-e,b-f, d-a+e-f) \neq (0,0,0)$ & $A$ is diagonalizable & $(c-e,b-f, d-a+e-f) = (0,0,0)$ \\[0.3cm]
\hline 
$\lambda_1 = \upsilon,\, \upsilon \not= \lambda_2 \not= \lambda_3 \not= \upsilon$ & $\upsilon (\alpha_1-\upsilon) \neq \alpha_0 $ and $\alpha_1^2 \neq 4\alpha_0$ & 
$\lambda_1 = \upsilon,\, \upsilon \not= \lambda_2 \not= \lambda_3 \not= \upsilon$ & $\upsilon (\alpha_1-\upsilon) \neq \alpha_0 $ and $\alpha_1^2 \neq 4\alpha_0$ 
\\ 
$\lambda_2,\, \lambda_3 \in \R$ & $\alpha_1^2>4\alpha_0$ & $\lambda_2,\, \lambda_3 \not\in \R$ and $\lambda_2=\overline{\lambda_3}$ & $\alpha_1^2 < 4\alpha_0$ \\
\hline
\end{tabular}
}
 \caption{
Eigenvalues $\lambda_1, \lambda_2, \lambda_3$
 of a $3 \times 3$ matrix $A$ with constant row-sum $\upsilon\not=0$. Here,  $\alpha_0 = \mbox{det}(A) / \upsilon$, $\alpha_1 = \mbox{tr}(A) - \upsilon$ and it is followed  the notation (\ref{eq:mat}) for the entries of the matrix $A$. 
In the top left case $m_a(v) = 3$, in the top right case $m_a(v) =2$ and in the other cases, $m_a(v) =1$. 
 }
 \label{tab:esA}
 \end{table}

\begin{rem}
(i) Note that when the network has one asymmetric input,   the Jacobian matrix of $f^{\mathcal{N}}$ at the origin is $J_f^{\mathcal{N}}= f_0 \mbox{Id}_3+ f_1 A_1$ which has eigenvalues given by $f_0+\mu f_1$, where $\mu$ runs through the eigenvalues of the adjacency matrix $A_1$. This is a special case of (\ref{eq_Jac}).
\\
(ii)
Among the results obtained in ~\cite{ADS20}, it is remarked that the minimal network in Figure~\ref{fig:rep_min} represents the unique ODE-class of minimal three-cell networks with six asymmetric inputs. The Jacobian matrix of $f^{\mathcal{N}}$ at the origin for such network has the form $J_f^{\mathcal{N}}= f_0 \mbox{Id}_3+ \sum_{i=1}^6 f_i  A_i$, where $A_i$, for $i=1, \ldots, 6$ are the network adjacency matrices. It also follows from~\cite{ADS20} that the six adjacency matrices $A_i$ plus the $3 \times 3$ identity matrix generate the linear space of the $3\times 3$ matrices with constant row sum.Therefore, besides the valency eigenvalue of $J_f^{\mathcal{N}}$, the other eigenvalues  are generically arbitrary and simple. Thus there exists an open set of generic functions $f$ where the eigenvalues of the Jacobian matrix $J_f^{\mathcal{N}}$ are distinct and real.
\hfill $\Diamond$
\end{rem}

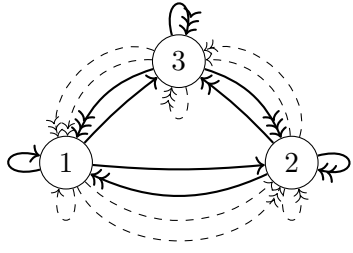
\begin{figure}[h!]
\vspace{-4mm}
\hspace{-4mm}
\begin{tikzpicture}
[scale=.15,auto=left, node distance=1.5cm, every node/.style={circle,draw}]
 \node[fill=white] (n1) at (4,0) {\small{1}};
  \node[fill=white] (n2) at (24,0) {\small{2}};
 \node[fill=white] (n3) at (14,9)  {\small{3}};

\draw[ arrows={->}, thick]  (n1) to [loop left] (n1);
\draw[ arrows={->}, thick] (n1) edge  [bend left=-5] (n2);
\draw[ arrows={->}, thick]  (n1) edge  [bend left=5]  (n3);

 \draw[ arrows={->>}, thick]  (n2) to [loop right] (n2);
\draw[ arrows={->>}, thick]  (n2) edge  [bend left=25] (n1);
\draw[ arrows={->>}, thick] (n2) edge  [bend left=-5] (n3);

 \draw[ arrows={->>>}, thick]  (n3) to [loop above] (n3);
 \draw[ arrows={->>>}, thick]   (n3) edge  [bend left=-25] (n1);
 \draw[ arrows={->>>}, thick]  (n3) edge  [bend left=25] (n2);

 \draw[ arrows={->}, dashed] (n1) to [loop below] (n1);
\draw[ arrows={->}, dashed]  (n1) edge  [bend left=-45] (n2);
\draw[ arrows={->}, dashed]  (n2) edge  [bend left=-45] (n3);

\draw[ arrows={->>}, dashed]  (n2) to [in=285,out=255, loop] (n2);
\draw[ arrows={->>}, dashed]  (n2) edge  [bend left=-65] (n3);
\draw[ arrows={->>}, dashed]  (n3) edge  [bend left=-45] (n1);

\draw[ arrows={->>>}, dashed] (n3) to [loop below] (n3);
\draw[ arrows={->>>}, dashed] (n3) edge  [bend left=-65] (n1);
\draw[ arrows={->>>}, dashed] (n1) edge  [bend left=-65] (n2);
 \end{tikzpicture} 
		\caption{A representative of the minimal class of the networks with three-cells and six asymmetric inputs.}
		\label{fig:rep_min}
\end{figure}

\begin{exam}
Table~\ref{tab:val13cell} lists, up to ODE-equivalence, the minimal connected networks with three-cells and one asymmetric input and the eigenvalues and eigenvectors of the associated adjacency matrices.
\hfill $\Diamond$
\end{exam}

\begin{table}
 \begin{center}
 \resizebox{1 \textwidth}{!}{ 
 {\tiny 
 \begin{tabular}{|cc|c|c|c||cc|c|c|c|}
\hline 
&  & 2D   & Adjacency  & Eigenvalues/ & &  & 2D   & Adjacency  & Eigenvalues/ \\
& Network &  Synchrony &  Matrix & Eigenvectors & & Network &  Synchrony &  Matrix & Eigenvectors \\
&  &  Subspaces &   &  & &  &  Subspaces &   & \\

\hline 
A &
\begin{tikzpicture}
 [scale=.15,auto=left, node distance=1.5cm, every node/.style={circle,draw}]
 \node[fill=white] (n1) at (4,0) {\small{1}};
  \node[fill=white] (n2) at (24,0) {\small{2}};
 \node[fill=white] (n3) at (14,9)  {\small{3}};
 \draw[->, thick] (n1) edge  [bend left=-10] (n2); 
 \draw[->, thick] (n2) edge  [bend left=-10] (n3); 
\draw[->, thick] (n3) edge [bend right=10] (n1); 
\end{tikzpicture} & 
- & 
$\left[
\begin{array}{ccc}
0 & 0 & 1 \\
1 & 0 & 0 \\
0 & 1 & 0
\end{array}
\right]$ & 
$ \begin{array}{rl}
1 & (1,1,1) \\
\psi & (1, \psi^2, \psi) \\
 \psi^2 & (1, \psi, \psi^2) 
 \end{array}$ 
 &
F& 
\begin{tikzpicture}
 [scale=.15,auto=left, node distance=1.5cm, every node/.style={circle,draw}]
 \node[fill=white] (n1) at (4,0) {\small{1}};
  \node[fill=white] (n2) at (24,0) {\small{2}};
 \node[fill=white] (n3) at (14,9)  {\small{3}};
 \draw[->, thick] (n1) edge  [bend left=-10] (n2); 
 \draw[->, thick] (n2) edge  [bend left=-10] (n1); 
\draw[->, thick] (n1) edge [bend right=10] (n3); 
\end{tikzpicture} &
$\begin{array}{l}
\Delta_1 \\
\Delta_3 \\
\end{array}$ &
$\left[
\begin{array}{ccc}
0 & 1 & 0 \\
1 & 0 & 0 \\
1 & 0 & 0
\end{array}
\right]$ & 
$\begin{array}{rl}
1 & (1,1,1)\\
-1 & (-1,1,1)\\
0 & (0,0,1) 
\end{array}$\\
\hline 
C& 
\begin{tikzpicture}
 [scale=.15,auto=left, node distance=1.5cm, every node/.style={circle,draw}]
 \node[fill=white] (n1) at (4,0) {\small{1}};
  \node[fill=white] (n2) at (24,0) {\small{2}};
 \node[fill=white] (n3) at (14,9)  {\small{3}};
 \draw[->, thick] (n1) edge[loop above] (n1); 
 \draw[->, thick] (n1) edge  [bend left=-10] (n2); 
 \draw[->, thick] (n1) edge  [bend left=-10] (n3); 
 \end{tikzpicture} & 
 $\begin{array}{l}
\Delta_1 \\
\Delta_2 \\
\Delta_3 \\
\end{array}$ &
$\left[ 
\begin{array}{ccc}
1 & 0 & 0 \\
1 & 0 & 0 \\
1 & 0 & 0
\end{array}
\right]$ & 
$\begin{array}{rl}
1 & (1,1,1)\\
0 & (0,1,0)\\
0 & (0,0,1) 
\end{array}$
&
D & 
\begin{tikzpicture}
 [scale=.15,auto=left, node distance=1.5cm, every node/.style={circle,draw}]
 \node[fill=white] (n1) at (4,0) {\small{1}};
  \node[fill=white] (n2) at (24,0) {\small{2}};
 \node[fill=white] (n3) at (14,9)  {\small{3}};
 \draw[->, thick] (n1) edge[loop above] (n1); 
 \draw[->, thick] (n1) edge  [bend left=-10] (n2); 
 \draw[->, thick] (n2) edge  [bend left=-10] (n3); 
 \end{tikzpicture} &
 $\begin{array}{l}
\Delta_3 
\end{array}$ &
$\left[
\begin{array}{ccc}
1 & 0 & 0 \\
1 & 0 & 0 \\
0 & 1 & 0
\end{array}
\right]$ & 
$\begin{array}{rl}
1 & (1,1,1)\\
0^* & (0,0,1)  
\end{array}$\\
\hline 
\end{tabular}}
}
\end{center}
\caption{Three-cell connected networks with one asymmetric input, 
up to re-enumeration of the cells. Here $\psi =e^{i2\pi/3}$ and 
$\Delta_l = \{ x:\, x_j = x_k \mbox{ for } j,k \not= l\}$, where $1\leq l \leq 3$. 
Eigenvalues having algebraic multiplicity two and geometric multiplicity one are marked with a star $*$.} 
\label{tab:val13cell}
\end{table}

\subsection{Lattice structures}

We characterize now, for the three-cell networks with asymmetric inputs,  the lattice of synchrony subspaces.
Here, the bottom element is the full-synchrony subspace and the top element is the network phase space. Whenever there are synchrony subspaces with dimension two, they correspond to the middle elements. Each lattice element is labelled with the eigenvalues of the corresponding quotient network. So the bottom element is labelled with the valency eigenvalue $\upsilon$.

Consider a  connected network ${\mathcal{N}}$ with three-cells and asymmetric inputs, a generic coupled cell system $f^{\mathcal{N}}$ and the corresponding Jacobian $J^{\mathcal{N}}_f$ (at the origin). Denote by  $m_a(\upsilon)$ the algebraic multiplicity of the eigenvalue  $\upsilon$  corresponding to  the constant row-sum or valency of $J^{\mathcal{N}}_f$.  
It follows from  \cite[Proposition 5.6]{ADS19} that $\upsilon$ is a semisimple eigenvalue and $m_a(\upsilon)$ is equal to the number of source components in the network. 
Thus $m_a(\upsilon)<3$, otherwise the network has three source components and it is disconnected.
We have then that there exists at least one more eigenvalue which we denote by $\mu$. 
In the next result, we refer to the eigenvalues of $J^{\mathcal{N}}_f$ for generic $f$ as $\upsilon, \mu$ in case there are only two distinct eigenvalues, and $\upsilon, \mu_1, \mu_2$ in case there are three distinct eigenvalues.

\begin{rem}\label{rem:exisysub}
Let ${\mathcal{N}}$ be a network with $n$ cells and $\mu$ a semisimple eigenvalue such that 
$\R^n = \Delta \oplus E_{\mu}$ 
for some network synchrony subspace $\Delta$.
Given a polydiagonal subspace $\Delta'$ containing $\Delta$, i.e., $\Delta\subset \Delta'$, 
trivially,
there are $v_1,\dots,v_m\in E_{\mu}$ such that we have $\Delta'= \Delta \oplus <v_1,\dots,v_m>$. Hence, 
as $\Delta'$ is invariant for any $J_f^{\mathcal{N}}$ and it is a polydiagonal, we conclude that $\Delta'$  is also
a network synchrony subspace.
\hfill $\Diamond$
\end{rem}

\begin{thm}\label{thm:lala}
The possible annotated synchrony lattice structures for connected three-cell networks  with asymmetric inputs are presented in Figure~\ref{Parte1}.
\end{thm}

{\begin{figure}
\begin{subfigure}{.4\textwidth}
\center
\begin{tikzpicture}[every node/.style={scale=0.7},node distance=0.4cm, baseline=(current  bounding  box.center)]
\node (n1) [ellipse,draw]  {\phantom{00}$\upsilon$\phantom{00}};
\node (n2) [ellipse,draw]  [above =of n1] {\phantom{0}$\upsilon, \mu$\phantom{0}};
\node (n5) [ellipse,draw]  [above=of n2] {$\upsilon, \upsilon,\mu$};

\draw (n1) to (n2);
\draw (n2) to (n5);
\end{tikzpicture}
\caption{$m_a(\upsilon)=2$ and $m_a(\mu) = 1$ }\label{lat:C2L1v}
\end{subfigure}
\begin{subfigure}{.4\textwidth}
\center
\begin{tikzpicture}[every node/.style={scale=0.7},node distance=0.4cm, baseline=(current  bounding  box.center)]
\node (n1) [ellipse,draw]  {\phantom{00}$\upsilon$\phantom{00}};
\node (n2) [ellipse,draw]  [above =of n1] {\phantom{0}$\upsilon, \mu$\phantom{0}};
\node (n3) [ellipse,draw]  [right =of n2] {\phantom{0}$\upsilon, \mu$\phantom{0}};
\node (n4) [ellipse,draw]  [left=of n2] {\phantom{0}$\upsilon, \mu$\phantom{0}};
\node (n5) [ellipse,draw]  [above=of n2] {$\upsilon, \mu,\mu$};

\draw (n1) to (n2);
\draw (n1) to  (n3);
\draw (n1) to  (n4);
\draw (n2) to  (n5);
\draw (n3) to  (n5);
\draw (n4) to  (n5);
\end{tikzpicture}
\caption{$m_a(\upsilon)=1$ and $m_g(\mu)=2$}\label{lat:C2L3s}
\end{subfigure}
\rule{0.8\textwidth}{0.4pt}
\begin{subfigure}{.4\textwidth}
\center
\begin{tikzpicture}[every node/.style={scale=0.7},node distance=0.4cm, baseline=(current  bounding  box.center)]
\node (n1) [ellipse,draw]  {\phantom{00}$\upsilon$\phantom{00}};
\node (n2) [ellipse,draw]  [above =of n1] {\phantom{0}$\upsilon, \mu$\phantom{0}};
\node (n5) [ellipse,draw]  [above=of n2] {\phantom{0}$\upsilon, \mu^*$\!\phantom{0}};

\draw (n1) to (n2);
\draw (n2) to  (n5);
\end{tikzpicture}
\caption{ } \label{lat:C2L1d}
\end{subfigure}
\begin{subfigure}{.4\textwidth}
\center
\begin{tikzpicture}[every node/.style={scale=0.7},node distance=0.4cm, baseline=(current  bounding  box.center)]
\node (n1) [ellipse,draw]  {\phantom{00}$\upsilon$\phantom{00}};
\node (n) [ellipse]  [above =of n1] {\phantom{00000}};
\node (n2) [ellipse,draw]  [above =of n] {$\upsilon, \mu^*$};
\draw (n1) to (n2);
\end{tikzpicture} 
\caption{ } \label{lat:C2L0d}
\end{subfigure}
\caption*{$m_a(\upsilon)=1$, $m_a(\mu) = 2$ and $m_g(\mu) =1$}
\rule{0.8\textwidth}{0.4pt}
\begin{subfigure}{.3\textwidth}
\center
\begin{tikzpicture}[every node/.style={scale=0.7},node distance=0.4cm, baseline=(current  bounding  box.center)]
\node (n1) [ellipse,draw]  {\phantom{00}$\upsilon$\phantom{00}};
\node (n) [ellipse]  [above =of n1] {\phantom{00000}};
\node (n3) [ellipse,draw]  [above=of n1, xshift=-2cm] {\phantom{0}$\upsilon, \mu_1$\phantom{0}};
\node (n4) [ellipse,draw]  [above=of n1, xshift=2cm] {\phantom{0}$\upsilon, \mu_2$\phantom{0}};
\node (n5) [ellipse,draw]  [above=of n,] {$\upsilon, \mu_1,\mu_2$};

\draw (n1) to  (n3);
\draw (n1) to  (n4);
\draw (n3) to  (n5);
\draw (n4) to  (n5);
\end{tikzpicture}
\caption{   }\label{lat:C3L2}
\end{subfigure}
\begin{subfigure}{.3\textwidth}
\center
\begin{tikzpicture}[every node/.style={scale=0.7},node distance=0.4cm, baseline=(current  bounding  box.center)]
\node (n5) [ellipse,draw]   {$\upsilon, \mu_1,\mu_2$};
\node (n2) [ellipse,draw]  [below =of n5] {\phantom{0}$\upsilon, \mu_1$\phantom{0}};
\node (n1) [ellipse,draw]  [below=of n2] {\phantom{00}$\upsilon$\phantom{00}};
\draw (n1) to (n2);
\draw (n2) to  (n5);
\end{tikzpicture}
\caption{   } \label{lat:C3L1}
\end{subfigure}
\begin{subfigure}{.3\textwidth}
\center
\begin{tikzpicture}[every node/.style={scale=0.7},node distance=0.4cm, baseline=(current  bounding  box.center)]
\node (n1) [ellipse,draw]  {\phantom{00}$\upsilon$\phantom{00}};
\node (n) [ellipse]  [above =of n1] {\phantom{00000}};
\node (n2) [ellipse,draw]  [above =of n] {$\upsilon, \mu_1,\mu_2$};
\draw (n1) to (n2);
\end{tikzpicture}
\caption{   }\label{lat:C3L0}
\end{subfigure}
\caption*{$m_a(\upsilon)=1$ and $m_a(\mu_1)=m_a(\mu_2)=1$}

\caption{Annotated synchrony lattice structures for connected three-cell networks with asymmetric inputs where $\upsilon$ denotes the valency eigenvalue and $\mu,\mu_1, \mu_2$ denote the other network eigenvalues. The defective eigenvalues are marked as $\mu^*$.} \label{Parte1}

\end{figure}}

\begin{proof}
Consider a connected network ${\mathcal{N}}$ with three-cells, a generic coupled cell system $f^{\mathcal{N}}$ and the corresponding Jacobian $J^{\mathcal{N}}_f$ (at the origin). \\
(i) Suppose that $m_a(\upsilon)=2$. Then the eigenvalue $\mu$ has multiplicity one and the two source components of the network have one cell each.
Otherwise, the network will be disconnected. Note that if $i,j$ are the cells of the two source components then the polydiagonal defined by the equality $x_i = x_j$  is a two-dimensional network synchrony space. Moreover, the two-cell quotient associated to that  synchrony subspace with dimension two has the eigenvalues $\upsilon$ and $\mu$, because it is connected. If there were more than one synchrony subspaces with dimension two, then there were at least two linear independent eigenvectors of $\mu$. But, that is not possible, since the algebraic multiplicity of the eigenvalue $\mu$ is one. Therefore, there is only one two-dimensional synchrony subspace and  the lattice structure must be the one in Figure~\ref{lat:C2L1v}. \\
(ii) Suppose that $\upsilon$ and $\mu$ are the unique eigenvalues of $J^{\mathcal{N}}_f$, where $m_a(\upsilon)=1$ and $\mu$ has geometric multiplicity two.
Then $\R^3=\Delta_0 \oplus E_{\mu}$ and by Remark~\ref{rem:exisysub}, every polydiagonal subspace of $\R^3$ is a synchrony subspace of ${\mathcal{N}}$.  
Thus the lattice structure must be the one in Figure~\ref{lat:C2L3s}. \\
(iii) Suppose that $\upsilon$ and $\mu$ are the unique eigenvalues,  $m_a(\upsilon)=1$  and $\mu$ has geometric multiplicity one.
Again, we know that any synchrony subspace with dimension two has the eigenvalues $\upsilon$ and $\mu$. If there were two or more synchrony subspace with dimension two, we would obtain at least two linear independent eigenvectors associated with $\mu$, a contradiction.  
Therefore, the lattice structure must be one of the following two in Figures~\ref{lat:C2L1d}-\ref{lat:C2L0d}, where the defective eigenvalue is marked with a star $^*$. \\
(iv) Suppose that the network has three distinct eigenvalues,  $\upsilon$, $\mu_1$ and $\mu_2$, with multiplicity one. 
The quotient network associated to any two dimensional synchrony subspace has the eigenvalue $\upsilon$ and $\mu_i$ for some $i=1,2$.
Since the eigenvalues $\mu_1$ and $\mu_2$ have multiplicity one, two-dimensional synchrony subspaces can not have the same eigenvalues.
Thus we have three cases depending on the number of two-dimensional synchrony subspaces, see Figures~\ref{lat:C3L2}-\ref{lat:C3L0}. 
\end{proof}

\subsection{Bifurcation diagrams}

The results obtained in Section~\ref{sec:ssbresults}  combined with the possible synchrony lattice structures described in Theorem~\ref{thm:lala} above 
are now applied to connected networks with three-cells and asymmetric inputs. 
Concretely, we take each of the synchrony lattice structures and prove which network synchrony subspaces support a bifurcation branch of steady-state solutions when a network eigenvalue crosses zero. 
This information is collected into a bifurcation diagram. 
Remarkably, we get that each of the synchrony lattice structures presented in Theorem~\ref{thm:lala} has a distinct bifurcation diagram. 

\begin{thm}\label{thm:bif_diag} Let ${\mathcal{N}}$ be a connected three-cell network  with asymmetric inputs and take its associated annotated synchrony lattice structure $L$ which has to be one of the seven lattice structures in Figure~\ref{Parte1}. 
Then the structure of the bifurcation diagram of ${\mathcal{N}}$ is the one indicated in Figure~\ref{bd} for the lattice structure $L$.
\end{thm}

\begin{figure}
\begin{subfigure}{.4\textwidth}
\center 
\begin{tikzpicture}
 \draw (0,0) --   (3.5,0);

\draw (1,-0.05) -- (1,0.05); 
\draw (2.5,-0.05) -- (2.5,0.05); 

\node (bc1) at (1,-0.5) {$\upsilon=0$};
\node (bc2) at (2.5,-0.5) {$\mu=0$};

\node (bs1) at (1.8,1) {$\Delta_0$};
\draw[domain=0:1, smooth, variable=\y, black] plot ({1+\y*\y}, {0.8*\y});

\node (bs12) at (1.8,2) {$\R^3$};
\draw[domain=0:1, smooth, variable=\y, black] plot ({1+\y*\y*\y}, {1.8*\y});

\node (a) at (3.3,3) {$$};

\node (bs2) at (3.3,1) {$\Delta_1$};
\draw[domain=0:1, smooth, variable=\y, black] plot ({2.5+\y*\y}, {0.8*\y});
\end{tikzpicture}
\caption{ For networks with  lattice structure Figure~\ref{lat:C2L1v}.}\label{lat:bdC2L1v}
\end{subfigure}
\begin{subfigure}{.4\textwidth}
\center 
\begin{tikzpicture}
\draw (0,0) --   (3.5,0);

\draw (1,-0.05) -- (1,0.05); 
\draw (2.5,-0.05) -- (2.5,0.05); 

\node (bc1) at (1,-0.5) {$\upsilon=0$};

\node (bs1) at (1.8,1) {$\Delta_0$};
\draw[domain=0:1, smooth, variable=\y, black] plot ({1+\y*\y}, {0.8*\y});

\node (bc2) at (2.5,-0.5) {$\mu=0$};

\node (bs2) at (3.3,1) {$\Delta_1$};
\draw[domain=0:1, smooth, variable=\y, black] plot ({2.5+\y*\y*\y*\y}, {0.8*\y});

\node (bs12) at (3.3,2) {$\Delta_2$};
\draw[domain=0:1, smooth, variable=\y, black] plot ({2.5+\y*\y*\y}, {1.8*\y});

\node (bs12) at (3.3,3) {$\Delta_3$};
\draw[domain=0:1, smooth, variable=\y, black] plot ({2.5+\y*\y*\y*\y}, {2.8*\y});
\end{tikzpicture}
\caption{For networks with lattice structure Figure~\ref{lat:C2L3s} (if Assumptions~\ref{asps:LSred2} hold).}\label{lat:bdC2L3s}
\end{subfigure}

\begin{subfigure}{.4\textwidth}
\center
\begin{tikzpicture}
\draw (0,0) --   (3.5,0);

\draw (1,-0.05) -- (1,0.05); 
\draw (2.5,-0.05) -- (2.5,0.05); 

\node (bc1) at (1,-0.5) {$\upsilon=0$};

\node (bs1) at (1.8,1) {$\Delta_0$};
\draw[domain=0:1, smooth, variable=\y, black] plot ({1+\y*\y}, {0.8*\y});

\node (bc2) at (2.5,-0.5) {$\mu=0$};

\node (bs2) at (3.3,1) {$\Delta_1$};
\draw[domain=0:1, smooth, variable=\y, black] plot ({2.5+\y*\y*\y*\y}, {0.8*\y});

\node (bs12) at (3.3,2) {$\R^3$};
\draw[domain=0:1, smooth, variable=\y, black] plot ({2.5+\y*\y*\y}, {1.8*\y});
\end{tikzpicture}
\caption{For networks with lattice structure Figure~\ref{lat:C2L1d} (if condition~(\ref{eq:2detdef}) holds).} \label{lat:bdC2L1d}
\end{subfigure}
\begin{subfigure}{.4\textwidth}
\center
\begin{tikzpicture}
\draw (0,0) --   (3.5,0);

\draw (1,-0.05) -- (1,0.05); 
\draw (2.5,-0.05) -- (2.5,0.05); 

\node (bc1) at (1,-0.5) {$\upsilon=0$};

\node (bs1) at (1.8,1) {$\Delta_0$};
\draw[domain=0:1, smooth, variable=\y, black] plot ({1+\y*\y}, {0.8*\y});

\node (bc2) at (2.5,-0.5) {$\mu=0$};
\node (a) at (3.3,2) {$$};

\node (bs2) at (3.3,1) {$\R^3$};
\draw[domain=0:1, smooth, variable=\y, black] plot ({2.5+\y*\y}, {0.8*\y});
\end{tikzpicture}
\caption{For networks with lattice structure Figure~\ref{lat:C2L0d}. }\label{lat:bdC2L0d}
\end{subfigure}

\begin{subfigure}{.4\textwidth}
\center
\begin{tikzpicture}
 \draw (0,0) --   (5,0);

\draw (1,-0.05) -- (1,0.05); 
\draw (2.5,-0.05) -- (2.5,0.05); 
\draw (4,-0.05) -- (4,0.05); 

\node (bc1) at (1,-0.5) {$\upsilon=0$};
\node (bc2) at (2.5,-0.5) {$\mu_1=0$};
\node (bc3) at (4,-0.5) {$\mu_2=0$};

\node (bs1) at (1.8,1) {$\Delta_0$};
\node (bs2) at (3.3,1) {$\Delta_1$};
\node (bs3) at (4.8,1) {$\Delta_2$};

\node (f) at (4.8,1.5) {$$};
\draw[domain=0:1, smooth, variable=\y, black] plot ({1+\y*\y}, {0.8*\y});
\draw[domain=0:1, smooth, variable=\y, black] plot ({2.5+\y*\y}, {0.8*\y});
\draw[domain=0:1, smooth, variable=\y, black] plot ({4+\y*\y}, {0.8*\y});
\end{tikzpicture}
\caption{ For networks with lattice structure Figure~\ref{lat:C3L2}.}\label{lat:bdC3L2}
\end{subfigure}
\begin{subfigure}{.4\textwidth}
\center
\begin{tikzpicture}
 \draw (0,0) --   (5,0);

\draw (1,-0.05) -- (1,0.05); 
\draw (2.5,-0.05) -- (2.5,0.05); 
\draw (4,-0.05) -- (4,0.05); 

\node (bc1) at (1,-0.5) {$\upsilon=0$};
\node (bc2) at (2.5,-0.5) {$\mu_1=0$};
\node (bc3) at (4,-0.5) {$\mu_2=0$};

\node (bs1) at (1.8,1) {$\Delta_0$};
\node (bs2) at (3.3,1) {$\Delta_1$};
\node (bs3) at (4.8,1) {$\R^3$};
\node (f) at (4.8,1.5) {$$};
\draw[domain=0:1, smooth, variable=\y, black] plot ({1+\y*\y}, {0.8*\y});
\draw[domain=0:1, smooth, variable=\y, black] plot ({2.5+\y*\y}, {0.8*\y});
\draw[domain=0:1, smooth, variable=\y, black] plot ({4+\y*\y}, {0.8*\y});
 
\end{tikzpicture}
\caption{ For networks with lattice structure  Figure~\ref{lat:C3L1}.} \label{lat:bdC3L1}
\end{subfigure}

\begin{subfigure}{.4\textwidth}
\center
\begin{tikzpicture}
 \draw (0,0) --   (2,0);

\draw (1,-0.05) -- (1,0.05); 

\node (bc1) at (1,-0.5) {\small$\upsilon=0$};

\node (bs1) at (1.8,1) {$\Delta_0$};

\node (f) at (1.8,1.5) {$$};
\draw[domain=0:1, smooth, variable=\y, black] plot ({1+\y*\y}, {0.8*\y}); 
\end{tikzpicture}
\caption{ For networks with lattice structure Figure~\ref{lat:C3L0} when $\mu_1=\overline{\mu_2}, \mu_2$ are nonreal.}\label{lat:bdC3L0a}
\end{subfigure}
\begin{subfigure}{.4\textwidth}
\center
\begin{tikzpicture}
 \draw (0,0) --   (5,0);

\draw (1,-0.05) -- (1,0.05); 
\draw (2.5,-0.05) -- (2.5,0.05); 
\draw (4,-0.05) -- (4,0.05); 

\node (bc1) at (1,-0.5) {\small$\upsilon=0$};
\node (bc2) at (2.5,-0.5) {\small$\mu_1=0$};
\node (bc3) at (4,-0.5) {\small$\mu_2=0$};

\node (bs1) at (1.8,1) {$\Delta_0$};
\node (bs2) at (3.3,1) {$\R^3$};
\node (bs3) at (4.8,1) {$\R^3$};
\node (f) at (4.8,1.5) {$$};
\draw[domain=0:1, smooth, variable=\y, black] plot ({1+\y*\y}, {0.8*\y});
\draw[domain=0:1, smooth, variable=\y, black] plot ({2.5+\y*\y}, {0.8*\y});
\draw[domain=0:1, smooth, variable=\y, black] plot ({4+\y*\y}, {0.8*\y});
 
\end{tikzpicture}
\caption{ For networks with lattice structure Figure~\ref{lat:C3L0} when $\mu_1,\mu_2$ are real and distinct.} \label{lat:bdC3L0b}
\end{subfigure}

\caption{Bifurcation diagrams displaying the synchrony of bifurcation branches of steady-state solutions emerging from bifurcation problems with the mention bifurcation condition. Here, $\Delta_0$ denotes the full-synchrony subspace and $\Delta_1$, $\Delta_2$, $\Delta_3$ denote two-dimensional synchrony subspaces. Also, $\upsilon$ is the valency eigenvalue, and $\mu$, $\mu_1$, $\mu_2$ are other network eigenvalues. It does not display the stability, growth-rate nor the number of branches.}
\label{bd}
\end{figure}

\begin{proof}
We start by noticing that  the full synchrony subspace $\Delta_0$ is $\upsilon$-simple and $\upsilon$-maximal and any two-dimensional synchrony subspace in Figure~\ref{Parte1} is also $\mu$-simple and $\mu$-maximal. Applying Theorem~\ref{thm:LSred1}, we conclude that the full synchrony subspace $\Delta_0$ and any two-dimensional synchrony subspace support a bifurcation branch for a bifurcation problem given by the condition $\upsilon=0$ and $\mu=0$, respectively.

If a synchrony space $\Delta$ is $\mu$-simple and $\mu$-submaximal, then there is no bifurcation branch with synchrony $\Delta$ for any generic $f\in\mathcal{V}_{\mu}({\mathcal{N}})$, as the dimension of the center subspace does not increase from the $\mu$-maximal subspace. \\
(i) For the lattice in Figure~\ref{lat:C2L1v}, we see that there is one two-dimensional synchrony space and that the space $\R^3$ is valency synchrony-breaking. 
By Proposition~\ref{prop:valsynbre}, the network phase space
$\R^3$ supports a bifurcation branch for bifurcation problems given by condition $\upsilon=0$. 
Thus bifurcation problems given by the valency have a bifurcation branch with synchrony $\Delta_0$ and another with synchrony $\R^3$.
In \cite{ADS19}, it is proven that there are two bifurcation branches with synchrony $\R^3$.
However we are only interested in the synchrony of the branches and we only draw one branch with synchrony $\R^3$ in the diagram.
For the second condition $\mu=0$, we know that the two-dimensional synchrony space supports a bifurcation branch for a bifurcation problem given by that condition.
This means that there are two and one types of synchrony emerging for bifurcation problems given by the condition $\upsilon=0$ and $\mu=0$, respectively.
Hence the diagram bifurcation is given in Figure~\ref{lat:bdC2L1v}.\\
(ii) For the lattice structure in Figure~\ref{lat:C2L3s}, there are three two-dimensional synchrony spaces, and the space $\R^3$ is $\mu$-semisimple with multiplicity $2$ and $\mu$-submaximal with order $3$.
Note that $2^2-1=3$ and 
so Theorem~\ref{thm:LSred2} can be applied to $\R^3$.
If the  Assumptions~\ref{asps:LSred2} hold, we conclude that $\R^3$ does not support a bifurcation branch. 
Since the two-dimensional synchrony spaces support a bifurcation branch, we have bifurcation branches with synchronies $\Delta_1$, $\Delta_2$ and $\Delta_3$ for bifurcation problems given by the condition $\mu=0$.
Moreover, a bifurcation problem given by the valency, that is, $\upsilon=0$, has a bifurcation branch with synchrony $\Delta_0$.
Thus,  there are three (resp. one) types of synchrony branches of steady-state solutions emerging when the bifurcation condition imposed is $\mu=0$ (resp. $\upsilon=0$)  and the bifurcation diagram  is given in Figure~\ref{lat:bdC2L3s}.
\\
(iii) Taking now the lattice structure in Figure~\ref{lat:C2L1d}, there is one two-dimensional synchrony space, and the space $\R^3$ is $\mu$-defective with multiplicity $(1,2)$ and $\mu$-submaximal with order $1$. Assuming that condition~(\ref{eq:2detdef}) holds, from Theorem~\ref{thm:LSred3} it follows that 
 $\R^3$ supports a bifurcation branch of steady-state solutions under the bifurcation condition $\mu=0$.
As the two-dimensional synchrony spaces also support a bifurcation branch of steady-state solutions under the bifurcation condition $\mu=0$, we have that 
there are two types of synchrony branches of steady-state solutions emerging when $\mu=0$. Under the valency bifurcation condition $\upsilon=0$, there is a bifurcation branch of steady-state solutions with synchrony $\Delta_0$ and the diagram bifurcation is so given in Figure~\ref{lat:bdC2L1d}.
\\ 
(iv) We consider now the synchrony lattice structure given in  Figure~\ref{lat:C2L0d}, where 
there is no two-dimensional synchrony space, and the space $\R^3$ is $\mu$-defective with multiplicity $(1,2)$ and $\mu$-maximal.
It follows from Remark~\ref{latfigura60} that the space $\R^3$ supports a bifurcation branch of steady-state solutions under the bifurcation condition $\mu=0$.
Moreover,  $\Delta_0$ supports a bifurcation branch of steady-state solutions under the bifurcation condition  $\upsilon=0$. See the 
bifurcation diagram in Figure~\ref{lat:bdC2L0d}.\\
(v) In the synchrony lattice structure of Figure~\ref{lat:C3L2}, there are two two-dimensional synchrony spaces, $\Delta_1$ which is $\mu_1$-simple and $\mu_1$-maximal and $\Delta_2$ which is $\mu_2$-simple and $\mu_2$-maximal and the space $\R^3$ is simple and submaximal. Thus $\Delta_0$, $\Delta_1$ and $\Delta_2$ support a bifurcation branch for bifurcation problems given by the conditions $\upsilon=0$, $\mu_1=0$ and $\mu_2=0$, respectively. 
Moreover, the space $\R^3$ does not support a bifurcation branch. We obtain the bifurcation diagram in Figure~\ref{lat:bdC3L2}. 
 \\
(vi) In the synchrony lattice structure of Figure~\ref{lat:C3L1}, the two-dimensional synchrony space $\Delta_1$ is $\mu_1$-simple and $\mu_1$-maximal 
and the space $\R^3$ is $\mu_2$-simple and $\mu_2$-maximal. Hence, from Theorem~\ref{thm:LSred1} applied to the synchrony spaces $\Delta_0$, $\Delta_1$ and $\R^3$, 
each synchrony space supports a bifurcation branch of steady-state solutions under the corresponding bifurcation condition and the bifurcation diagram is given in 
Figure~\ref{lat:bdC3L1}. \\
(vii) Last, we study the annotated lattice given in Figure~\ref{lat:C3L0}.
In this case, the network eigenvalues can be complex conjugated and this leads to two different bifurcation diagrams.
If the network eigenvalues $\mu_1$ and $\mu_2$ are conjugate complex numbers, $\mu_1=\overline{\mu_2}$, then only the full synchrony subspace $\Delta_0$ supports a bifurcation branch of steady-state solutions under bifurcation condition $\upsilon=0$ and we get the bifurcation diagram in Figure~\ref{lat:bdC3L0a}.
If the $\mu_1$ and $\mu_2$ are real and distinct, then the space $\R^3$ is $\mu_1$-simple and $\mu_1$-maximal and it is also $\mu_2$-simple and $\mu_2$-maximal.
Now, we can apply Theorem~\ref{thm:LSred1} to $\R^3$ for the two bifurcation conditions $\mu_1=0$ and $\mu_2=0$.
Thus the space $\R^3$ supports a bifurcation branch for bifurcation problems given by $\mu_1=0$ and $\mu_2=0$ and the diagram bifurcation is given in 
Figure~\ref{lat:bdC3L0b}.
\end{proof}

\section{Case study: three-cell networks with one, two or six asymmetric inputs}\label{subsec_case_study}

The minimal connected three-cell networks with one, two and six asymmetric inputs are enumerated, up to ODE-equivalence, in \cite{ADS20}.
Moreover, it is proved that any three-cell network with $k$-asymmetric inputs , $k \in \NN$, is ODE-equivalent to a minimal three-cell network with at most six asymmetric inputs. 
In this section, for each such  three-cell connected network, we study which synchrony subspaces support a synchrony-breaking bifurcation branch of steady-state 
solutions. 

We start by recalling the classification given in \cite{ADS20} of the minimal connected three-cell networks with one, two and six asymmetric inputs. 

\begin{thm}[\cite{ADS20}]\label{thm:classification}
There are exactly:\\
(i) Four ODE-classes of connected three-cell networks with one asymmetric input, see the minimal representatives  in Table~\ref{tab:val13cell}.\\
(ii) Forty eight ODE-classes of connected three-cell networks with two asymmetric inputs, see the minimal representatives in Tables~\ref{tab:C3L0RI1.tex}-\ref{tab:C2L3.tex}. The networks in Tables~\ref{tab:C3L0RI1.tex}-\ref{tab:C3L0RI2.tex} have no  two-dimensional synchrony subspaces; the networks in Tables~\ref{tab:C3L1.tex}-\ref{tab:C3L12.tex} have exactly one two-dimensional synchrony subspace; the networks in Table~\ref{tab:C3L2.tex} have exactly two two-dimensional synchrony subspaces; the network in Table~\ref{tab:C2L3.tex} has three two-dimensional synchrony subspaces.\\
(iii) One ODE-class of connected three-cell networks with six asymmetric inputs, with minimal representative in Figure~\ref{fig:rep_min} and it has no two-dimensional synchrony spaces. 
\end{thm}

Table~\ref{tab:val13cell} includes the two-dimensional synchrony subspaces of each connected three-cell network with one asymmetric input in Table~\ref{tab:val13cell}. 

\begin{table}
\resizebox{1 \textwidth}{!}{ 
    {\tiny 
 \begin{tabular}{|c|c|c|c|}
  \hline

 $E_6 \& B_1$ 
\begin{tikzpicture}
  [scale=.15,auto=left, node distance=1.5cm, every node/.style={circle,draw}]
 \node[fill=white] (n1) at (4,0) {\small{1}};
 \node[fill=white] (n2) at (24,0) {\small{2}}; \node[fill=white] (n3) at (14,9)  {\small{3}};
\draw[->, thick] (n1) to  [in=120,out=70,looseness=5]  (n1);
\draw[->>, thick] (n1) to  [in=205,out=155,looseness=5] (n1);
 \draw[->, thick] (n3) to  [in=110,out=160,looseness=5]  (n3);
\draw[->, thick] (n1) edge  [bend left=-10] (n2); 
 \draw[<<-, thick] (n2) edge  [bend left=10] (n3); 
\draw[->>, thick] (n2) edge  [bend left=10] (n3); 
\end{tikzpicture}
 &
$B_1 \& F_2$ 
\begin{tikzpicture}
  [scale=.15,auto=left, node distance=1.5cm, every node/.style={circle,draw}]
 \node[fill=white] (n1) at (4,0) {\small{1}};
 \node[fill=white] (n2) at (24,0) {\small{2}}; \node[fill=white] (n3) at (14,9)  {\small{3}};
 \draw[->, thick] (n1) to  [in=120,out=70,looseness=5] (n1);
\draw[<<-, thick] (n1) edge  [bend left=10] (n2); 
\draw[->>, thick] (n1) edge  [bend left=10] (n2); 
 \draw[<-, thick] (n2) edge  [bend left=-10] (n3); 
 \draw[->, thick] (n2) edge  [bend left=-10] (n3); 
\draw[->>, thick] (n2) edge  [bend left=10] (n3); 
\end{tikzpicture}
&
 $C_1 \& A_2$ 
 \begin{tikzpicture}
  [scale=.15,auto=left, node distance=1.5cm, every node/.style={circle,draw}]
 \node[fill=white] (n1) at (4,0) {\small{1}};
 \node[fill=white] (n2) at (24,0) {\small{2}}; \node[fill=white] (n3) at (14,9)  {\small{3}};
 \draw[->, thick] (n1) to  [in=120,out=70,looseness=5] (n1);
\draw[<<-, thick] (n1) edge  [bend left=10] (n2); 
\draw[->, thick] (n1) edge  [bend left=-10] (n2); 
 \draw[<<-, thick] (n2) edge  [bend left=10] (n3); 
\draw[<-, thick] (n3) edge [bend right=10] (n1); 
\draw[<<-, thick] (n3) edge [bend right=-10] (n1); 
\end{tikzpicture}
 & 
$A_2 \& A_1$ 
\begin{tikzpicture}
  [scale=.15,auto=left, node distance=1.5cm, every node/.style={circle,draw}]
 \node[fill=white] (n1) at (4,0) {\small{1}};
 \node[fill=white] (n2) at (24,0) {\small{2}}; \node[fill=white] (n3) at (14,9)  {\small{3}};
\draw[<-, thick] (n1) edge  [bend left=-10] (n2); 
\draw[->>, thick] (n3) edge [bend right=-10] (n1); 
\draw[->>, thick] (n1) edge  [bend left=10] (n2); 
 \draw[<-, thick] (n2) edge  [bend left=-10] (n3); 
\draw[<-, thick] (n3) edge [bend right=10] (n1); 
\draw[->>, thick] (n2) edge  [bend left=10] (n3); 
\end{tikzpicture}
\\ \hline
 $B_1 \& B_3$ 
 \begin{tikzpicture}
  [scale=.15,auto=left, node distance=1.5cm, every node/.style={circle,draw}]
 \node[fill=white] (n1) at (4,0) {\small{1}};
 \node[fill=white] (n2) at (24,0) {\small{2}}; \node[fill=white] (n3) at (14,9)  {\small{3}};
 \draw[->, thick] (n1) to  [in=120,out=70,looseness=5] (n1);
 \draw[->>, thick] (n3) to  [in=70,out=20,looseness=5] (n3);
\draw[<<-, thick] (n1) edge  [bend left=10] (n2); 
\draw[->>, thick] (n1) edge  [bend left=10] (n2); 
 \draw[<-, thick] (n2) edge  [bend left=-10] (n3); 
\draw[->, thick] (n2) edge  [bend left=-10] (n3); 
\end{tikzpicture}
 \\ \cline{1-1} 
\end{tabular}}
 }
 \caption{Minimal three-cell networks with two asymmetric inputs and no 2D synchrony subspaces.} 
 \label{tab:C3L0RI1.tex}
 \end{table}

\begin{table}
\resizebox{1 \textwidth}{!}{ 
    {\tiny 
 \begin{tabular}{|c|c|c|c|}
 \hline
 $D_1 \& F_3$ 
 \begin{tikzpicture}
  [scale=.15,auto=left, node distance=1.5cm, every node/.style={circle,draw}]
 \node[fill=white] (n1) at (4,0) {\small{1}};
 \node[fill=white] (n2) at (24,0) {\small{2}}; \node[fill=white] (n3) at (14,9)  {\small{3}};
 \draw[->, thick] (n1) to  [in=120,out=70,looseness=5] (n1);
\draw[<<-, thick] (n1) edge  [bend left=10] (n2); 
\draw[->, thick] (n1) edge  [bend left=-10] (n2); 
 \draw[<<-, thick] (n2) edge  [bend left=10] (n3); 
 \draw[->, thick] (n2) edge  [bend left=-10] (n3); 
\draw[->>, thick] (n2) edge  [bend left=10] (n3); 
\end{tikzpicture}
 &
 $D_1 \& D_6$ 
 \begin{tikzpicture}
  [scale=.15,auto=left, node distance=1.5cm, every node/.style={circle,draw}]
 \node[fill=white] (n1) at (4,0) {\small{1}};
 \node[fill=white] (n2) at (24,0) {\small{2}}; \node[fill=white] (n3) at (14,9)  {\small{3}};
 \draw[->, thick] (n1) to  [in=120,out=70,looseness=5] (n1);
 \draw[->>, thick] (n3) to  [in=70,out=20,looseness=5] (n3);
\draw[<<-, thick] (n1) edge  [bend left=10] (n2); 
\draw[->, thick] (n1) edge  [bend left=-10] (n2); 
 \draw[<<-, thick] (n2) edge  [bend left=10] (n3); 
\draw[->, thick] (n2) edge  [bend left=-10] (n3); 
\end{tikzpicture}
 &
 $D_1 \& F_6$ 
 \begin{tikzpicture}
  [scale=.15,auto=left, node distance=1.5cm, every node/.style={circle,draw}]
 \node[fill=white] (n1) at (4,0) {\small{1}};
 \node[fill=white] (n2) at (24,0) {\small{2}}; \node[fill=white] (n3) at (14,9)  {\small{3}};
 \draw[->, thick] (n1) to  [in=120,out=70,looseness=5] (n1);
\draw[->>, thick] (n3) edge [bend right=-10] (n1); 
\draw[->, thick] (n1) edge  [bend left=-10] (n2); 
\draw[->>, thick] (n1) edge  [bend left=10] (n2); 
\draw[<<-, thick] (n3) edge [bend right=-10] (n1); 
\draw[->, thick] (n2) edge  [bend left=-10] (n3); 
\end{tikzpicture}
&
  $D_1 \& A_1$ 
 \begin{tikzpicture}
  [scale=.15,auto=left, node distance=1.5cm, every node/.style={circle,draw}]
 \node[fill=white] (n1) at (4,0) {\small{1}};
 \node[fill=white] (n2) at (24,0) {\small{2}}; \node[fill=white] (n3) at (14,9)  {\small{3}};
 \draw[->, thick] (n1) to  [in=120,out=70,looseness=5] (n1);
\draw[->>, thick] (n3) edge [bend right=-10] (n1); 
\draw[->, thick] (n1) edge  [bend left=-10] (n2); 
\draw[->>, thick] (n1) edge  [bend left=10] (n2); 
 \draw[->, thick] (n2) edge  [bend left=-10] (n3); 
\draw[->>, thick] (n2) edge  [bend left=10] (n3); 
\end{tikzpicture}
 \\ \hline 
 $D_1 \& D_2$ 
 \begin{tikzpicture}
  [scale=.15,auto=left, node distance=1.5cm, every node/.style={circle,draw}]
 \node[fill=white] (n1) at (4,0) {\small{1}};
 \node[fill=white] (n2) at (24,0) {\small{2}}; \node[fill=white] (n3) at (14,9)  {\small{3}};
\draw[->, thick] (n1) to  [in=120,out=70,looseness=5]  (n1);
\draw[->>, thick] (n1) to  [in=205,out=155,looseness=5] (n1);
\draw[->, thick] (n1) edge  [bend left=-10] (n2); 
 \draw[<<-, thick] (n2) edge  [bend left=10] (n3); 
\draw[<<-, thick] (n3) edge [bend right=-10] (n1); 
\draw[->, thick] (n2) edge  [bend left=-10] (n3); 
\end{tikzpicture}
 &  
 $D_1 \& D_5$ 
 \begin{tikzpicture}
  [scale=.15,auto=left, node distance=1.5cm, every node/.style={circle,draw}]
 \node[fill=white] (n1) at (4,0) {\small{1}};
 \node[fill=white] (n2) at (24,0) {\small{2}}; \node[fill=white] (n3) at (14,9)  {\small{3}};
 \draw[->, thick] (n1) to  [in=120,out=70,looseness=5] (n1);
 \draw[->>, thick] (n3) to  [in=70,out=20,looseness=5] (n3);
\draw[->>, thick] (n3) edge [bend right=-10] (n1); 
\draw[->, thick] (n1) edge  [bend left=-10] (n2); 
\draw[->>, thick] (n1) edge  [bend left=10] (n2); 
\draw[->, thick] (n2) edge  [bend left=-10] (n3); 
\end{tikzpicture}
& 
$D_1 \& B_1$ 
\begin{tikzpicture}
  [scale=.15,auto=left, node distance=1.5cm, every node/.style={circle,draw}]
 \node[fill=white] (n1) at (4,0) {\small{1}};
 \node[fill=white] (n2) at (24,0) {\small{2}}; \node[fill=white] (n3) at (14,9)  {\small{3}};
\draw[->, thick] (n1) to  [in=120,out=70,looseness=5]  (n1);
\draw[->>, thick] (n1) to  [in=205,out=155,looseness=5] (n1);
\draw[->, thick] (n1) edge  [bend left=-10] (n2); 
 \draw[<<-, thick] (n2) edge  [bend left=10] (n3); 
 \draw[->, thick] (n2) edge  [bend left=-10] (n3); 
\draw[->>, thick] (n2) edge  [bend left=10] (n3); 
\end{tikzpicture}
&
 $D_1 \& B_2$ 
 \begin{tikzpicture}
  [scale=.15,auto=left, node distance=1.5cm, every node/.style={circle,draw}]
 \node[fill=white] (n1) at (4,0) {\small{1}};
 \node[fill=white] (n2) at (24,0) {\small{2}}; \node[fill=white] (n3) at (14,9)  {\small{3}};
 \draw[->, thick] (n1) to  [in=120,out=70,looseness=5] (n1);
 \draw[->>, thick] (n2) to  [in=70,out=120,looseness=5] (n2);
\draw[->>, thick] (n3) edge [bend right=-10] (n1); 
\draw[->, thick] (n1) edge  [bend left=-10] (n2); 
\draw[<<-, thick] (n3) edge [bend right=-10] (n1); 
\draw[->, thick] (n2) edge  [bend left=-10] (n3); 
\end{tikzpicture}
 \\ \hline   
$D_1 \& E_4$ 
\begin{tikzpicture}
  [scale=.15,auto=left, node distance=1.5cm, every node/.style={circle,draw}]
 \node[fill=white] (n1) at (4,0) {\small{1}};
 \node[fill=white] (n2) at (24,0) {\small{2}}; \node[fill=white] (n3) at (14,9)  {\small{3}};
\draw[->, thick] (n1) to  [in=120,out=70,looseness=5]  (n1);
\draw[->>, thick] (n1) to  [in=205,out=155,looseness=5] (n1);
 \draw[->>, thick] (n3) to  [in=70,out=20,looseness=5] (n3);
\draw[->, thick] (n1) edge  [bend left=-10] (n2); 
 \draw[<<-, thick] (n2) edge  [bend left=10] (n3); 
\draw[->, thick] (n2) edge  [bend left=-10] (n3); 
\end{tikzpicture}
& 
 $E_6 \& A_2$ 
 \begin{tikzpicture}
  [scale=.15,auto=left, node distance=1.5cm, every node/.style={circle,draw}]
 \node[fill=white] (n1) at (4,0) {\small{1}};
 \node[fill=white] (n2) at (24,0) {\small{2}}; \node[fill=white] (n3) at (14,9)  {\small{3}};
 \draw[->, thick] (n1) to  [in=120,out=70,looseness=5] (n1);
 \draw[->, thick] (n3) to  [in=110,out=160,looseness=5]  (n3);
\draw[<<-, thick] (n1) edge  [bend left=10] (n2); 
\draw[->, thick] (n1) edge  [bend left=-10] (n2); 
 \draw[<<-, thick] (n2) edge  [bend left=10] (n3); 
\draw[<<-, thick] (n3) edge [bend right=-10] (n1); 
\end{tikzpicture}
&
$B_1 \& A_2$ 
\begin{tikzpicture}
  [scale=.15,auto=left, node distance=1.5cm, every node/.style={circle,draw}]
 \node[fill=white] (n1) at (4,0) {\small{1}};
 \node[fill=white] (n2) at (24,0) {\small{2}}; \node[fill=white] (n3) at (14,9)  {\small{3}};
 \draw[->, thick] (n1) to  [in=120,out=70,looseness=5] (n1);
\draw[<<-, thick] (n1) edge  [bend left=10] (n2); 
\draw[<-, thick] (n2) edge  [bend left=-10] (n3); 
\draw[<<-, thick] (n2) edge  [bend left=10] (n3); 
\draw[<<-, thick] (n3) edge [bend right=-10] (n1); 
\draw[->, thick] (n2) edge  [bend left=-10] (n3); 
\end{tikzpicture}
&
$F_1 \& A_2$ 
\begin{tikzpicture}
  [scale=.15,auto=left, node distance=1.5cm, every node/.style={circle,draw}]
 \node[fill=white] (n1) at (4,0) {\small{1}};
 \node[fill=white] (n2) at (24,0) {\small{2}}; \node[fill=white] (n3) at (14,9)  {\small{3}};
\draw[<-, thick] (n1) edge  [bend left=-10] (n2); 
 \draw[<<-, thick] (n1) edge  [bend left=10] (n2); 
\draw[->, thick] (n1) edge  [bend left=-10] (n2); 
 \draw[<<-, thick] (n2) edge  [bend left=10] (n3); 
\draw[<-, thick] (n3) edge [bend right=10] (n1); 
\draw[<<-, thick] (n3) edge [bend right=-10] (n1); 
\end{tikzpicture}
\\ \hline 
$D_1 \& A_2$ 
\begin{tikzpicture}
  [scale=.15,auto=left, node distance=1.5cm, every node/.style={circle,draw}]
 \node[fill=white] (n1) at (4,0) {\small{1}};
 \node[fill=white] (n2) at (24,0) {\small{2}}; \node[fill=white] (n3) at (14,9)  {\small{3}};
 \draw[->, thick] (n1) to  [in=120,out=70,looseness=5] (n1);
\draw[<<-, thick] (n1) edge  [bend left=10] (n2); 
\draw[->, thick] (n1) edge  [bend left=-10] (n2); 
 \draw[<<-, thick] (n2) edge  [bend left=10] (n3); 
\draw[<<-, thick] (n3) edge [bend right=-10] (n1); 
\draw[->, thick] (n2) edge  [bend left=-10] (n3); 
\end{tikzpicture}
&
$F_1 \& A_1$ 
\begin{tikzpicture}
  [scale=.15,auto=left, node distance=1.5cm, every node/.style={circle,draw}]
 \node[fill=white] (n1) at (4,0) {\small{1}};
 \node[fill=white] (n2) at (24,0) {\small{2}}; \node[fill=white] (n3) at (14,9)  {\small{3}};
\draw[<-, thick] (n1) edge  [bend left=-10] (n2); 
\draw[->>, thick] (n3) edge [bend right=-10] (n1); 
\draw[->, thick] (n1) edge  [bend left=-10] (n2); 
\draw[->>, thick] (n1) edge  [bend left=10] (n2); 
\draw[<-, thick] (n3) edge [bend right=10] (n1); 
\draw[->>, thick] (n2) edge  [bend left=10] (n3); 
\end{tikzpicture} 
 \\ \cline{1-2}  
\end{tabular}}
 }
 \caption{Minimal three-cell networks with two asymmetric inputs and no 2D synchrony subspaces.} 
 \label{tab:C3L0RI2.tex}
 \end{table}

\begin{table}
\resizebox{1 \textwidth}{!}{ 
    {\tiny 
 \begin{tabular}{|c|c|c|c|}
 \hline
 $D_1 \& E_1$ 
 \begin{tikzpicture}
  [scale=.15,auto=left, node distance=1.5cm, every node/.style={circle,draw}]
 \node[fill=white] (n1) at (4,0) {\small{1}};
 \node[fill=white] (n2) at (24,0) {\small{2}}; \node[fill=white] (n3) at (14,9)  {\small{3}};
\draw[->, thick] (n1) to  [in=120,out=70,looseness=5]  (n1);
\draw[->>, thick] (n1) to  [in=205,out=155,looseness=5] (n1);
 \draw[->>, thick] (n2) to  [in=70,out=120,looseness=5] (n2);
\draw[->, thick] (n1) edge  [bend left=-10] (n2); 
\draw[<<-, thick] (n3) edge [bend right=-10] (n1); 
\draw[->, thick] (n2) edge  [bend left=-10] (n3); 
\end{tikzpicture}
&
$D_1 \& F_1$ 
\begin{tikzpicture}
  [scale=.15,auto=left, node distance=1.5cm, every node/.style={circle,draw}]
 \node[fill=white] (n1) at (4,0) {\small{1}};
 \node[fill=white] (n2) at (24,0) {\small{2}}; \node[fill=white] (n3) at (14,9)  {\small{3}};
 \draw[->, thick] (n1) to  [in=120,out=70,looseness=5] (n1);
\draw[<<-, thick] (n1) edge  [bend left=10] (n2); 
\draw[->, thick] (n1) edge  [bend left=-10] (n2); 
\draw[->>, thick] (n1) edge  [bend left=10] (n2); 
\draw[<<-, thick] (n3) edge [bend right=-10] (n1); 
\draw[->, thick] (n2) edge  [bend left=-10] (n3); 
\end{tikzpicture}
&
 $D_1 \& F_2$ 
 \begin{tikzpicture}
  [scale=.15,auto=left, node distance=1.5cm, every node/.style={circle,draw}]
 \node[fill=white] (n1) at (4,0) {\small{1}};
 \node[fill=white] (n2) at (24,0) {\small{2}}; \node[fill=white] (n3) at (14,9)  {\small{3}};
 \draw[->, thick] (n1) to  [in=120,out=70,looseness=5] (n1);
\draw[<<-, thick] (n1) edge  [bend left=10] (n2); 
\draw[->, thick] (n1) edge  [bend left=-10] (n2); 
\draw[->>, thick] (n1) edge  [bend left=10] (n2); 
 \draw[->, thick] (n2) edge  [bend left=-10] (n3); 
\draw[->>, thick] (n2) edge  [bend left=10] (n3); 
\end{tikzpicture}
&
 $D_1 \& B_3$ 
 \begin{tikzpicture}
  [scale=.15,auto=left, node distance=1.5cm, every node/.style={circle,draw}]
 \node[fill=white] (n1) at (4,0) {\small{1}};
 \node[fill=white] (n2) at (24,0) {\small{2}}; \node[fill=white] (n3) at (14,9)  {\small{3}};
 \draw[->, thick] (n1) to  [in=120,out=70,looseness=5] (n1);
 \draw[->>, thick] (n3) to  [in=70,out=20,looseness=5] (n3);
\draw[<<-, thick] (n1) edge  [bend left=10] (n2); 
\draw[->, thick] (n1) edge  [bend left=-10] (n2); 
\draw[->>, thick] (n1) edge  [bend left=10] (n2); 
\draw[->, thick] (n2) edge  [bend left=-10] (n3); 
\end{tikzpicture}
\\ \hline 
$C_1 \& B_1$ 
\begin{tikzpicture}
  [scale=.15,auto=left, node distance=1.5cm, every node/.style={circle,draw}]
 \node[fill=white] (n1) at (4,0) {\small{1}};
 \node[fill=white] (n2) at (24,0) {\small{2}}; \node[fill=white] (n3) at (14,9)  {\small{3}};
\draw[->, thick] (n1) to  [in=120,out=70,looseness=5]  (n1);
\draw[->>, thick] (n1) to  [in=205,out=155,looseness=5] (n1);
\draw[->, thick] (n1) edge  [bend left=-10] (n2); 
 \draw[<<-, thick] (n2) edge  [bend left=10] (n3); 
\draw[<-, thick] (n3) edge [bend right=10] (n1); 
\draw[->>, thick] (n2) edge  [bend left=10] (n3); 
\end{tikzpicture}
 &
 $D_1 \& F_4$ 
 \begin{tikzpicture}
  [scale=.15,auto=left, node distance=1.5cm, every node/.style={circle,draw}]
 \node[fill=white] (n1) at (4,0) {\small{1}};
 \node[fill=white] (n2) at (24,0) {\small{2}}; \node[fill=white] (n3) at (14,9)  {\small{3}};
 \draw[->, thick] (n1) to  [in=120,out=70,looseness=5] (n1);
\draw[->>, thick] (n3) edge [bend right=-10] (n1); 
\draw[->, thick] (n1) edge  [bend left=-10] (n2); 
 \draw[<<-, thick] (n2) edge  [bend left=10] (n3); 
 \draw[->, thick] (n2) edge  [bend left=-10] (n3); 
\draw[->>, thick] (n2) edge  [bend left=10] (n3); 
\end{tikzpicture}
& 
$C_1 \& B_3$ 
\begin{tikzpicture}
  [scale=.15,auto=left, node distance=1.5cm, every node/.style={circle,draw}]
 \node[fill=white] (n1) at (4,0) {\small{1}};
 \node[fill=white] (n2) at (24,0) {\small{2}}; \node[fill=white] (n3) at (14,9)  {\small{3}};
 \draw[->, thick] (n1) to  [in=120,out=70,looseness=5] (n1);
 \draw[->>, thick] (n3) to  [in=70,out=20,looseness=5] (n3);
\draw[<<-, thick] (n1) edge  [bend left=10] (n2); 
\draw[->, thick] (n1) edge  [bend left=-10] (n2); 
\draw[->>, thick] (n1) edge  [bend left=10] (n2); 
\draw[<-, thick] (n3) edge [bend right=10] (n1); 
\end{tikzpicture}
 &
 $E_6 \& F_3$ 
 \begin{tikzpicture}
  [scale=.15,auto=left, node distance=1.5cm, every node/.style={circle,draw}]
 \node[fill=white] (n1) at (4,0) {\small{1}};
 \node[fill=white] (n2) at (24,0) {\small{2}}; \node[fill=white] (n3) at (14,9)  {\small{3}};
 \draw[->, thick] (n1) to  [in=120,out=70,looseness=5] (n1);
 \draw[->, thick] (n3) to  [in=110,out=160,looseness=5]  (n3);
\draw[<<-, thick] (n1) edge  [bend left=10] (n2); 
\draw[->, thick] (n1) edge  [bend left=-10] (n2); 
 \draw[<<-, thick] (n2) edge  [bend left=10] (n3); 
\draw[->>, thick] (n2) edge  [bend left=10] (n3); 
\end{tikzpicture}
\\ \hline 
$D_1 \& E_6$ 
\begin{tikzpicture}
  [scale=.15,auto=left, node distance=1.5cm, every node/.style={circle,draw}]
 \node[fill=white] (n1) at (4,0) {\small{1}};
 \node[fill=white] (n2) at (24,0) {\small{2}}; \node[fill=white] (n3) at (14,9)  {\small{3}};
\draw[->, thick] (n1) to  [in=120,out=70,looseness=5]  (n1);
\draw[->>, thick] (n1) to  [in=205,out=155,looseness=5] (n1);
 \draw[->>, thick] (n3) to  [in=70,out=20,looseness=5] (n3);
\draw[->, thick] (n1) edge  [bend left=-10] (n2); 
\draw[->>, thick] (n1) edge  [bend left=10] (n2); 
\draw[->, thick] (n2) edge  [bend left=-10] (n3); 
\end{tikzpicture}
 &
 $E_6 \& F_6$ 
 \begin{tikzpicture}
  [scale=.15,auto=left, node distance=1.5cm, every node/.style={circle,draw}]
 \node[fill=white] (n1) at (4,0) {\small{1}};
 \node[fill=white] (n2) at (24,0) {\small{2}}; \node[fill=white] (n3) at (14,9)  {\small{3}};
 \draw[->, thick] (n1) to  [in=120,out=70,looseness=5] (n1);
 \draw[->, thick] (n3) to  [in=110,out=160,looseness=5]  (n3);
\draw[->>, thick] (n3) edge [bend right=-10] (n1); 
\draw[->, thick] (n1) edge  [bend left=-10] (n2); 
\draw[->>, thick] (n1) edge  [bend left=10] (n2); 
\draw[<<-, thick] (n3) edge [bend right=-10] (n1); 
\end{tikzpicture}
&
 $E_6 \& F_4$ 
 \begin{tikzpicture}
  [scale=.15,auto=left, node distance=1.5cm, every node/.style={circle,draw}]
 \node[fill=white] (n1) at (4,0) {\small{1}};
 \node[fill=white] (n2) at (24,0) {\small{2}}; \node[fill=white] (n3) at (14,9)  {\small{3}};
 \draw[->, thick] (n1) to  [in=120,out=70,looseness=5] (n1);
 \draw[->, thick] (n3) to  [in=110,out=160,looseness=5]  (n3);
\draw[->>, thick] (n3) edge [bend right=-10] (n1); 
\draw[->, thick] (n1) edge  [bend left=-10] (n2); 
 \draw[<<-, thick] (n2) edge  [bend left=10] (n3); 
\draw[->>, thick] (n2) edge  [bend left=10] (n3); 
\end{tikzpicture}
&
$B_1 \& F_1$ 
\begin{tikzpicture}
  [scale=.15,auto=left, node distance=1.5cm, every node/.style={circle,draw}]
 \node[fill=white] (n1) at (4,0) {\small{1}};
 \node[fill=white] (n2) at (24,0) {\small{2}}; \node[fill=white] (n3) at (14,9)  {\small{3}};
 \draw[->, thick] (n1) to  [in=120,out=70,looseness=5] (n1);
\draw[<<-, thick] (n1) edge  [bend left=10] (n2); 
\draw[->>, thick] (n1) edge  [bend left=10] (n2); 
 \draw[<-, thick] (n2) edge  [bend left=-10] (n3); 
\draw[<<-, thick] (n3) edge [bend right=-10] (n1); 
\draw[->, thick] (n2) edge  [bend left=-10] (n3); 
\end{tikzpicture}
\\ \hline 
$F_1 \& F_2$ 
\begin{tikzpicture}
  [scale=.15,auto=left, node distance=1.5cm, every node/.style={circle,draw}]
 \node[fill=white] (n1) at (4,0) {\small{1}};
 \node[fill=white] (n2) at (24,0) {\small{2}}; \node[fill=white] (n3) at (14,9)  {\small{3}};
\draw[<-, thick] (n1) edge  [bend left=-10] (n2); 
 \draw[<<-, thick] (n1) edge  [bend left=10] (n2); 
\draw[->, thick] (n1) edge  [bend left=-10] (n2); 
\draw[->>, thick] (n1) edge  [bend left=10] (n2); 
\draw[<-, thick] (n3) edge [bend right=10] (n1); 
\draw[->>, thick] (n2) edge  [bend left=10] (n3); 
\end{tikzpicture}
& 
$F_1 \& F_3$ 
\begin{tikzpicture}
  [scale=.15,auto=left, node distance=1.5cm, every node/.style={circle,draw}]
 \node[fill=white] (n1) at (4,0) {\small{1}};
 \node[fill=white] (n2) at (24,0) {\small{2}}; \node[fill=white] (n3) at (14,9)  {\small{3}};
\draw[<-, thick] (n1) edge  [bend left=-10] (n2); 
 \draw[<<-, thick] (n1) edge  [bend left=10] (n2); 
\draw[->, thick] (n1) edge  [bend left=-10] (n2); 
 \draw[<<-, thick] (n2) edge  [bend left=10] (n3); 
\draw[<-, thick] (n3) edge [bend right=10] (n1); 
\draw[->>, thick] (n2) edge  [bend left=10] (n3); 
\end{tikzpicture}
 &  
$F_1 \& F_6$ 
\begin{tikzpicture}
  [scale=.15,auto=left, node distance=1.5cm, every node/.style={circle,draw}]
 \node[fill=white] (n1) at (4,0) {\small{1}};
 \node[fill=white] (n2) at (24,0) {\small{2}}; \node[fill=white] (n3) at (14,9)  {\small{3}};
\draw[<-, thick] (n1) edge  [bend left=-10] (n2); 
\draw[->>, thick] (n3) edge [bend right=-10] (n1); 
\draw[->, thick] (n1) edge  [bend left=-10] (n2); 
\draw[->>, thick] (n1) edge  [bend left=10] (n2); 
\draw[<-, thick] (n3) edge [bend right=10] (n1); 
\draw[<<-, thick] (n3) edge [bend right=-10] (n1); 
\end{tikzpicture}
 &  
 $D_1 \& F_5$ 
 \begin{tikzpicture}
  [scale=.15,auto=left, node distance=1.5cm, every node/.style={circle,draw}]
 \node[fill=white] (n1) at (4,0) {\small{1}};
 \node[fill=white] (n2) at (24,0) {\small{2}}; \node[fill=white] (n3) at (14,9)  {\small{3}};
 \draw[->, thick] (n1) to  [in=120,out=70,looseness=5] (n1);
\draw[->>, thick] (n3) edge [bend right=-10] (n1); 
\draw[->, thick] (n1) edge  [bend left=-10] (n2); 
 \draw[<<-, thick] (n2) edge  [bend left=10] (n3); 
\draw[<<-, thick] (n3) edge [bend right=-10] (n1); 
\draw[->, thick] (n2) edge  [bend left=-10] (n3); 
\end{tikzpicture}
\\ \hline 
\end{tabular}}
 }
 \caption{Minimal three-cell networks with two asymmetric inputs and one 2D synchrony subspace.} 
 \label{tab:C3L1.tex}
 \end{table}

\begin{table}
\resizebox{1 \textwidth}{!}{ 
   {\tiny 
 \begin{tabular}{|c|c|c|c|}
 \hline
 $E_6 \& E_4$ 
\begin{tikzpicture}
  [scale=.15,auto=left, node distance=1.5cm, every node/.style={circle,draw}]
 \node[fill=white] (n1) at (4,0) {\small{1}};
 \node[fill=white] (n2) at (24,0) {\small{2}}; \node[fill=white] (n3) at (14,9)  {\small{3}};
\draw[->, thick] (n1) to  [in=120,out=70,looseness=5]  (n1);
\draw[->>, thick] (n1) to  [in=205,out=155,looseness=5] (n1);
 \draw[->, thick] (n3) to  [in=110,out=160,looseness=5]  (n3);
 \draw[->>, thick] (n3) to  [in=70,out=20,looseness=5] (n3);
\draw[->, thick] (n1) edge  [bend left=-10] (n2); 
 \draw[<<-, thick] (n2) edge  [bend left=10] (n3); 
\end{tikzpicture}
 &  
$C_1 \& D_1$ 
\begin{tikzpicture}
  [scale=.15,auto=left, node distance=1.5cm, every node/.style={circle,draw}]
 \node[fill=white] (n1) at (4,0) {\small{1}};
 \node[fill=white] (n2) at (24,0) {\small{2}}; \node[fill=white] (n3) at (14,9)  {\small{3}};
\draw[->, thick] (n1) to  [in=120,out=70,looseness=5]  (n1);
\draw[->>, thick] (n1) to  [in=205,out=155,looseness=5] (n1);
\draw[->, thick] (n1) edge  [bend left=-10] (n2); 
\draw[->>, thick] (n1) edge  [bend left=10] (n2); 
\draw[<-, thick] (n3) edge [bend right=10] (n1); 
\draw[->>, thick] (n2) edge  [bend left=10] (n3); 
\end{tikzpicture}
&
$C_1 \& D_4$ 
\begin{tikzpicture}
  [scale=.15,auto=left, node distance=1.5cm, every node/.style={circle,draw}]
 \node[fill=white] (n1) at (4,0) {\small{1}};
 \node[fill=white] (n2) at (24,0) {\small{2}}; \node[fill=white] (n3) at (14,9)  {\small{3}};
 \draw[->, thick] (n1) to  [in=120,out=70,looseness=5] (n1);
 \draw[->>, thick] (n2) to  [in=70,out=120,looseness=5] (n2);
\draw[<<-, thick] (n1) edge  [bend left=10] (n2); 
\draw[->, thick] (n1) edge  [bend left=-10] (n2); 
\draw[<-, thick] (n3) edge [bend right=10] (n1); 
\draw[<<-, thick] (n3) edge [bend right=-10] (n1); 
\end{tikzpicture}
 &
$C_1 \& D_6$ 
\begin{tikzpicture}
  [scale=.15,auto=left, node distance=1.5cm, every node/.style={circle,draw}]
 \node[fill=white] (n1) at (4,0) {\small{1}};
 \node[fill=white] (n2) at (24,0) {\small{2}}; \node[fill=white] (n3) at (14,9)  {\small{3}};
 \draw[->, thick] (n1) to  [in=120,out=70,looseness=5] (n1);
 \draw[->>, thick] (n3) to  [in=70,out=20,looseness=5] (n3);
\draw[<<-, thick] (n1) edge  [bend left=10] (n2); 
\draw[->, thick] (n1) edge  [bend left=-10] (n2); 
 \draw[<<-, thick] (n2) edge  [bend left=10] (n3); 
\draw[<-, thick] (n3) edge [bend right=10] (n1); 
\end{tikzpicture}
 \\ \hline
$D_1 \& D_4$ 
\begin{tikzpicture}
  [scale=.15,auto=left, node distance=1.5cm, every node/.style={circle,draw}]
 \node[fill=white] (n1) at (4,0) {\small{1}};
 \node[fill=white] (n2) at (24,0) {\small{2}}; \node[fill=white] (n3) at (14,9)  {\small{3}};
 \draw[->, thick] (n1) to  [in=120,out=70,looseness=5] (n1);
 \draw[->>, thick] (n2) to  [in=70,out=120,looseness=5] (n2);
\draw[<<-, thick] (n1) edge  [bend left=10] (n2); 
\draw[->, thick] (n1) edge  [bend left=-10] (n2); 
\draw[<<-, thick] (n3) edge [bend right=-10] (n1); 
\draw[->, thick] (n2) edge  [bend left=-10] (n3); 
\end{tikzpicture}
\\ \cline{1-1}
\end{tabular}}
  }
 \caption{Minimal three-cell networks with two asymmetric inputs and one 2D synchrony subspace.} 
 \label{tab:C3L12.tex}
 \end{table}

 \begin{table}
 \begin{center}
 \resizebox{1 \textwidth}{!}{ 
 {\tiny 
 \begin{tabular}{|c|c|c|c|}
 \hline
$C_1 \& E_6$ \begin{tikzpicture}
  [scale=.15,auto=left, node distance=1.5cm, every node/.style={circle,draw}]
 \node[fill=white] (n1) at (4,0) {\small{1}};
 \node[fill=white] (n2) at (24,0) {\small{2}}; \node[fill=white] (n3) at (14,9)  {\small{3}};
\draw[->, thick] (n1) to  [in=120,out=70,looseness=5]  (n1);
\draw[->>, thick] (n1) to  [in=205,out=155,looseness=5] (n1);
 \draw[->>, thick] (n3) to  [in=70,out=20,looseness=5] (n3);
\draw[->, thick] (n1) edge  [bend left=-10] (n2); 
\draw[->>, thick] (n1) edge  [bend left=10] (n2); 
\draw[<-, thick] (n3) edge [bend right=10] (n1); 
\end{tikzpicture}
 &
$C_1 \& E_3$ 
\begin{tikzpicture}
  [scale=.15,auto=left, node distance=1.5cm, every node/.style={circle,draw}]
 \node[fill=white] (n1) at (4,0) {\small{1}};
 \node[fill=white] (n2) at (24,0) {\small{2}}; \node[fill=white] (n3) at (14,9)  {\small{3}};
\draw[->, thick] (n1) to  [in=120,out=70,looseness=5]  (n1);
\draw[->>, thick] (n1) to  [in=205,out=155,looseness=5] (n1);
 \draw[->>, thick] (n2) to  [in=70,out=120,looseness=5] (n2);
\draw[->, thick] (n1) edge  [bend left=-10] (n2); 
\draw[<-, thick] (n3) edge [bend right=10] (n1); 
\draw[->>, thick] (n2) edge  [bend left=10] (n3); 
\end{tikzpicture}
&
$C_1 \& F_1$ 
\begin{tikzpicture}
  [scale=.15,auto=left, node distance=1.5cm, every node/.style={circle,draw}]
 \node[fill=white] (n1) at (4,0) {\small{1}};
 \node[fill=white] (n2) at (24,0) {\small{2}}; \node[fill=white] (n3) at (14,9)  {\small{3}};
 \draw[->, thick] (n1) to  [in=120,out=70,looseness=5] (n1);
\draw[<<-, thick] (n1) edge  [bend left=10] (n2); 
\draw[->, thick] (n1) edge  [bend left=-10] (n2); 
\draw[->>, thick] (n1) edge  [bend left=10] (n2); 
\draw[<-, thick] (n3) edge [bend right=10] (n1); 
\draw[<<-, thick] (n3) edge [bend right=-10] (n1); 
\end{tikzpicture}
&
$C_1 \& F_2$ 
\begin{tikzpicture}
  [scale=.15,auto=left, node distance=1.5cm, every node/.style={circle,draw}]
 \node[fill=white] (n1) at (4,0) {\small{1}};
 \node[fill=white] (n2) at (24,0) {\small{2}}; \node[fill=white] (n3) at (14,9)  {\small{3}};
 \draw[->, thick] (n1) to  [in=120,out=70,looseness=5] (n1);
\draw[<<-, thick] (n1) edge  [bend left=10] (n2); 
\draw[->, thick] (n1) edge  [bend left=-10] (n2); 
\draw[->>, thick] (n1) edge  [bend left=10] (n2); 
\draw[<-, thick] (n3) edge [bend right=10] (n1); 
\draw[->>, thick] (n2) edge  [bend left=10] (n3); 
\end{tikzpicture}
\\ \hline 
$C_1 \& F_3$ 
\begin{tikzpicture}
  [scale=.15,auto=left, node distance=1.5cm, every node/.style={circle,draw}]
 \node[fill=white] (n1) at (4,0) {\small{1}};
 \node[fill=white] (n2) at (24,0) {\small{2}}; \node[fill=white] (n3) at (14,9)  {\small{3}};
 \draw[->, thick] (n1) to  [in=120,out=70,looseness=5] (n1);
\draw[<<-, thick] (n1) edge  [bend left=10] (n2); 
\draw[->, thick] (n1) edge  [bend left=-10] (n2); 
 \draw[<<-, thick] (n2) edge  [bend left=10] (n3); 
\draw[<-, thick] (n3) edge [bend right=10] (n1); 
\draw[->>, thick] (n2) edge  [bend left=10] (n3); 
\end{tikzpicture}
&
$E_6 \& F_5$ 
\begin{tikzpicture}
  [scale=.15,auto=left, node distance=1.5cm, every node/.style={circle,draw}]
 \node[fill=white] (n1) at (4,0) {\small{1}};
 \node[fill=white] (n2) at (24,0) {\small{2}}; \node[fill=white] (n3) at (14,9)  {\small{3}};
 \draw[->, thick] (n1) to  [in=120,out=70,looseness=5] (n1);
 \draw[->, thick] (n3) to  [in=110,out=160,looseness=5]  (n3);
\draw[->>, thick] (n3) edge [bend right=-10] (n1); 
\draw[->, thick] (n1) edge  [bend left=-10] (n2); 
 \draw[<<-, thick] (n2) edge  [bend left=10] (n3); 
\draw[<<-, thick] (n3) edge [bend right=-10] (n1); 
\end{tikzpicture}
&
$F_1 \& F_4$ 
\begin{tikzpicture}
  [scale=.15,auto=left, node distance=1.5cm, every node/.style={circle,draw}]
 \node[fill=white] (n1) at (4,0) {\small{1}};
 \node[fill=white] (n2) at (24,0) {\small{2}}; \node[fill=white] (n3) at (14,9)  {\small{3}};
\draw[<-, thick] (n1) edge  [bend left=-10] (n2); 
\draw[->>, thick] (n3) edge [bend right=-10] (n1); 
\draw[->, thick] (n1) edge  [bend left=-10] (n2); 
 \draw[<<-, thick] (n2) edge  [bend left=10] (n3); 
\draw[<-, thick] (n3) edge [bend right=10] (n1); 
\draw[->>, thick] (n2) edge  [bend left=10] (n3); 
\end{tikzpicture}
\\ 
 \cline{1-3} 
\end{tabular}}
}
 \end{center}
 \caption{Minimal three-cell networks with two asymmetric inputs and two 2D synchrony subspaces.} 
 \label{tab:C3L2.tex}
 \end{table}

\begin{table}
 \begin{center}
 {\tiny 
 \begin{tabular}{|c|}
 \hline 
 $C_1 \& C_2$ 
\begin{tikzpicture}
  [scale=.15,auto=left, node distance=1.5cm, every node/.style={circle,draw}]
 \node[fill=white] (n1) at (4,0) {\small{1}};
 \node[fill=white] (n2) at (24,0) {\small{2}}; \node[fill=white] (n3) at (14,9)  {\small{3}};
 \draw[->, thick] (n1) to  [in=120,out=70,looseness=5] (n1);
 \draw[->>, thick] (n2) to  [in=70,out=120,looseness=5] (n2);
\draw[<<-, thick] (n1) edge  [bend left=10] (n2); 
\draw[->, thick] (n1) edge  [bend left=-10] (n2); 
\draw[<-, thick] (n3) edge [bend right=10] (n1); 
\draw[->>, thick] (n2) edge  [bend left=10] (n3); 
\end{tikzpicture}
\\ \hline 
\end{tabular}}
 \end{center}
 \caption{Minimal three-cell network with two asymmetric inputs and three 2D synchrony subspaces.} 
 \label{tab:C2L3.tex}
 \end{table}

For each three-cell connected network  described in Theorem~\ref{thm:classification}, in order to investigate  which synchrony subspaces support a synchrony-breaking bifurcation branch of steady-state solutions, the strategy is the following. First, we use Proposition~\ref{Prop:eigen_reg_matrix} to see when the network eigenvalues are simple, semisimple or defective, and conjugate complex numbers or not. Combining this study about network eigenvalues and the number of  two-dimensional network synchrony subspaces given in Theorem~\ref{thm:classification}, we list the annotated network synchrony lattices. Finally, we use the results obtained  in Section~\ref{sec:ssbresults} to determine which synchrony subspaces support a synchrony-breaking bifurcation branch of steady-state solutions, assuming  a codimension-one steady-state bifurcation occurs determined by a degeneracy condition of the network Jacobian matrix at the origin.  In this last part, we remark that 
for three-cell networks with the synchrony lattice structure in Figure~\ref{lat:C2L3s} and corresponding bifurcation diagram given in Figure~\ref{lat:C2L3s}, Assumptions~\ref{asps:LSred2} have to be verified. For three-cell networks with  the synchrony lattice structure given in Figure~\ref{lat:C2L1d} and corresponding bifurcation diagram in Figure~\ref{lat:bdC2L1d}, condition~(\ref{eq:2detdef}) has to be verified. 
Table~\ref{tab:latdia} aggregates the information obtained in the previous steps and the results of this section which are obtained by explicit calculation for each network.

\begin{thm}\label{thm:classification2}   
Consider the three-cell networks listed in Theorem~\ref{thm:classification}.  For each such network ${\mathcal{N}}$ take coupled cell systems $f^{\mathcal{N}}$ where $f$ is generic and consider the corresponding Jacobian matrix $J^{\mathcal{N}}_f$ at the origin. We have the following:\\
(i) For the network $E_6 \& E_4$ in Table~\ref{tab:C3L12.tex}, $J^{\mathcal{N}}_f$ has  the valency eigenvalue with algebraic and geometric multiplicities $2$ and the eigenvalue $f_0$ with multiplicity $1$. 
\\
(ii) For the network $C$ in Table~\ref{tab:val13cell} and the network $C_1 \& C_2$ in Table~\ref{tab:C2L3.tex}, $J^{\mathcal{N}}_f$ has  the valency eigenvalue with multiplicity $1$ and the eigenvalue $f_0$ with algebraic and geometric multiplicities $2$. \\
(iii) For the network $D$ in Table~\ref{tab:val13cell} and networks $C_1 \& D_1$, $C_1 \& D_4$, $C_1 \& D_6$, $D_1 \& D_4$ in Table~\ref{tab:C3L12.tex}, $J^{\mathcal{N}}_f$ has the valency eigenvalue with multiplicity $1$ and the eigenvalue $f_0$ with algebraic multiplicity $2$ and geometric multiplicity $1$. 
\\
(iv) For the  network $F$ in Table~\ref{tab:val13cell}, the networks in Tables~\ref{tab:C3L1.tex} and ~\ref{tab:C3L2.tex} and networks $E_6 \& B_1$, $B_1 \& F_2$, $B_1 \& B_3$ in Table~\ref{tab:C3L0RI1.tex}, $J^{\mathcal{N}}_f$ has three distinct real eigenvalues with multiplicity $1$. \\
(v) For the networks in Table~\ref{tab:C3L0RI2.tex} and the network in Figure~\ref{fig:rep_min}, $J^{\mathcal{N}}_f$ has three distinct real eigenvalues with multiplicity $1$ on a open set of the generic functions.\\
(vi) For the network $A$ in Table~\ref{tab:val13cell} and networks $C_1 \& A_2$, $A_2 \& A_1$ in Table~\ref{tab:C3L0RI1.tex}, $J^{\mathcal{N}}_f$ has three distinct eigenvalues with multiplicity $1$ where the two eigenvalues different from the valency have nonzero imaginary parts.
\end{thm}

\begin{proof} 
The proof of this result goes through the direct application of Proposition~\ref{Prop:eigen_reg_matrix} to the networks with three-cells listed in Theorem~\ref{thm:classification}, obtaining  the fourth column in Table~\ref{tab:latdia}.  
We present details of  that in some illustrative examples. 
Recall the notation of Proposition~\ref{Prop:eigen_reg_matrix}. \\
(i) Consider the network ${\mathcal{N}}_1 = E_6 \& E_4$ in Table~\ref{tab:C3L12.tex} and $f:\mathbb{R}^3\rightarrow \mathbb{R}$ generic. We have that 
$$\upsilon=f_0+f_1+f_2, \quad 
\alpha_0 = f_0(f_0+f_1+f_2),\quad \alpha_1 =  2f_0+f_1+f_2\, . $$
Now $\alpha_0 =\upsilon^2$ if and only if $f_1+f_2=0$.
Assuming the non-degeneracy condition $f_1+f_2\neq 0$, it follows from Proposition~\ref{Prop:eigen_reg_matrix} that the eigenvalue $f_0+f_1+f_2$ has algebraic multiplicity lower that $3$.
Note that $\upsilon (\alpha_1 -\upsilon) = \alpha_0 $ for every $f$.
Thus the eigenvalue $f_0+f_1+f_2$ has algebraic multiplicity $2$ and the eigenvalue $f_0$ has algebraic multiplicity $1$, for $f$ generic.
\\
(ii) Consider the network ${\mathcal{N}}_2 =C_1 \& C_2$ in Table~\ref{tab:C2L3.tex} and $f:\mathbb{R}^3\rightarrow \mathbb{R}$ generic  
and assume the non-degenerated conditions  $\alpha_1\neq 2 \upsilon$ and $\upsilon (\alpha_1 -\upsilon) \neq \alpha_0$.
Like in the previous example, we have that $\alpha_1^2 = 4\alpha_0$ for every $f$.
However, in this case, we also have that $c-e=b-f= d-a+e-f = 0$ for every $f$.
It follows from Proposition~\ref{Prop:eigen_reg_matrix} that $J_f^{{\mathcal{N}}_2}$ has the eigenvalue $f_0+f_1+f_2$ with algebraic multiplicity $1$ and the eigenvalue $f_0$ with algebraic and geometric multiplicity $2$.
\\
(iii) Consider the network ${\mathcal{N}}_3 = C_1 \& D_1$ in Table~\ref{tab:C3L12.tex} and $f:\mathbb{R}^3\rightarrow \mathbb{R}$ generic. We have that 
$$\upsilon=f_0+f_1+f_2, \quad 
\alpha_0 = f_0^2,\quad \alpha_1 =2f_0\, .$$
In this case,  $\alpha_1=2 \upsilon$ if and only if $f_1+f_2=0$. Also, $\alpha_0 = \upsilon(\alpha_1-\upsilon)$ if and only if $f_1+f_2=0$. 
This condition is degenerated and we can take the corresponding inequality to obtain a non-degeneracy condition satisfied by any generic function $f$.
By Proposition~\ref{Prop:eigen_reg_matrix}, we know that the eigenvalue $f_0+f_1+f_2$ has algebraic multiplicity $1$.
Note that $\alpha_1^2 = 4\alpha_0$, for any $f$. 
Using the notation of Proposition~\ref{Prop:eigen_reg_matrix}, we have that $c-e=0$ if and only if $f_2=0$; also,  $b-f=0$ if and only if $f_2=0$. 
For a generic function $f$, assume that $f_2 \not= 0$.
 We obtain so that $f_0$ is an eigenvalue of $J_f^{{\mathcal{N}}_3}$ with algebraic multiplicity $2$ and geometric multiplicity $1$ for $f$ generic.
\\
(iv) Consider the network ${\mathcal{N}}_4 =E_6 \& B_1$ in Table~\ref{tab:C3L0RI1.tex} and $f:\mathbb{R}^3\rightarrow \mathbb{R}$ satisfying the non-degeneracy conditions $\upsilon (\alpha_1-\upsilon) \neq \alpha_0 $ and $\alpha_1^2 \neq 4\alpha_0$. 
By Proposition~\ref{Prop:eigen_reg_matrix}, the Jacobian matrix has 3 distinct eigenvalues with algebraic multiplicity $1$ for $f$ generic. 
Note that $\alpha_1^2-4\alpha_0= f_1^2+4f_2^2>0$ for any $f$ satisfying the non-degeneracy condition above.
Thus the eigenvalues are real, for every $f$ generic.
\\
(v) Consider the network ${\mathcal{N}}_5 = D_1 \& D_2$ in Table~\ref{tab:C3L0RI2.tex} and $f:\mathbb{R}^3\rightarrow \mathbb{R}$ satisfying the non-degenerated conditions  $\upsilon (\alpha_1-\upsilon) \neq \alpha_0 $ and $\alpha_1^2 \neq 4\alpha_0$.
Again using Proposition~\ref{Prop:eigen_reg_matrix}, the Jacobian matrix has 3 distinct eigenvalues with algebraic multiplicity $1$ for $f$ generic.
Note that $\alpha_1^2-4\alpha_0= 4f_1f_2 - 8f_1^2$ can be positive or negative for distinct open regions of functions. 
The three eigenvalues are real for any generic function $f$ in the open set given by $f_1f_2 - 2f_1^2>0$.\\
(vi) Consider the network ${\mathcal{N}}_6 = C_1 \& A_2$ in Table~\ref{tab:C3L0RI1.tex} and $f:\mathbb{R}^3\rightarrow \mathbb{R}$ satisfying the non-degeneracy  conditions  $\upsilon (\alpha_1-\upsilon) \neq \alpha_0 $ and $\alpha_1^2 \neq 4\alpha_0$.
By Proposition~\ref{Prop:eigen_reg_matrix}, the Jacobian matrix has 3 distinct eigenvalues with algebraic multiplicity $1$. 
Note that $\alpha_1^2-4\alpha_0= -3f_2^2<0$ for every $f$ satisfying the above non-degeneracy condition.  
Thus the two eigenvalues different of $f_1+f_2+f_3$ are conjugate complex numbers with  imaginary part different from $0$.
\end{proof}

\begin{table}
\begin{tabular}{|c|c|c|c|c|c|}
\hline 
Network 			&       Table 				& \# 2D  &            Network 			&      Annotated 				& Bifurcation \\
							&  										& 			&         eigenvalues 			&      lattice 					& diagram \\
 \hline 
 $E_6 \& E_4$ &  \ref{tab:C3L12.tex}& 1 							&  $\upsilon,\upsilon,\mu$&Figure~\ref{lat:C2L1v}	& Figure~\ref{lat:bdC2L1v}\\
\hline 
 $C_1 \& C_2$ &  \ref{tab:C2L3.tex}	&	3								&		$\upsilon,\mu,\mu$		&Figure~\ref{lat:C2L3s}	& Figure~\ref{lat:bdC2L3s} \\
  $C$ &  \ref{tab:val13cell} 			&&& &\\
 \hline 
  $C_1 \& D_1$, $C_1 \& D_4$ &  \ref{tab:C3L12.tex} &   & & &\\
  $C_1 \& D_6$, $D_1 \& D_4$ &  \ref{tab:C3L12.tex}	& 1 &  $\upsilon,\mu^*$     &Figure~\ref{lat:C2L1d}	& Figure~\ref{lat:bdC2L1d}\\
   $D$ 												&  \ref{tab:val13cell}&   && & \\
  \hline 
 $C_1 \& A_2$, $A_2 \& A_1$ & \ref{tab:C3L0RI1.tex} &		0						&	$\upsilon,\mu_1,\mu_2$		& Figure~\ref{lat:C3L0}	& Figure~\ref{lat:bdC3L0a}\\
$A$ 					& \ref{tab:val13cell}&		&		$\mu_1=\overline{\mu_2}$&												&\\
\hline
$E_6 \& B_1$, $B_1 \& F_2$ & \ref{tab:C3L0RI1.tex}&&&&\\
$B_1 \& B_3$ & \ref{tab:C3L0RI1.tex}	&	0							&	$\upsilon,\mu_1,\mu_2$	& Figure~\ref{lat:C3L0}	& Figure~\ref{lat:bdC3L0b}\\
 Figure~\ref{fig:rep_min} & 				&								&$\mu_1\neq\overline{\mu_2}$& 											&\\
\hline
 All & 	 \ref{tab:C3L0RI2.tex} 			& 0							&$\upsilon,\mu_1,\mu_2$			& Figure~\ref{lat:C3L0} & Figure~\ref{lat:bdC3L0a} or~\ref{lat:bdC3L0b}\\
\hline 
   All  &  \ref{tab:C3L1.tex} 			& 1							&$\upsilon,\mu_1,\mu_2$			& Figure \ref{lat:C3L1} & Figure~\ref{lat:bdC3L1} \\
				& 													&								&$\mu_1\neq\overline{\mu_2}$& 											&\\
   \hline 
    All &  \ref{tab:C3L2.tex}				&	2							&$\upsilon,\mu_1,\mu_2$			&Figure~\ref{lat:C3L2} & Figure~\ref{lat:bdC3L2} \\
     $F$ &  \ref{tab:val13cell}			&								&$\mu_1\neq\overline{\mu_2}$&											 & \\
    \hline 
\end{tabular}
\caption{This table aggregates the results obtained for the networks listed in Theorem~\ref{thm:classification}.
The third column, \# 2D, corresponds to the number of two-dimensional synchrony subspaces for the corresponding network. 
The fourth column displays the network eigenvalues as obtained in Theorem~\ref{thm:classification2}. 
The valency eigenvalue is denoted by $\upsilon$ and other network eigenvalues are denoted by $\mu$ or $\mu_1,\mu_2$.
Defective network eigenvalues with algebraic multiplicity two and geometry multiplicity one are marked with a star. 
There is also information on whether the eigenvalues $\mu_1,\mu_2$ are conjugate complex numbers or not. 
The fifth column corresponds to the annotated synchrony lattices as given in Theorem~\ref{thm:classificationlat} and the last column displays the bifurcation diagram.}
\label{tab:latdia}
\end{table}

\begin{rem} \label{rem:um_dois}
Among the three-cell networks in Table~\ref{tab:val13cell}, Tables~\ref{tab:C3L0RI1.tex}-\ref{tab:C2L3.tex} and Figure~\ref{fig:rep_min}, presented in Theorem~\ref{thm:classification}, taking a coupled cell system $f^{\mathcal{N}}$ where $f$ is generic and the corresponding Jacobian $J^{\mathcal{N}}_f$ at the origin, we 
have that: \\
(i) $J^{\mathcal{N}}_f$ is semisimple, except for the networks in Theorem~\ref{thm:classification2} (iii);\\
(ii) $J^{\mathcal{N}}_f$ has always a pair of conjugate complex eigenvalues with nonzero imaginary part for the networks in Theorem~\ref{thm:classification2} (vi). 
It has a pair of conjugate complex eigenvalues with nonzero imaginary part in a region of the functions $f$ for the networks in Theorem~\ref{thm:classification2} (v).
In this work we focus on steady-state bifurcations and we do not address the cases where the eigenvalues are conjugate complex numbers. 
We point out that in those cases,  Hopf bifurcation can occur. 
\hfill $\Diamond$
\end{rem}

Combining Theorems~\ref{thm:classification}-\ref{thm:classification2} and Theorem~\ref{thm:lala}, we have the following classification on the networks  annotated synchrony lattices:

\begin{thm}\label{thm:classificationlat}
The annotated synchrony lattice structures for the three-cell networks listed in Theorem~\ref{thm:classification} are presented in the fifth column of Table~\ref{tab:latdia}. 
\end{thm}

\begin{proof} 
For completeness, we illustrate how Theorems~\ref{thm:classification}-\ref{thm:classification2} determine the networks  annotated synchrony lattices for some of the networks listed in Theorem~\ref{thm:classification}.
\\  
(i) The Jacobian matrix for the network ${\mathcal{N}}_1 = E_6 \& E_4$ in Table~\ref{tab:C3L12.tex} has the valency eigenvalue $f_0+f_1+f_2$ with algebraic multiplicity $2$ and the eigenvalue $f_0$ with algebraic multiplicity $1$, for $f$ generic.
Since the network ${\mathcal{N}}_1$ has one two-dimensional synchrony subspace, we know that the annotated lattice of ${\mathcal{N}}_1$ is given in Figure~\ref{lat:C2L1v}.
\\
(ii) The Jacobian matrix for the network ${\mathcal{N}}_2 =C_1 \& C_2$ in Table~\ref{tab:C2L3.tex}, has the valency eigenvalue $f_0+f_1+f_2$ with algebraic multiplicity $1$ and the eigenvalue $f_0$ with algebraic and geometric multiplicity $2$.
We also know that the network ${\mathcal{N}}_2$ has three two-dimensional synchrony subspaces and its lattice is given in Figure~\ref{lat:C2L3s}.
\\
(iii) The Jacobian matrix for the network ${\mathcal{N}}_3 = C_1 \& D_1$ in Table~\ref{tab:C3L12.tex} has the valency eigenvalue $f_0+f_1+f_2$ with algebraic multiplicity $1$ and the eigenvalue $f_0$ with algebraic multiplicity $2$ and geometric multiplicity $1$, for $f$ generic.
Moreover, the network ${\mathcal{N}}_3$ has one two-dimensional synchrony subspace and its lattice is given in Figure~\ref{lat:C2L1d}.
\\
(iv) The Jacobian matrix for the network ${\mathcal{N}}_4 =E_6 \& B_1$ in Table~\ref{tab:C3L0RI1.tex} has 3 distinct eigenvalues with algebraic multiplicity $1$, for $f$ generic. 
The network ${\mathcal{N}}_4$ has no two-dimensional synchrony subspace and its annotated lattice is given in Figure~\ref{lat:C3L0}.
\\
(v) The Jacobian matrix for the network ${\mathcal{N}}_5 =  D_1 \& E_1$ in Table~\ref{tab:C3L1.tex}, has 3 distinct eigenvalues with algebraic multiplicity $1$, for $f$ generic.
And the network ${\mathcal{N}}_5$ has one two-dimensional synchrony subspace. Thus its annotated lattice is given in Figure~\ref{lat:C3L1}.
\\
(vi) The Jacobian matrix for the network ${\mathcal{N}}_6 = C_1 \& E_6$ in Table~\ref{tab:C3L2.tex} has 3 distinct eigenvalues with algebraic multiplicity $1$, for $f$ generic. 
Moreover, the network ${\mathcal{N}}_6$ has two two-dimensional synchrony subspace. Thus its annotated lattice is given in Figure~\ref{lat:C3L2}.
\end{proof}

\begin{rem} From the possible lattice structures presented in Theorem~\ref{thm:lala} for connected three-cell networks with asymmetric inputs, we have: \\
(i) Only the lattice structure in Figure~\ref{lat:C2L0d} does not appear when we restrict to networks with one or two asymmetric inputs. \\
(ii) The lattice structures of Figures~\ref{lat:C2L1v} and  \ref{lat:C3L1}  occur for connected three-cell networks with two asymmetric inputs but not for connected three-cell networks with one asymmetric input. 
\hfill $\Diamond$
\end{rem}

Finally, we classify which synchrony subspaces support a steady-state bifurcation branch for generic bifurcation problems on coupled cell systems of the three-cell networks given in Theorem~\ref{thm:classification}. This classification appears at the sixth column of Table~\ref{tab:latdia}.

\begin{thm}\label{thm:finalbif}
For the networks given in Theorem~\ref{thm:classification}, we have that every synchrony subspace supports a bifurcation branch of steady-state solutions,
for generic bifurcation problems  on coupled cell systems respecting the appropriate bifurcation condition, except for the network phase space $\R^3$ in the case of all the networks  in Tables~~\ref{tab:C3L0RI2.tex}, \ref{tab:C3L2.tex}-\ref{tab:C2L3.tex}, the networks $C_1 \& A_2$, $A_2 \& A_1$ in Table~\ref{tab:C3L0RI1.tex}, networks $A$, $C$, $F$ in Table~\ref{tab:val13cell} and the network in Figure~\ref{fig:rep_min}.  However, for each of the networks in Table~\ref{tab:C3L0RI2.tex} and the network in Figure~\ref{fig:rep_min},  there exists an open set of generic bifurcation problems, on coupled cell systems respecting the appropriate bifurcation condition, where the network phase space $\R^3$ supports a bifurcation branch. 
The sixth column of Table~\ref{tab:latdia} contains the bifurcation diagrams for the networks given in Theorem~\ref{thm:classification}.
\end{thm}

\begin{proof}
The proof of Theorem~\ref{thm:finalbif} follows as in the proof of Theorem~\ref{thm:bif_diag}, 
where the bifurcation diagram is obtained for each possible annotated synchrony lattice structure 
taking every connected three-cell network given in Theorem~\ref{thm:lala}.  
We include an illustrative example for each case studied in the proof of Theorem~\ref{thm:bif_diag} 
except case (iv), since the networks given in Theorem~\ref{thm:classification} do not have the 
synchrony lattice structure given in Figure~\ref{lat:C2L0d}. If a particular network has the lattice  given by 
Figure~\ref{lat:C2L3s} or Figure~\ref{lat:C2L1d}, then we need to check Assumptions~\ref{asps:LSred2} or condition (\ref{eq:2detdef}), respectively.
More specifically, for the examples illustrating cases (ii) and (iii) in the proof  of Theorem~\ref{thm:bif_diag}, 
we provide non-degeneracy conditions associated with Assumptions~\ref{asps:LSred2} and condition (\ref{eq:2detdef}), respectively. \\
(i) The network ${\mathcal{N}}_1 = E_6 \& E_4$ in Table~\ref{tab:C3L12.tex} has the annotated lattice given in Figure~\ref{lat:C2L1v}. 
The Jacobian matrix at the origin $J_f^{{\mathcal{N}}_1}$ has the eigenvalue $\upsilon=f_0+f_1+f_2$ with algebraic multiplicity $2$ and the simple eigenvalue $\mu=f_0$. 
So there are two steady-state bifurcation conditions $\upsilon=0$ and $\mu=0$. 
For the first condition, $\upsilon=0$, we consider $f\in\mathcal{V}_{\upsilon}({\mathcal{N}}_1)$.
The synchrony subspace $\Delta_0$ is $\upsilon$-simple and $\upsilon$-maximal, the two-dimensional synchrony subspace $\Delta_1$ is $\upsilon$-simple and $\upsilon$-submaximal and $\R^3$ is valency synchrony-breaking. 
It follows from Theorem~\ref{thm:LSred1} and Proposition~\ref{prop:valsynbre} that there are bifurcation branches of $f^{{\mathcal{N}}_1}$ with the synchrony $\Delta_0$ and $\R^3$.
Since $\Delta_1$ is $\upsilon$-simple and $\upsilon$-submaximal, we know that the bifurcation problem in $\Delta_1$ is reduced to a bifurcation problem in $\Delta_0$. 
Thus there is no bifurcation branches of $f^{{\mathcal{N}}_1}$ with synchrony exactly equal to $\Delta_1$.
Figure~\ref{lat:bdC2L1v} displays the two synchrony subspaces that support a bifurcation branch at $\upsilon=0$.
For the second condition, $\mu=0$, we consider $f\in\mathcal{V}_{\mu}({\mathcal{N}}_1)$.
As $\Delta_0$ does not have the eigenvalue $\mu$, the synchrony space $\Delta_1$ is $\mu$-simple and $\mu$-maximal and $\R^3$ is $\mu$-simple and $\mu$-submaximal.
Using Theorem~\ref{thm:LSred1}, we see that there is a bifurcation branch of $f^{{\mathcal{N}}_1}$ with the synchrony $\Delta_1$.
Moreover, we also know that there is no bifurcation branch of $f^{{\mathcal{N}}_1}$ with synchrony exactly equal to $\Delta_0$ or $\R^3$.
Figure~\ref{lat:bdC2L1v} displays that $\Delta_1$ supports a bifurcation branch at $\mu=0$ and 
the bifurcation diagram for ${\mathcal{N}}_1$ is given by Figure~\ref{lat:bdC2L1v}.\\
(ii) Consider the network ${\mathcal{N}}_2 =C_1 \& C_2$ in Table~\ref{tab:C2L3.tex}. The lattice of ${\mathcal{N}}_2$ is given in Figure~\ref{lat:C2L3s} and  $J_f^{{\mathcal{N}}_2}$ has the eigenvalue $\upsilon=f_0+f_1+f_2$ with algebraic multiplicity $1$ and the eigenvalue $\mu=f_0$ with algebraic and geometric multiplicity $2$.
Thus there are two steady-state bifurcation conditions $\upsilon=0$ and $\mu=0$, and there are three two-dimensional synchrony subspaces $\Delta_1$, $\Delta_2$ and $\Delta_3$. For the first condition, $\upsilon=0$, we consider $f\in\mathcal{V}_{\upsilon}({\mathcal{N}}_2)$.
The synchrony subspace $\Delta_0$ is $\upsilon$-simple and $\upsilon$-maximal, and $\Delta_1$, $\Delta_2$, $\Delta_3$ and $\R^3$ are $\upsilon$-simple and $\upsilon$-submaximal. Thus there is a bifurcation branch of $f^{{\mathcal{N}}_2}$ with the synchrony $\Delta_0$ and no bifurcation branches of $f^{{\mathcal{N}}_2}$ with synchrony exactly equal to $\Delta_1$, $\Delta_2$, $\Delta_3$ or $\R^3$. Hence the bifurcation diagram at $\upsilon=0$ is given in Figure~\ref{lat:bdC2L3s}.
For the second condition, $\mu=0$, we consider $f\in\mathcal{V}_{\mu}({\mathcal{N}}_2)$.
There is no bifurcation branch of $f^{{\mathcal{N}}_2}$ with synchrony exactly equal to $\Delta_0$, since $\mu$ is not an eigenvalue in $\Delta_0$.
The synchrony subspaces $\Delta_1$, $\Delta_2$ and $\Delta_3$ are $\mu$-simple and $\mu$-maximal and we apply Theorem~\ref{thm:LSred1} to each two-dimensional synchrony subspace. So, there are bifurcation branches of $f^{{\mathcal{N}}_2}$ with the synchronies $\Delta_1$, $\Delta_2$ and $\Delta_3$.
The network phase space $\R^3$ is $\mu$-semisimple with multiplicity $2$ and $\mu$-submaximal with order $3$.
Using Theorem~\ref{thm:LSred2}, we can conclude that $\R^3$ does not support a bifurcation branch. 
We need to check that the conditions in Theorem~\ref{thm:LSred2} are satisfied.  
The $\mu$-maximal synchrony subspaces $\Delta_1$, $\Delta_2$ and $\Delta_3$ are $2$-determined (see Remark~\ref{rem:detbif}). 
Next, we obtain explicit non-degeneracy conditions on the function $f$ such that it satisfies Assumptions~\ref{asps:LSred2}.
Using the notation used in the proof of Theorem~\ref{thm:LSred2}, set 
$$v_1=(0,0,1),\quad \quad v_2=(-f_2, f_1, f_1),\quad \quad \quad v^*_1=(0,-1,1),\quad \quad v^*_2=\frac{1}{f_1+f_2}(-1, 1, 0).$$
Following the proof of Theorem~\ref{thm:LSred2}, the reduced function $g$ of $f^{{\mathcal{N}}_2}$ given by the Lyapunov-Schmidt Reduction has the following second-order Taylor expansion: 
$$h(x_1, x_2, \lambda) =\left(\begin{array}{c} 
												h_1(x_1, x_2, \lambda)\\
												 h_2(x_1, x_2, \lambda)\end{array}\right)=\left(\begin{array}{c}
												f_{0\lambda}\lambda x_1 + \frac{f_{00}}{2}x_1^2+ (f_{00}f_1-f_{01}f_2+f_{02}f_1)x_1x_2\\
												f_{0\lambda}\lambda  x_2 +  	\frac{f_{00}(f_1-f_2)-2f_{01}f_2+2f_{02} f_1}{2}x_2^2 
												\end{array}\right)\, .$$
Consider the following non-degeneracy conditions: 
$$f_{0\lambda}\neq 0,\quad\quad f_{00}\neq 0, \quad\quad f_1+f_2\neq 0, \quad\quad f_1-f_2\neq 0,$$
$$f_{00}(f_1-f_2)-2f_{01}f_2+2f_{02} f_1 \neq 0\, .$$
Under these conditions, Assumptions~\ref{asps:LSred2} hold and there is no bifurcation branch of $f^{{\mathcal{N}}_2}$ with synchrony exactly equal to $\R^3$. 
Thus Figure~\ref{lat:bdC2L3s} is the bifurcation diagram of ${\mathcal{N}}_2$. \\
(iii) Consider the network ${\mathcal{N}}_3 = C_1 \& D_1$ in Table~\ref{tab:C3L12.tex} with lattice given in Figure~\ref{lat:C2L1d}.
 We know that $\upsilon=f_0+f_1+f_2$ is an eigenvalue of $J_f^{{\mathcal{N}}_3}$ with algebraic multiplicity $1$ and that $\mu=f_0$ is an eigenvalue of $J_f^{{\mathcal{N}}_3}$ with algebraic multiplicity $2$ and geometric multiplicity $1$ for $f$ generic. 
As the previous case, we know that there is a bifurcation branch of $f^{{\mathcal{N}}_2}$ with synchrony $\Delta_0$ when $f\in\mathcal{V}_{\upsilon}({\mathcal{N}}_3)$ and no bifurcation branches exactly with synchrony $\Delta_1$ or $\R^3$.
For the bifurcation condition $\mu=0$, we consider $f\in \mathcal{V}_{\mu}({\mathcal{N}}_3)$.
There is no bifurcation branch of $f^{{\mathcal{N}}_3}$ with synchrony exactly equal to $\Delta_0$ and there is a bifurcation branch of $f^{{\mathcal{N}}_3}$ with synchrony $\Delta_1$.
The network phase space $\R^3$ is $\mu$-defective with multiplicity $(1,2)$ and $\mu$-submaximal with order $1$ and we will use Theorem~\ref{thm:LSred3} to prove that there is a bifurcation branch of $f^{{\mathcal{N}}_3}$ with synchrony $\R^3$.
In order to apply Theorem~\ref{thm:LSred3}, we need to check if condition (\ref{eq:2detdef}) holds.
Following the proof of Theorem~\ref{thm:LSred3}, we set 
$$v_1=(0,0,1),\quad \quad v_2=(0, 1/f_2, 0),\quad \quad v^*_1=(-1/2,-1/2,1),\quad \quad v^*_2=(-f_2,f_2,0).$$
The network ${\mathcal{N}}_3$ satisfies condition (\ref{eq:2detdef}) for $p=q=0$, if the following non-degeneracy condition holds
$$<v^*_2,[v_2-(PJ_f^{{\mathcal{N}}_3})^{(-1)} Pv_1]*[v_2-(PJ_f^{{\mathcal{N}}_3})^{(-1)} Pv_1]>=\frac{3f_1^2+4f_1f_2+f_2^2}{3f_2(f_1+f_2)^2}\neq 0,$$
where $P$ is the projection into $\mbox{Im} (J_f^{{\mathcal{N}}_3})^2= \Delta_0$. 
Therefore, by Theorem~\ref{thm:LSred3}, there is a bifurcation branch with synchrony $\R^3$ for every generic $f\in\mathcal{V}_{\upsilon}({\mathcal{N}})$.
The bifurcation diagram has so  two branches of steady-state solutions with synchrony $\Delta_1$ and $\R^3$ at $\mu=0$ and the bifurcation diagram of ${\mathcal{N}}_3$ is given in Figure~\ref{lat:bdC2L1d}.
\\
(iv) Consider the network ${\mathcal{N}}_4 = C_1 \& E_6$ in Table~\ref{tab:C3L2.tex} with lattice given in Figure~\ref{lat:C3L2}.
The Jacobian matrix $J_f^{{\mathcal{N}}_4}$ has the following three simple eigenvalues: $\upsilon=f_0+f_1+f_2$, $\mu_1=f_0+f_2$ and $\mu_2=f_0$.
The full-synchrony subspace $\Delta_0$ is $\upsilon$-simple and $\upsilon$-maximal.
Note also that one of the two-dimensional synchrony spaces $\Delta_1$ is $\mu_1$-simple and $\mu_1$-maximal and the other $\Delta_2$ is $\mu_2$-simple and $\mu_2$-maximal.
As before, there is bifurcation branch with synchrony $\Delta_0$, $\Delta_1$ or $\Delta_2$ for bifurcation problems given by the condition $\upsilon=0$, $\mu_1=0$ or $\mu_2=0$, respectively. 
In this case, $\R^3$ is $\mu$-simple and $\mu$-submaximal, for any network eigenvalue, $\upsilon$, $\mu_1$ or $\mu_2$.
Independently of the bifurcation condition, $\R^3$ does not support a bifurcation branch.
Then the bifurcation diagram has one branch emerging at each bifurcation condition $\upsilon=0$, $\mu_1=0$ and $\mu_2=0$ with synchrony $\Delta_0$, $\Delta_1$ or $\Delta_2$ leading to the bifurcation diagram given in Figure~\ref{lat:bdC3L2}.
\\
(v) Consider the network ${\mathcal{N}}_5 = C_1 \& B_1$ in Table~\ref{tab:C3L1.tex} with the lattice given in Figure~\ref{lat:C3L1}.
The Jacobian matrix $J_f^{{\mathcal{N}}_5}$ has the following three simple eigenvalues:  $\upsilon=f_0+f_1+f_2$, $\mu_1=f_0+f_2$ and $\mu_2=f_0-f_2$. 
Thus we need to consider three steady-state bifurcation conditions: $\upsilon=0$, $\mu_1=0$ and $\mu_2=0$.
Note that the full-synchrony subspace $\Delta_0$ is $\upsilon$-simple and $\upsilon$-maximal, the two-dimensional synchrony subspace $\Delta_1$ is $\mu_1$-simple and $\mu_1$-maximal. 
The space $\R^3$ is $\mu_2$-simple and $\mu_2$-maximal.
The study for each bifurcation condition is similar and we can apply Theorem~\ref{thm:LSred1}.
Then, we have a bifurcation branch inside each synchrony space for bifurcation problems with the associated bifurcation conditions 
and the bifurcation diagram is given in Figure~\ref{lat:bdC3L1}.
\\
(vi)
Consider the network ${\mathcal{N}}_6 = C_1 \& A_2$ in Table~\ref{tab:C3L0RI1.tex} with the lattice given in Figure~\ref{lat:C3L0}.
The Jacobian matrix $J_f^{{\mathcal{N}}_6}$ has the following three simple eigenvalues:  $\upsilon=f_0+f_1+f_2$, $\mu_1=f_0-\frac{f_2}{2}+i\frac{f_2\sqrt{3}}{2}$ and $\mu_2=f_0-\frac{f_2}{2}-i\frac{f_2\sqrt{3}}{2}$. 
Since the eigenvalues are $\mu_1$ and $\mu_2$ are conjugate complex numbers, they do not lead to a steady-state bifurcation condition. 
We need to consider only the  bifurcation condition $\upsilon=0$.
In this case, we take $f\in \mathcal{V}_{\upsilon}({\mathcal{N}}_4)$.
The full-synchrony subspace $\Delta_0$ is $\upsilon$-simple and $\upsilon$-maximal and $\R^3$ is $\upsilon$-simple and $\upsilon$-submaximal.
Thus there is exactly one bifurcation branch with synchrony $\Delta_0$ and the bifurcation diagram is given in Figure~\ref{lat:bdC3L0a}.
\\
(vii) 
Consider the network ${\mathcal{N}}_7 =E_6 \& B_1$ in Table~\ref{tab:C3L0RI1.tex} with lattice given in Figure~\ref{lat:C3L0}.
The Jacobian matrix $J_f^{{\mathcal{N}}_7}$ has the following three simple eigenvalues:  $\upsilon=f_0+f_1+f_2$, $\mu_1=f_0 + \frac{f_1}{2} + \frac{\sqrt{f_1^2 + 4f_2^2}}{2}$ and $\mu_2=f_0+\frac{f_1}{2}-\frac{\sqrt{f_1^2 + 4f_2^2}}{2}$.
Note that $f_1^2 + 4f_2^2\geq 0$, then there are three steady-state bifurcation conditions: $\upsilon=0$, $\mu_1=0$ and $\mu_2=0$.
As before, there exists a bifurcation branch with synchrony $\Delta_0$ for any generic $f\in\mathcal{V}_{\upsilon}({\mathcal{N}}_5)$. The bifurcation diagram for the bifurcation 
condition  $\upsilon=0$ has a branch with synchrony $\Delta_0$.
The total phase space $\R^3$ is $\mu_1$-simple and $\mu_1$-maximal.
Thus, there exists a bifurcation branch of steady-state solutions with synchrony $\R^3$ for any generic $f\in\mathcal{V}_{\mu_1}({\mathcal{N}}_5)$ and the bifurcation diagram for the bifurcation condition  $\mu_1=0$ has a branch with synchrony $\R^3$.
The total phase space $\R^3$ is also $\mu_2$-simple and $\mu_2$-maximal.
So the bifurcation diagram at $\mu_2=0$ is identical to the previous case and the bifurcation diagram at $\mu_2=0$ has a branch with synchrony $\R^3$.
Therefore, the bifurcation diagram for ${\mathcal{N}}_7$ is given in Figure~\ref{lat:bdC3L0b}. 
\\
(viii) 
 Consider the network ${\mathcal{N}}_8 = D_1 \& D_2$ in Table~\ref{tab:C3L0RI2.tex} with lattice given in Figure~\ref{lat:C3L0}.
The Jacobian matrix $J_f^{{\mathcal{N}}_8}$ has the following eigenvalues: $\upsilon=f_0+f_1+f_2$, $\mu_1=f_0+\sqrt{f_1f_2}$ and $\mu_2=f_0-\sqrt{f_1f_2}$. 
Note that the eigenvalues $\mu_1$ and $\mu_2$ are real or conjugate complex numbers if $f_1f_2\geq 0$ or $f_1f_2<0$.
Moreover, the space of functions can be divided into two disjoint regions of functions depending if there is one or there are three steady-state bifurcation conditions. 
In the region given by $f_1f_2<0$, the analysis is similar to the case (vi) studied above.
Therefore, in this region the bifurcation diagram for ${\mathcal{N}}_8$ is given in Figure~\ref{lat:bdC3L0a}.
The second region is given by $f_1f_2>0$ and it is identical to the previous case (vii).
Here, the bifurcation diagram is given in Figure~\ref{lat:bdC3L0b}.
Thus the network ${\mathcal{N}}_8$ has two different bifurcation diagrams: Figure~\ref{lat:bdC3L0a} or Figure~\ref{lat:bdC3L0b}.
\end{proof}

\begin{rem}\label{rem:detbif}
(i) The networks $C_1\&D_1$, $C_1 \& D_4$ and $C_1 \& D_6$ and $D_1 \& D_4$ in Table~\ref{tab:C3L12.tex} satisfy the conditions of Theorem~\ref{thm:LSred3} for one of their eigenvalues. The steady-state bifurcation branches of the first three of those networks have been studied in \cite{NRS16} and they correspond to the networks $A$, $B$ and $C$, respectively, in that work.\\
(ii) It follows from \cite[Theorem 6.7]{SG11} and the considerations about determinacy given in the proof of Theorem~\ref{thm:LSred1} that the bifurcation problems considered in Theorem~\ref{thm:finalbif} are at most $3$ determined when the bifurcation condition is $\mu$-simple. 
This means that $g_{x^{2}}\neq 0$ or $g_{x^{3}}\neq 0$, where $g$ is the reduced function obtained by the Lyapunov-Schmidt Reduction. 
By explicit computation for the networks considered here, we have that $g_{x^{2}}=0$ if and only if the network and the bifurcation condition is one of the following: 
$E_6\& F_5$ when $f_0+f_1-f_2=0$, $C_1\& B_1$ when $f_0-f_2=0$, $E_6\& F_6$ when $f_0+f_1-f_2=0$, $E_6\& F_4$ when $f_0+f_1-f_2=0$, $B_1\& F_1$ when $f_0+f_1-f_2=0$, $F_1\& F_2$ when $f_0-f_1-f_2=0$, $F_1\& F_6$ when $f_0-f_1-f_2=0$, and $F$ when $f_0-f_1=0$.  

The condition $g_{x^{2}}=0$ is usually associated with a $\mathbb{Z}_2$-symmetry of the bifurcation problem. However, the authors of \cite{SG11} noted that this is not always the case for coupled cell systems of regular networks. That occurs, in particular, for  the networks $E_6\& F_6$ and $F_1\& F_6$, that we study here, where $g_{x^{2}} =0$, without the networks or any of their quotient networks having $\mathbb{Z}_2$-symmetry.
\hfill $\Diamond$
\end{rem}

\section{Conclusions}

This work contributes to the classification of the codimension-one steady-state synchrony-breaking bifurcations for networks with identical cells and asymmetric inputs.
In order to achieve this goal, we start by deriving general results regarding the codimension-one steady-state bifurcation problems from a full synchrony equilibrium covering connected networks 
with any number of identical cells and any number of asymmetric inputs.
The results take into account how the network synchrony spaces intersect the eigenspaces of the Jacobian matrix at a full synchrony equilibrium.
In particular, they are organized by the type of the eigenvalues, i.e.,  simple, semisimple or defective, and maximal or submaximal. 

These bifurcation results are then applied to the class of networks of three-cells with any number of asymmetric inputs, after we have obtained a classification of their eigenvalues and lattices of synchrony subspaces.
A direct application is the steady-state bifurcation analysis for the three-cell quotient networks of a given network with $n >3$ cells.
We then restrict to three-cell networks with one, two or six asymmetric inputs and, for each such network, we are able to identify the synchrony subspaces that support a synchrony-breaking bifurcation branch of steady-state solutions. 
We believe that the classification obtained here also holds for the three-cell networks with three, four and five asymmetric inputs, as the eigenvalues of the Jacobian at the full equilibrium tend to be simple as we increase the number of inputs.

\vspace{5mm}

\noindent {\bf Acknowledgments} \\
MA and AD were partially supported by CMUP, member of LASI, which is financed by national funds through FCT -- Funda\c c\~ao para a Ci\^encia e a Tecnologia, I.P., under the projects with reference UIDB/00144/2020 and UIDP/00144/2020. PS was supported by Grant BEETHOVEN2 of the National Science Centre, Poland, no. 2016/23/G/ST1/04081.

\end{document}